\documentclass[11pt,a4paper]{article}
\newcommand{\vectt}[1]{\boldsymbol{\mathbf{#1}}}

\newcommand{\dx}{\mathrm{d}x}
\newcommand{\dy}{\mathrm{d}y}
\newcommand{\dr}{\mathrm{d}r}

\newcommand{\im}{\mathrm{i}}
\newcommand{\E}{\mathrm{e}}
\newcommand{\fdx}{\frac{\mathrm{d}}{\dx}}

\def\XXint#1#2#3{{\setbox0=\hbox{$#1{#2#3}{\int}$}
\vcenter{\hbox{$#2#3$}}\kern-.5\wd0}}

\usepackage{color}
\definecolor{deepblue}{rgb}{0,0,0.5}
\definecolor{deepred}{rgb}{0.6,0,0}
\definecolor{deepgreen}{rgb}{0,0.5,0}

\usepackage{amsmath,amssymb}
\usepackage{bm}
\usepackage{amsthm}
\usepackage{anyfontsize}
\usepackage{url}
\usepackage{graphicx}
\usepackage{array}
\usepackage{varwidth}
\usepackage{geometry}
\usepackage{ esint }
\definecolor{ao(english)}{rgb}{0.0, 0.5, 0.0}
\usepackage{listings}
\usepackage{fancyhdr}
\usepackage{enumitem}
\usepackage[
    backend=bibtex,
   bibstyle=ieee,
    citestyle=numeric-comp,
    natbib=true,
    sortlocale=en_GB,
    sorting = nyt,
    url=true,
    doi=true,
    arxiv=true,
    eprint=true
]{biblatex}
\DeclareFieldFormat{eprint:arXiv}{%
  arXiv preprint: \href{https://arxiv.org/abs/#1}{#1 \texttt{[\printfield{eprintclass}]}}%
}
\usepackage[font=footnotesize]{caption}
\usepackage{mathtools}
\usepackage{subfig}
\usepackage{appendix}
\usepackage{standalone}
\usepackage{tikz}
\usetikzlibrary{trees}
\usepackage{hyperref}
\usepackage{xcolor}
\hypersetup{
    colorlinks,
    linkcolor={blue!50!black},
    citecolor={blue!50!black},
    urlcolor={blue!80!black}
}

\usepackage{todonotes}

\RequirePackage{algorithm}[0.1]

\RequirePackage[capitalize,nameinlink]{cleveref}[0.19]
\crefformat{equation}{\textup{#2(#1)#3}}
\crefrangeformat{equation}{\textup{#3(#1)#4--#5(#2)#6}}
\crefmultiformat{equation}{\textup{#2(#1)#3}}{ and \textup{#2(#1)#3}}
{, \textup{#2(#1)#3}}{, and \textup{#2(#1)#3}}
\crefrangemultiformat{equation}{\textup{#3(#1)#4--#5(#2)#6}}%
{ and \textup{#3(#1)#4--#5(#2)#6}}{, \textup{#3(#1)#4--#5(#2)#6}}{, and \textup{#3(#1)#4--#5(#2)#6}}

\Crefformat{equation}{#2Equation~\textup{(#1)}#3}
\Crefrangeformat{equation}{Equations~\textup{#3(#1)#4--#5(#2)#6}}
\Crefmultiformat{equation}{Equations~\textup{#2(#1)#3}}{ and \textup{#2(#1)#3}}
{, \textup{#2(#1)#3}}{, and \textup{#2(#1)#3}}
\Crefrangemultiformat{equation}{Equations~\textup{#3(#1)#4--#5(#2)#6}}%
{ and \textup{#3(#1)#4--#5(#2)#6}}{, \textup{#3(#1)#4--#5(#2)#6}}{, and \textup{#3(#1)#4--#5(#2)#6}}

\newmuskip\pFqmuskip

\newcommand*\pFq[6][8]{%
  \begingroup 
  \pFqmuskip=#1mu\relax
  \mathchardef\normalcomma=\mathcode`,
  \mathcode`\,=\string"8000
  \begingroup\lccode`\~=`\,
  \lowercase{\endgroup\let~}\pFqcomma
  {}_{#2}F_{#3}{\left(\genfrac..{0pt}{}{#4}{#5};#6\right)}%
  \endgroup
}
\newcommand{\pFqcomma}{{\normalcomma}\mskip\pFqmuskip}

\usepackage{algpseudocode}
\algnewcommand{\Initialize}[1]{
  \State \textbf{Initialize:}
  \Statex \hspace*{\algorithmicindent}\parbox[t]{.8\linewidth}{\raggedright #1}
}
\algnewcommand{\Indent}[2]{
  \State {#1}
  \vspace{-2mm}
  \Statex \hspace*{\algorithmicindent}\parbox[t]{.9\linewidth}{\raggedright #2}
}

\newtheorem{proposition}{Proposition}[section]
\newtheorem{theorem}{Theorem}[section]
\newtheorem{lemma}{Lemma}[section]
\newtheorem{definition}{Definition}[section]
\newtheorem{corollary}{Corollary}[section]
\newtheorem{remark}{Remark}[section]
\numberwithin{equation}{section}

\title{A sparse hierarchical $hp$-finite element method \\ on disks and annuli}

\author{Ioannis P.~A.~Papadopoulos\thanks{\scriptsize Weierstrass Institute for Applied Analysis and Stochastics, Berlin, DE, \tt{papadopoulos@wias-berlin.de};}
\and Sheehan Olver\thanks{\scriptsize Department of Mathematics, Imperial College London, UK, {\tt s.olver@imperial.ac.uk}.}}

\begin{document}

\maketitle
\date
\thispagestyle{empty}
\pagestyle{fancy}
\lhead{I.~P.~A.~Papadopoulos, S.~Olver}

\begin{abstract}
We develop a sparse hierarchical $hp$-finite element method ($hp$-FEM) for the Helmholtz equation with variable coefficients posed on a two-dimensional disk or annulus. The mesh is an inner disk cell (omitted if on an annulus domain) and concentric annuli cells.  The discretization preserves the Fourier mode decoupling of rotationally invariant operators, such as the Laplacian, which manifests as block diagonal mass and stiffness matrices. Moreover, the matrices have a sparsity pattern independent of the order of the discretization and admit an optimal complexity factorization.  The sparse $hp$-FEM can handle radial discontinuities in the right-hand side and in rotationally invariant Helmholtz coefficients.  Rotationally anisotropic coefficients that are approximated by low-degree polynomials in Cartesian coordinates also result in sparse linear systems.  We consider examples such as a high-frequency Helmholtz equation with radial discontinuities  and rotationally anisotropic coefficients,  singular source terms,  the time-dependent Schr\"odinger equation, and an extension to a three-dimensional cylinder domain, with a quasi-optimal solve, via the Alternating Direction Implicit (ADI) algorithm.
\end{abstract}

\section{Introduction}
\label{sec:introduction}

In this work, we develop a sparse hierarchical $hp$-finite element method ($hp$-FEM) on disks and annuli. The cells in the mesh are stacked concentric annuli where, if the domain is a disk, the innermost cell is a disk.  The FEM basis consists of Zernike and Zernike annular polynomials, multivariate orthogonal polynomials in Cartesian coordinates, that are defined on disks and annuli, respectively.  The stiffness (weak Laplacian) and mass matrices are sparse and banded irrespective of the polynomial order truncation and the number of cells in the mesh considered. Moreover, for rotationally invariant operators, such as the (weak) Laplacian, the induced matrices are block diagonal where the submatrices correspond to the Fourier mode decoupling. Thus the solve reduces to parallelizable sparse one-dimensional solves for each Fourier mode. In particular, each submatrix along the block diagonal in the Helmholtz operator may be factorized in optimal  $\mathcal{O}(N_h N_p)$  complexity. Here $N_h$ denotes the number of cells in the mesh and $N_p$ is the highest degree\footnote{FEM and spectral method literature denote the degree by $p$ and $n$, respectively. We use both conventions in this work, utilizing the notation that is most natural in the context considered.} of the polynomial basis. Thus we achieve an $\mathcal{O}(N_h N_p^2)$ optimal complexity solve in two dimensions. After an initial ``arrowhead'' of size $N_h \times N_h$, the local stiffness and mass submatrices on the block diagonal contain three and five nonzero diagonals, respectively, for increasing degree $N_p$ with minimal coupling across elements. The global stiffness and mass matrices are also block diagonal and the submatrices have a so-called Banded-Block-Banded ($B^3$) Arrowhead matrix structure with block-bandwidths $(1,1)$ and $(2,2)$, respectively, and a sub-block-bandwidth of $1$ \cite[Def.~4.1]{Olver2023}. The mesh and spy plots of the global mass and stiffness submatrices are given in \cref{fig:spy-plots}. By considering the tensor-product space of the basis on the disk with a univariate sparse continuous $hp$-FEM basis for the interval \cite{Olver2023}, one obtains an FEM basis for the cylinder that is highly effective at handling discontinuities in the radial and $z$-directions as exemplified in \cref{fig:adi-plots2}. By leveraging the Alternating Direction Implicit (ADI) algorithm \cite{Fortunato2020}, we show one obtains a quasi-optimal $\mathcal{O}(N_h N_p^3 \log (N_h N_p))$ complexity solve  for the screened Poisson equation.

\begin{figure}[h!]
\centering
\includegraphics[width =0.25 \textwidth]{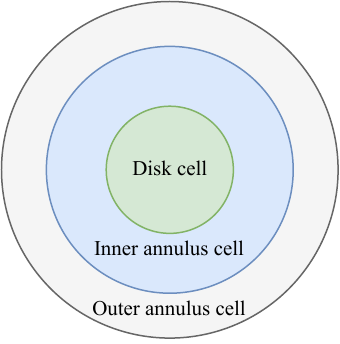} 
\includegraphics[width =0.35 \textwidth]{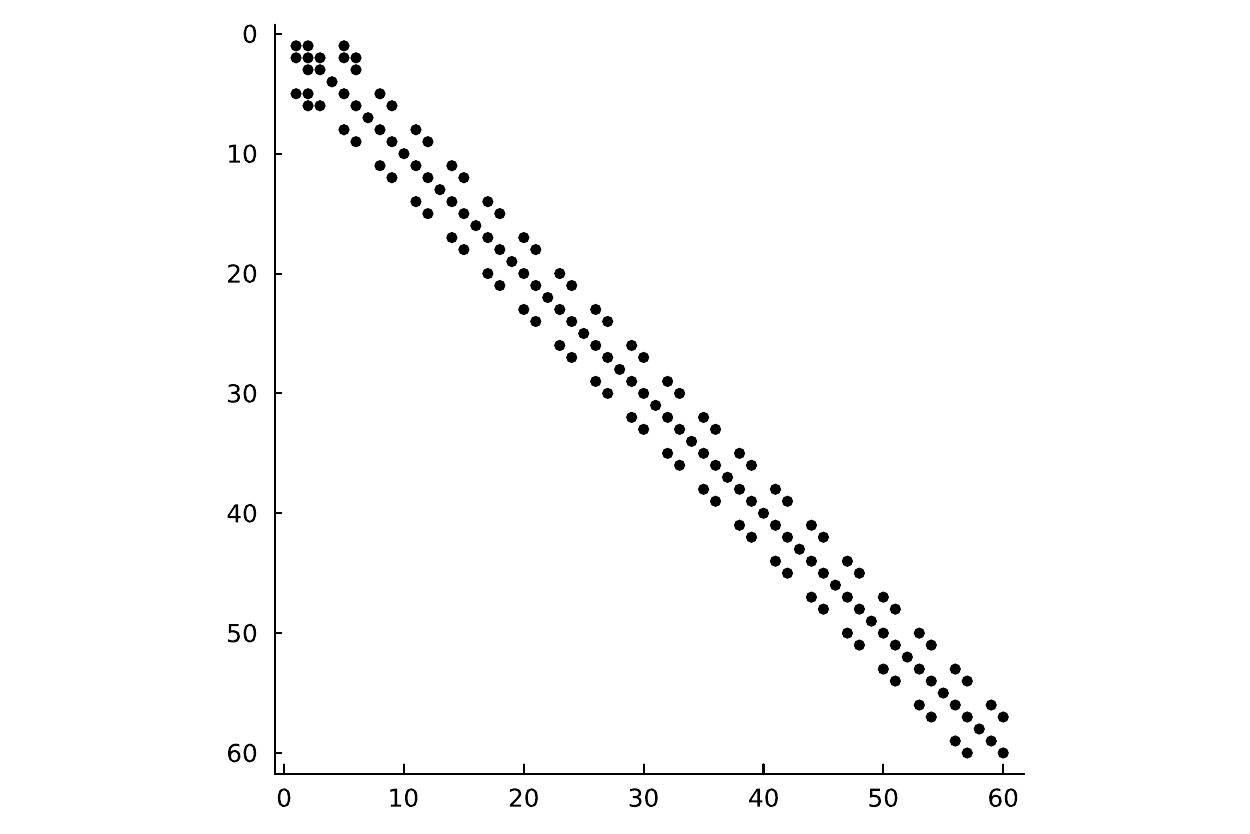} 
\includegraphics[width =0.35 \textwidth]{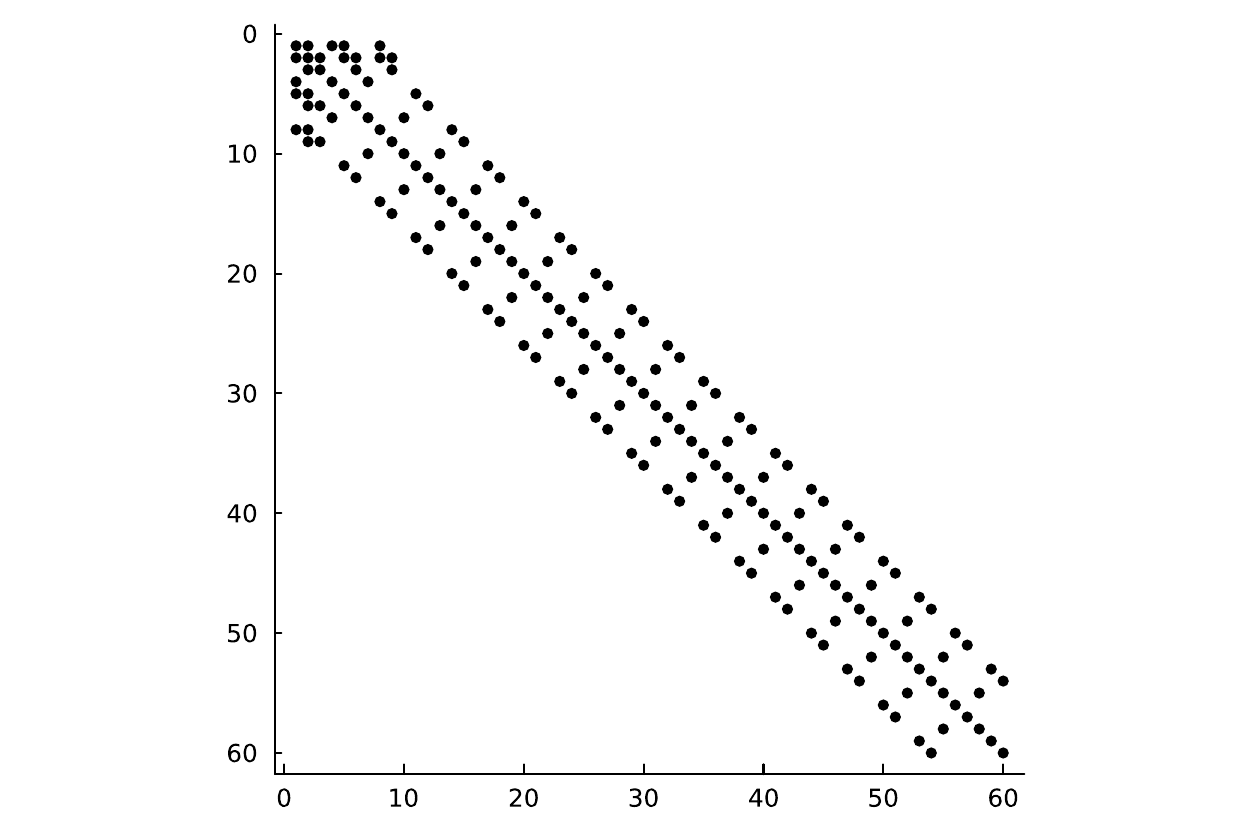} 
\caption{Meshing a disk domain into 3 cells, an inner disk, and two concentric annuli (left). Spy plots of the first block diagonal submatrix (corresponding to the first Fourier mode) in the stiffness (middle) and mass (right) matrices when truncating the hierarchical basis at degree $N_p = 20$ on each cell in the displayed mesh. The matrices are sparse and banded with a bandwidth independent of $N_p$.}
\label{fig:spy-plots}
\end{figure}

\begin{figure}[h!]
\centering
\includegraphics[width =0.28 \textwidth]{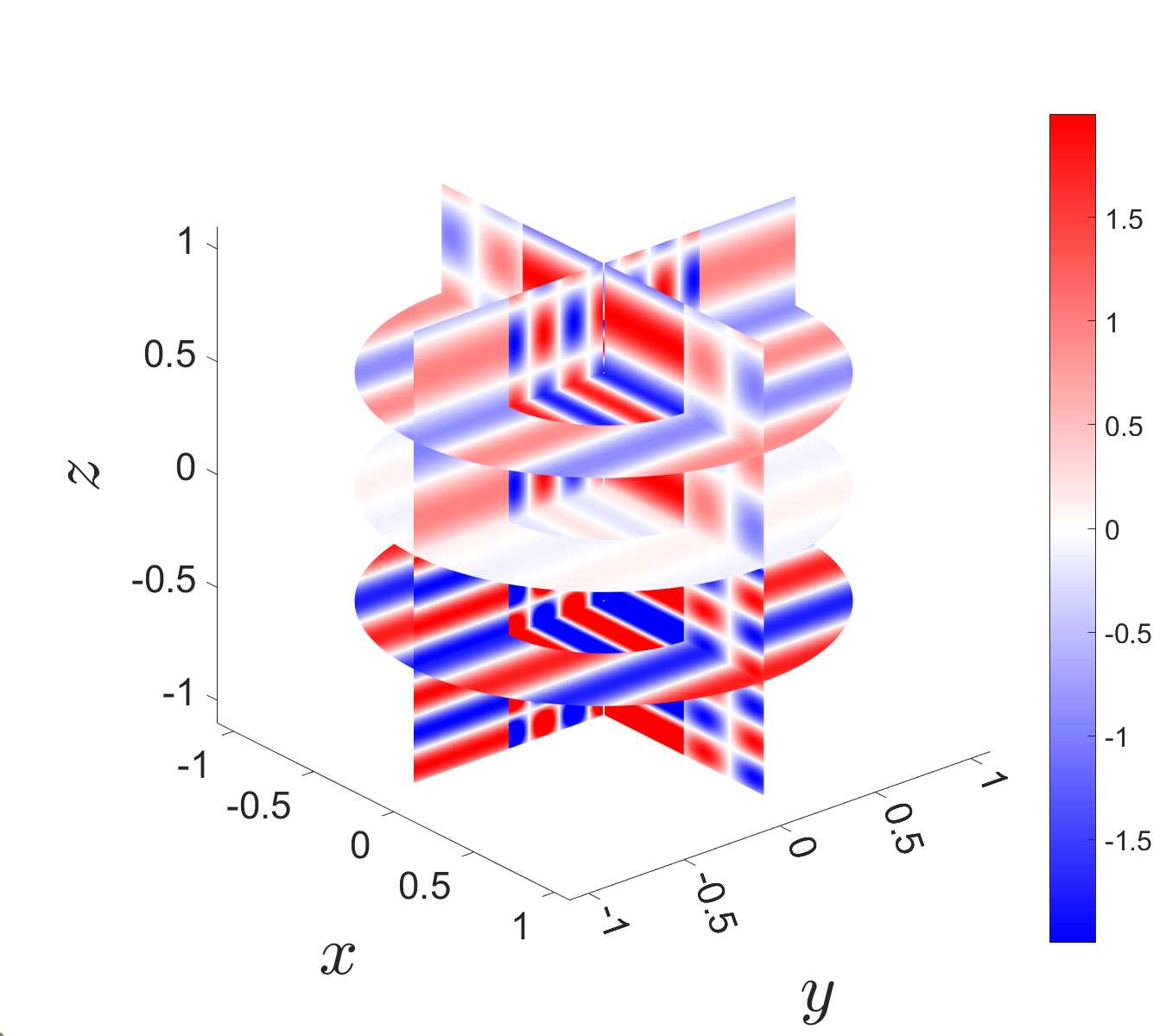}
\includegraphics[width =0.32 \textwidth]{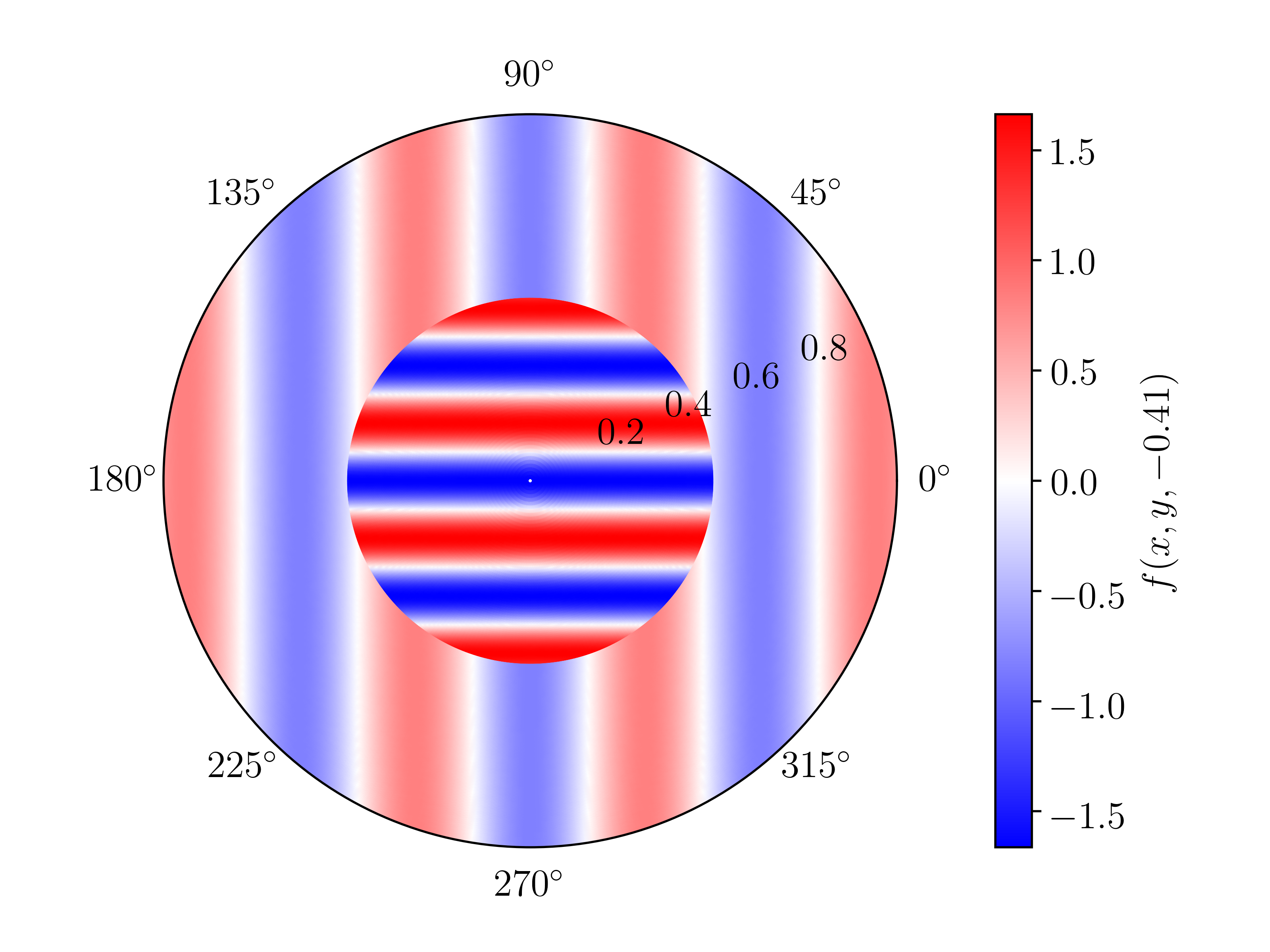}
\includegraphics[width =0.32 \textwidth]{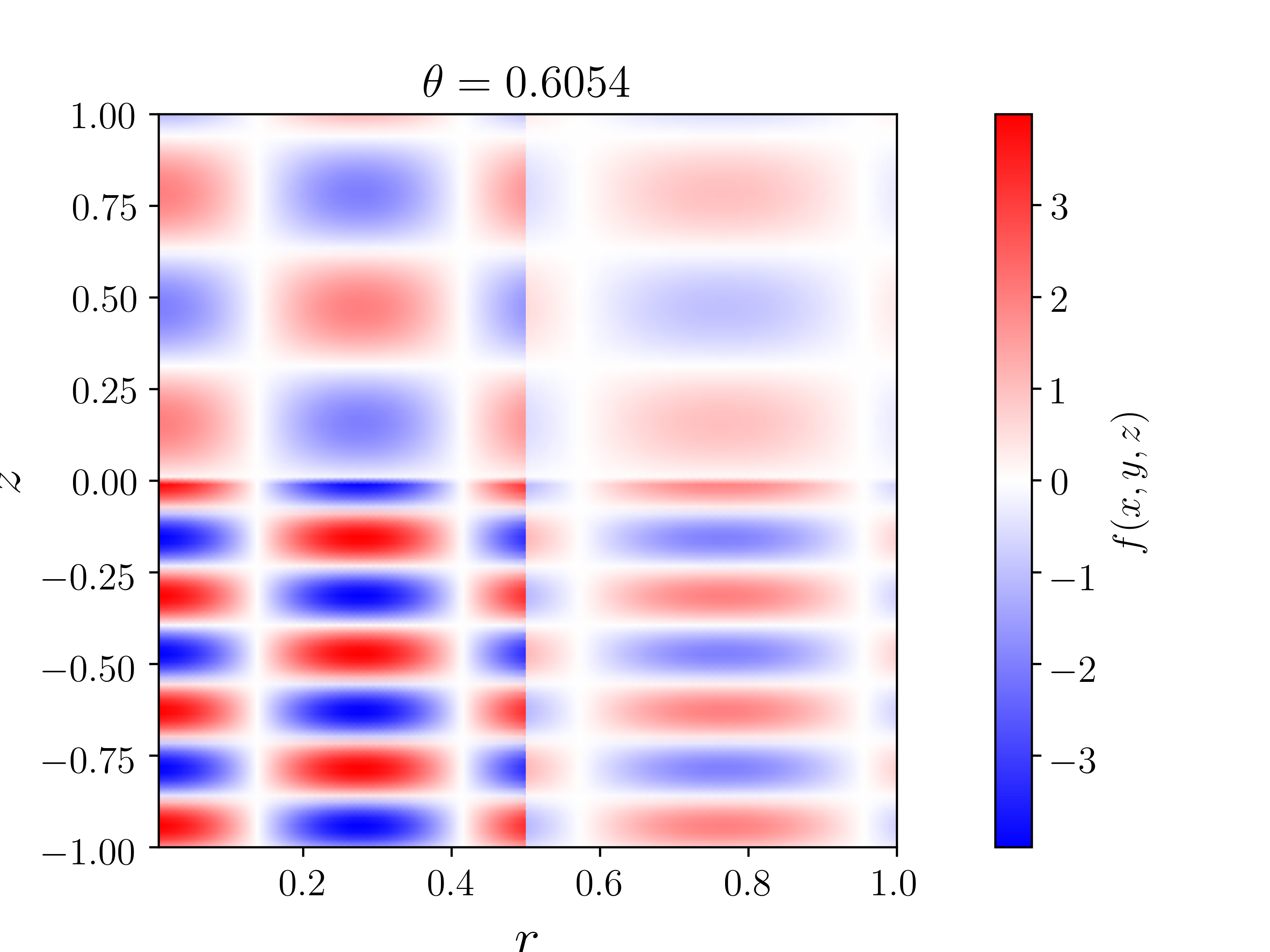}\\
\includegraphics[width =0.28 \textwidth]{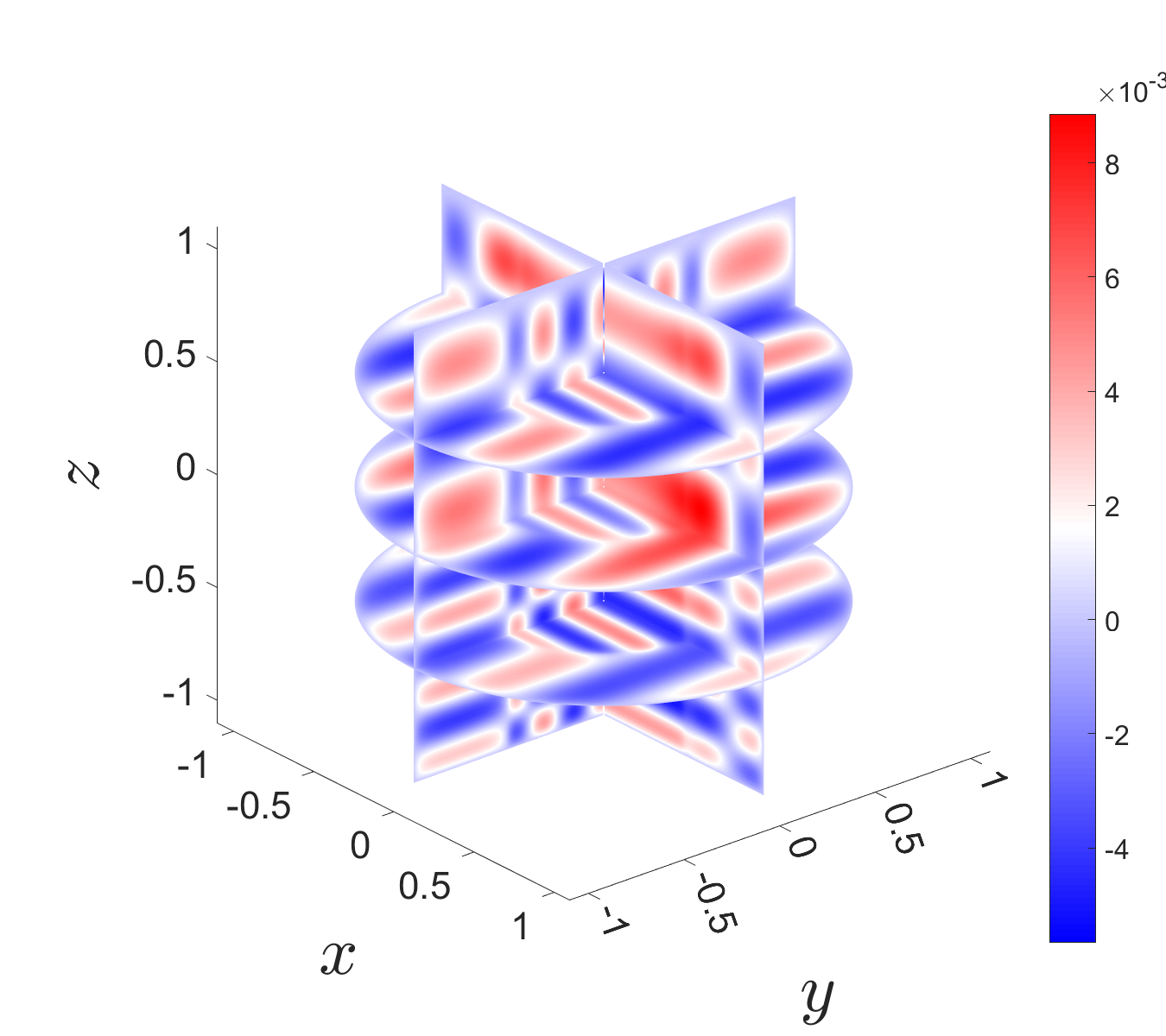}
\includegraphics[width =0.32 \textwidth]{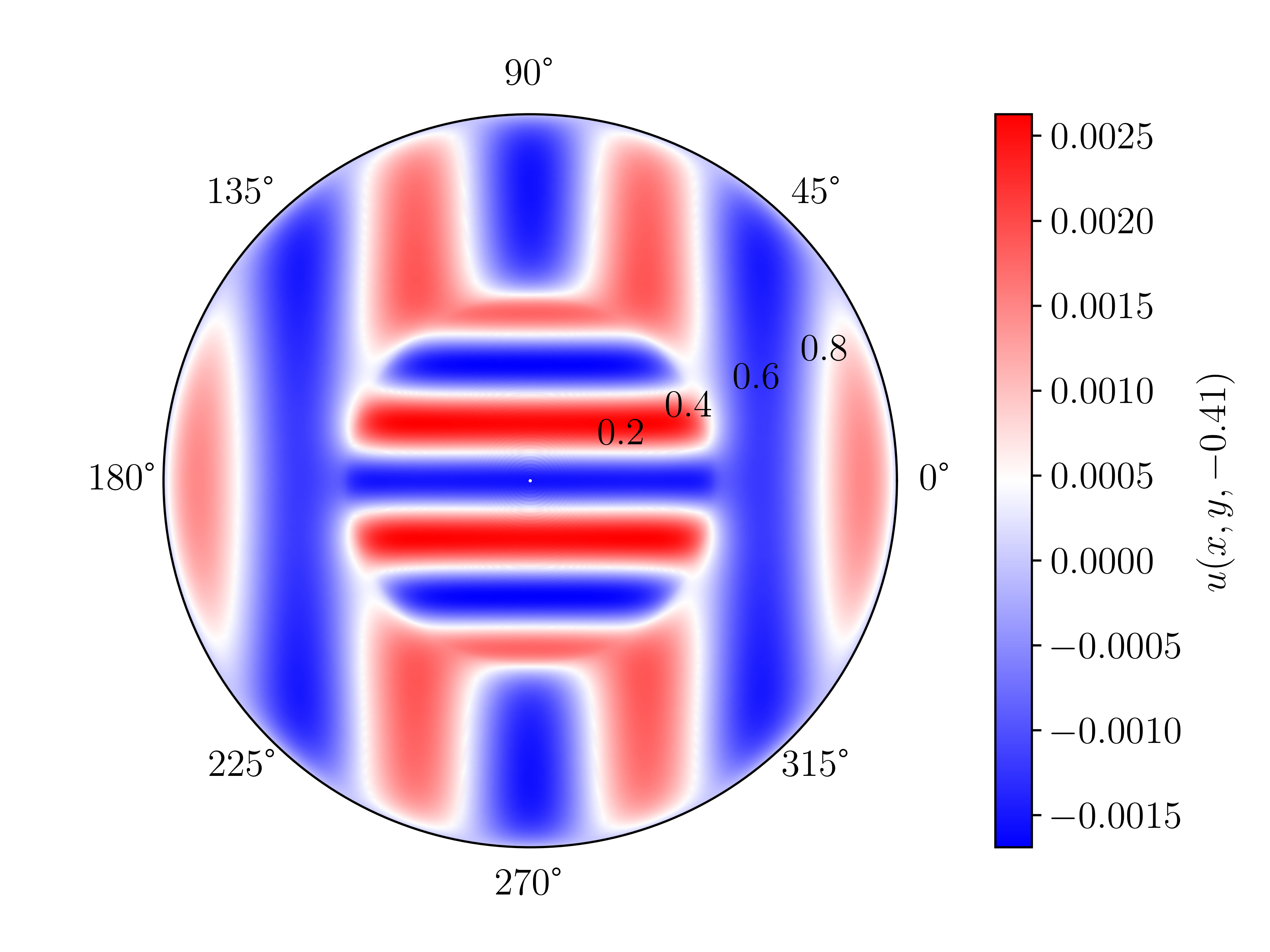}
\includegraphics[width =0.32 \textwidth]{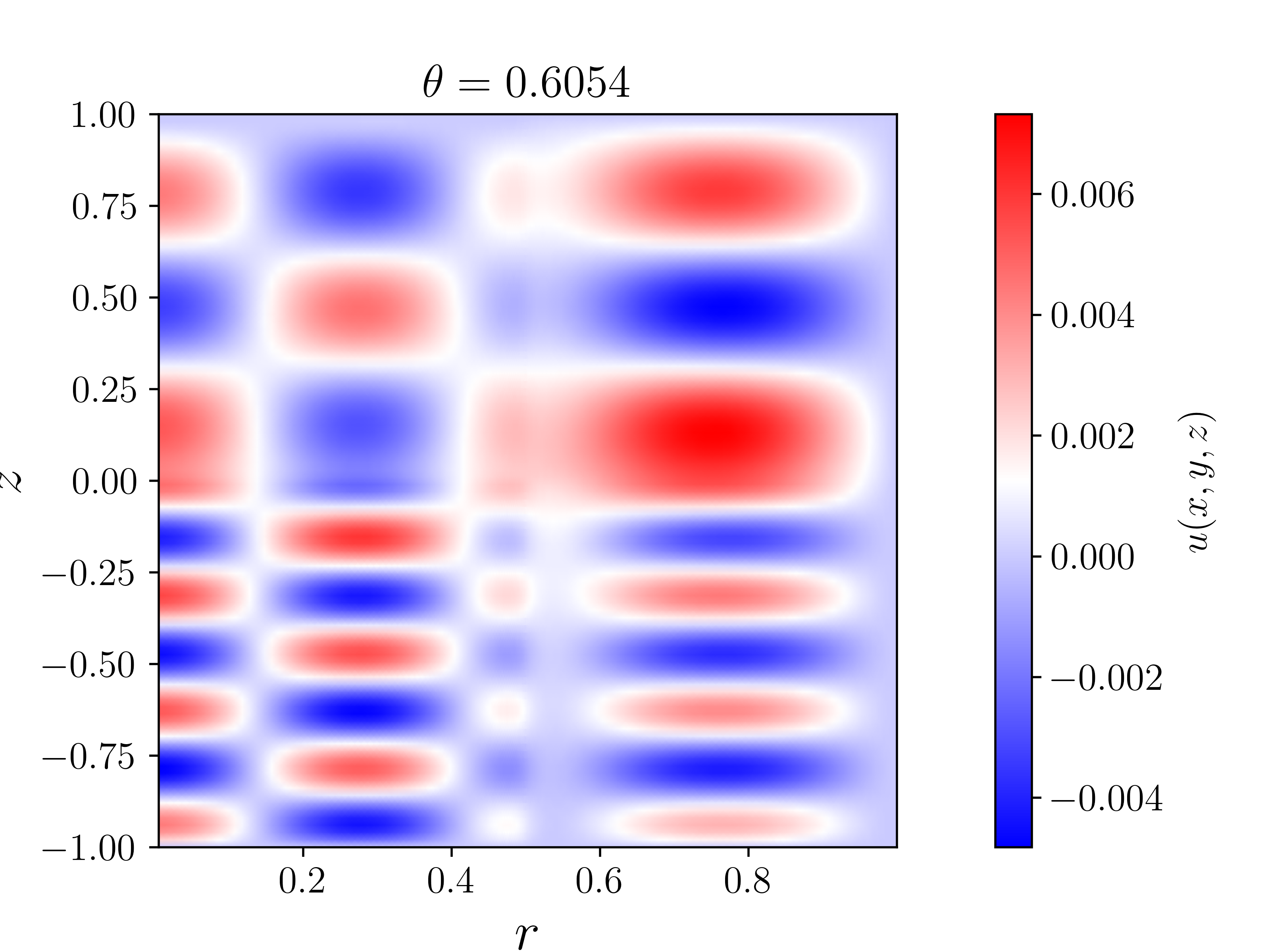}
\caption{Plots of the right-hand side, $f(x,y,z)$ (top row) and the solution $u(x,y,z)$ (bottom row) in the screened 3D Poisson equation of \cref{sec:examples:ADI} with the right-hand side and Helmholtz coefficient as given in \cref{eq:adi:6}. The first column is a visualization on the 3D domain and the second and third columns are 2D slices through the $(x,y)$ and $z$-planes, respectively. Note the discontinuities in $f(x,y,z)$ at $r=1/2$ and $z=0$.}
\label{fig:adi-plots2}
\end{figure}

Studies on high-order FEM and spectral element methods (SEM) is an extremely active area of research \cite{Babuska1981a, Babuska1981b, Fortunato2021, Burns2020, Brubeck2022}. High-order methods typically lead to fast convergence to the true solution, stabilize discretizations, and avoid pitfalls associated with low-order methods, e.g.~locking in linear elasticity \cite{Ainsworth2022}. The computational bottleneck is almost always the loss of sparsity, assembly costs, and the ill-conditioning of the matrices that arise after discretization of the equation operators. For many classical FEMs, attempting to assemble the induced mass matrix on a single three-dimensional element with a truncation degree of $N_p=30$ will surpass the working memory of a standard desktop. One remedy is the use of matrix-free Krylov methods that only require the action of the discretized operators on vectors. However, since the mass and stiffness matrices are often ill-conditioned, a good preconditioner is required in order for the Krylov methods to converge in a reasonable number of iterations \cite{Brubeck2022}. 

An alternative to using a classical basis, and developing a preconditioner for each problem, is to develop a basis that promotes sparsity in the discretized problem even for a high polynomial order and number of elements. The advantage is that the matrices may be explicitly assembled and a direct solver employed for fast convergence. \emph{Sparse} spectral methods may be traced back to the   integral reformulation method of Clenshaw, which uses recurrence relationships for integration of Chebyshev polynomials \cite{clenshaw1957numerical}. In 2013, the ultraspherical spectral method was introduced for general ODEs \cite{Olver2013} and recently the ultraspherical spectral element method was also developed \cite{Fortunato2021}. Using the ideas of the ultraspherical method as a base, sparse spectral methods were constructed for two-dimensional domains e.g.~triangles \cite{Olver2019}, disks/balls \cite{Boyd2011, VasilDisk, Wilber2017,Boulle2020,Meyer2021,Li2014,Atkinson2019,Ellison2022}, annuli \cite{Mahajan1981, Molina2020}, disk slices and trapeziums \cite{Snowball2020}, and spherical caps \cite{Snowball2021}. The cited works mostly focus on discretizing the strong formulation of the partial differential equation (PDE) which has the disadvantage of not preserving the symmetry of the operators (such as the identity). However, we note that some remedies exist \cite{Aurentz2020}.

Utilizing orthogonal polynomials to construct high-order sparse finite element methods stems back to the analysis of the one-dimensional hierarchical $p$-FEM  basis first attributed to Szab\'o and Babu{\v{s}}ka \cite{babuvska1983lecture}, see also \cite[Ch.~2.5.2]{szabo2011introduction} and \cite[Ch.~3.1]{Schwab1998}. The mass and stiffness matrices induced by this basis have a special sparsity structure recently coined as a $B^3$-Arrowhead matrix structure \cite[Def.~4.1]{Olver2023}. In \cite[Cor.~4.3]{Olver2023}, it was shown that $B^3$-Arrowhead matrices permit an optimal complexity $\mathcal{O}(N_h N_p)$ reverse Cholesky factorization, ultimately leading to a quasi-optimal $\mathcal{O}(N_h N_p \log^2 N_p)$ complexity solve for the 1D Poisson equation: from the expansion of the right-hand side to the evaluation of the solution on a grid. Note that other efficient solvers exist, e.g.~via static condensation \cite[Ch.~3.2]{Schwab1998}. 

Extensions to two-dimensional quadrilateral finite elements may be achieved via a tensor-product space of the 1D hierarchical $p$-FEM basis or by constructing a serendipity element, cf.~\cite{Babuska1991} and \cite[Ch.~4.4]{Schwab1998}. Recently it was shown that the tensor-product space admits, via the ADI algorithm \cite{Fortunato2020}, a quasi-optimal $\mathcal{O}(N_h N_p^2 \log^2 N_p)$ solve for the Poisson equation on a quadrilateral domain \cite{Olver2023}. Extensions to two-dimensional simplex finite elements have also been considered, e.g.~by Babu{\v{s}}ka et al.~\cite{Babuska1991} as well as Beuchler and Sch\"oberl \cite{Beuchler2006}. Other works of a similar theme include \cite{Beuchler2007, Beuchler2012, Jia2022, Beuchler2012b, Dubiner1991, Karniadakis2005} and \cite[App.~A]{Snowball2020}. The FEM basis constructed in this work may be thought of as an extension of these other hierarchical bases to the disk. 



The choice of the mesh in this work caters towards solving high-frequency Helmholtz equations with radial discontinuities in the Helmholtz coefficient $\lambda$ and the right-hand side $f$. We consider such an example in \cref{sec:examples:high-frequency}. Unlike spectral method discretizations of the strong form, this approach preserves symmetry and positive-definiteness. This makes it suitable for a unitary preserving discretization of the time-dependent Schr\"odinger  equation via an exponential integrator as considered in \cref{sec:examples:schrodinger}.  We show that the method can tackle rotationally anisotropic coefficients in \cref{sec:examples:non-separable} and singular source terms via $hp$-refinement in \cref{sec:examples:hp-refinement}.  By considering the tensor-product with a univariate basis, the hierarchical basis extends to three-dimensional cylinders. By utilizing the ADI algorithm \cite{Fortunato2020}, the three-dimensional solve has $\mathcal{O}(N_h N_p^3 \log (N_h^{1/4} N_p))$ quasi-optimal complexity as discussed in \cref{sec:ADI}.

\section{Mathematical setup}

Let $\Omega \subset \mathbb{R}^d$, $d \in \{2,3\}$, denote an open, bounded, and connected domain with a Lipschitz boundary. In this work $\Omega$ is a disk, an annulus, or a cylinder.  For $0 < a < b$, we denote a disk and annulus domain, respectively, as
\begin{align}
\Omega_{0, a} &\coloneqq \{ (x,y) \in \mathbb{R}^2 : \|(x,y)\|_{2} <a\},\\
\Omega_{a, b} &\coloneqq \{ (x,y) \in \mathbb{R}^2 : a < \|(x,y)\|_{2} <b\},
\end{align}
where $\|\cdot\|_{2}$ denotes the Euclidean norm. We use $\Omega_0$ to denote the unit disk $\Omega_0 \coloneqq \Omega_{0,1}$  and $\Omega_\rho \coloneqq \Omega_{\rho,1}$, $0 < \rho < 1$, for an annulus with outer radius one.

Let $W^{s,p}(\Omega)$ denote the family of Sobolev spaces \cite{Adams2003} and $H^{s}(\Omega) \coloneqq W^{s,2}(\Omega)$, $s > 0$. We denote the Lebesgue space by $L^p(\Omega)$, $p \in [1,\infty]$. $H^1_0(\Omega)$ denotes the space of functions that live in $H^1(\Omega)$ that have a boundary trace of zero \cite{Gagliardo1957}. Moreover, let $H^{-1}(\Omega)$ denote the dual space of $H^1_0(\Omega)$. If $X$ is a Banach space and $H$ is a Hilbert space, then $\langle \cdot, \cdot \rangle_{X^*,X}$ denotes the duality pairing between a function in $X$ and a functional in the dual space $X^*$ and $\langle \cdot, \cdot\rangle_H$ denotes the inner product in $H$.

Although we consider more complex equations in \cref{sec:examples}, the canonical equation that exemplifies the core principles of the $hp$-FEM is the Helmholtz equation with a variable coefficient. The Helmholtz equation seeks a $u \in H^1_0(\Omega)$ that satisfies, for a given $\lambda \in L^2(\Omega)$ and $f \in H^{-1}(\Omega)$:
\begin{align}
\langle \nabla u, \nabla v\rangle_{L^2(\Omega)} + \langle \lambda u,  v\rangle_{L^2(\Omega)} = \langle f, v \rangle_{H^{-1}(\Omega), H^1_0(\Omega)} \;\; \text{for all}\;\; v \in H^1_0(\Omega). \label{eq:helmholtz}
\end{align}
If $\lambda \geq 0$ a.e.~in $\Omega$, then the existence and uniqueness of $u$ follows as a direct consequence of the Lax--Milgram theorem \cite{Evans2010}. In such a regime, the equation is coercive and we hereby refer to this case as the \emph{screened Poisson equation}. In contrast large negative choices of $\lambda < 0$ induce oscillations in the solution which are traditionally hard to resolve with low-order numerical methods.

We now rewrite \cref{eq:helmholtz} in quasimatrix notation \cite{SOActa}.  Let $\Phi = \{ \phi_i\}_{i=0}^\infty$ denote the set of continuous piecewise polynomials that form the hierarchical basis for $H^1_0(\Omega)$. Then the quasimatrix ${\bm \Phi}$ is defined to be a row vector where each entry is a basis function, i.e.
\begin{align}
{\bm \Phi}(x,y) \coloneqq \left( \phi_0(x,y) \;\; \phi_1(x,y) \;\; \phi_2(x,y) \;\; \dots\right).
\end{align}
Linear operations such as $\fdx$ acting on quasimatrices are understood entry-wise. For any function $u \in H^1(\Omega)$, there exists a column coefficient vector ${\bf u}$ such that $u(x,y) = \sum_{i=0}^\infty {\bf u}_i \phi_i(x,y)= {\bm \Phi}(x,y) {\bf u}$. Throughout this work we expand the right-hand side in a basis of discontinuous piecewise polynomials denoted by $\Psi$. We define the $L^2$-inner product between the two quasimatrices ${\bm \Phi}$ and ${\bm \Psi}$ as follows:
\begin{align}
 \langle{\bm \Phi}^\top, {\bm \Psi}\rangle_{L^2(\Omega)} \coloneqq
 \begin{pmatrix}
\langle \phi_0, \psi_0\rangle_{L^2(\Omega)} & \langle\phi_0, \psi_1\rangle_{L^2(\Omega)}  & \cdots \\
\langle\phi_1, \psi_0\rangle_{L^2(\Omega)} & \langle\phi_1, \psi_1\rangle_{L^2(\Omega)} & \cdots \\
\vdots & \vdots &  \ddots
 \end{pmatrix}.
\end{align} 
Rewrite $u(x,y) = {\bm \Phi}(x,y) {\bf u}$ and consider $f \in L^2(\Omega)$ such that $f(x,y) = {\bm \Psi}(x,y) {\bf f}$. We define the load vector as ${\bf b} 
\coloneqq G_{\Phi, \Psi} {\bf f}$ where $G_{\Phi,\Psi} \coloneqq \langle {\bm \Phi}^\top, {\bm \Psi}\rangle_{L^2(\Omega)}$ is the Gram matrix between $\Phi$ and $\Psi$. We find that the Helmholtz equation \cref{eq:helmholtz} may be rewritten as find ${\bf u}$ that satisfies
\begin{align}
(A + M_\lambda) {\bf u} = {\bf b}
\label{eq:helmholtz-LA}
\end{align}
where the stiffness matrix, $A = \langle (\nabla {\bm \Phi})^\top, \nabla {\bm \Phi}\rangle_{L^2(\Omega)}$, and the weighted mass matrix, $M_\lambda  = \langle {\bm \Phi}^\top, \lambda {\bm \Phi}\rangle_{L^2(\Omega)}$, are symmetric infinite-dimensional matrices. $M_\lambda$ (provided $\lambda \geq 0$ a.e.) and $A$ are symmetric positive-definite. If $\lambda \equiv 1$ then we drop the subscript ${}_\lambda$ and call $M$ the mass matrix. The goal of this work is to choose the hierarchical basis $\Phi$ with spectral approximation properties that promotes sparsity in the stiffness and mass matrices for an annulus and disk domain.

\begin{remark}[Alternative discretizations for disk domains]
Other sparsity preserving discretizations, that decompose the domain, already exist for the disk. A direct extension of the spectral Galerkin method in \cite{shen1997} considers a tensor-product basis of the classical one-dimensional hierarchical $p$-FEM basis in the radial direction \cite{Babuska1991,Olver2023} with a Fourier discretization in the angular direction, i.e.~the basis functions are of the form $P_n(r) \sin(m \theta + j\pi/2)$, $r \in [0,1]$, $\theta \in [0,2\pi)$, $m \in \mathbb{N}$, $j\in\{0,1\}$, where $P_n$ is a piecewise polynomial of degree $n$.  
However, the FEM established in this work has the following advantages:
\begin{enumerate}
\itemsep=0pt
\item Zernike polynomials often represent functions to the same accuracy with half the degrees of freedom when compared to a Chebyshev$\otimes$Fourier expansion \cite[Sec.~6.2]{Boyd2011}.
\item The discretization preserves symmetry of the PDE operator on annular cells unlike in the spectral Galerkin method found in \cite[Sec.~5.2]{shen1997}. Kwan \cite{kwan2009} recovered symmetry but the basis functions are nonstandard. 
\item $A + M_\lambda$ is approximately twice as sparse as that of a $p$-FEM$\otimes$Fourier expansion, with $7$ or less nonzero entries per row when $\lambda$ is piecewise constant.
\item Thanks to the symmetric $B^3$-Arrowhead structure, the $n \times n$ linear systems admit a simple optimal complexity $\mathcal{O}(n)$ factorization via a reverse Cholesky factorization.
\item On disk domains, there have been several studies showing that discretizations of the form $P_n(r) \sin(m \theta + j\pi/2)$, $r \in [0,1]$, $\theta \in [0,2\pi)$, that treat the origin as a boundary point are suboptimal. They suffer from both coordinate singularities and artificially clustered grids at the origin cf.~\cite[Sec.~6]{Boyd2011}, \cite[Sec.~2]{Wilber2017}, \cite[Sec.~1]{VasilDisk}. In turn this leads to worse approximations and numerical instabilities. There are some alternatives, e.g.~using double-wrapped Chebyshev \cite{Wilber2017}, leveraging a quadratic argument \cite{Boyd2011, kwan2009} or, as discussed in this work, one-sided Jacobi polynomials (i.e.~Zernike polynomials) that resolve the origin coordinate singularity issue \cite{Papadopoulos2023, Boyd2011, VasilDisk}.
\end{enumerate}
Although the FEM matrix entries must be numerically computed, as discussed in \cref{rem:quadrature}, the computational cost to find the entries is optimal, scaling as $O(N_h N_p^2)$. Moreover, for coefficients that are not piecewise constant, the entries must also be numerically approximated in the case of a $p$-FEM $\otimes$ Fourier discretization.
\end{remark}

\section{Orthogonal polynomials}

In the previous section, we noted that our goal is to construct an FEM basis that promotes sparsity in the stiffness and mass matrices. In \cref{sec:basis} we show that a suitable basis consists of so-called hat and bubble functions, otherwise known as external and internal shape functions, respectively. In this section we introduce the multivariate orthogonal polynomials that are used to define the hat and bubble functions. 

\subsection{Jacobi and semiclassical Jacobi polynomials}

At its core, the hat and bubble functions consist of scaled-and-shifted (semi)classical Jacobi polynomials multiplied with harmonic polynomials. The Jacobi polynomials $\{ P^{(a,b)}_n(x) \}_{n\in \mathbb{N}_0}$ are a family of complete univariate bases of classical orthonormal polynomials on the interval $(-1,1)$ with basis parameters $a, b \in \mathbb{R}$ such that $a, b >-1$ \cite[Sec.~18.3]{dlmf}. They are orthonormal with respect to the $L^2$-weighted inner product
\begin{align}
\int_{-1}^1  P^{(a,b)}_n(x) P^{(a,b)}_m(x) \, (1-x)^a (1+x)^b \dx = \delta_{nm}.
\end{align}
A number of common orthogonal polynomials are special cases of Jacobi polynomials, e.g.~Chebyshev and ultraspherical polynomials. A \emph{weighted} orthogonal polynomial refers to an orthogonal polynomial multiplied by its orthogonality weight, e.g.~the weighted Jacobi polynomials are $W^{(a,b)}_n(x) \coloneqq (1-x)^a (1+x)^b P_n^{(a,b)}(x)$. 

Semiclassical Jacobi orthogonal polynomials $\{ Q^{t,(a,b,c)}_n(x) \}_{n\in \mathbb{N}_0}$ are a shifted generalization of the Jacobi polynomials. These are univariate orthogonal polynomials with respect to the inner product
\begin{align}
\int_{0}^1  Q^{t,(a,b,c)}_n(x) Q^{t,(a,b,c)}_m(x) \, x^a (1-x)^b (t-x)^c \dx = \delta_{nm},
\end{align}
where $t > 1$, $a,b > -1$, and $c \in \mathbb{R}$. They were introduced by Magnus \cite[Sec.~5]{Magnus1995} and are a building block for a variety of methods. When $c = 0$, these become scaled-and-shifted Jacobi polynomials and we drop the $t$ dependence. That is, we have for any $t>1$,
\begin{align}
Q_n^{t,(a,b,0)}(x) = 2^{(a+b)/2} P_n^{(a,b)}(1-2x).
\end{align}
As with all univariate orthogonal polynomials, a three term recurrence exists for the generation of the (semiclassical) Jacobi polynomials. Equivalently, there exist tridiagonal Jacobi matrices, denoted by $X_{(a,b)}$ and $X_{t,(a,b,c)}$, such that
\begin{align}
x {\bf P}^{(a,b)}(x)  = {\bf P}^{(a,b)}(x) X_{(a,b)}, \;\; x {\bf Q}^{t,(a,b,c)}(x)  = {\bf Q}^{t,(a,b,c)}(x) X_{t,(a,b,c)},
\end{align}
where ${\bf P}^{(a,b)}(x)$ and ${\bf Q}^{t,(a,b,c)}(x)$ denote the quasimatrix of the bases $\{ P^{(a,b)}_n(x) \}_{n\in \mathbb{N}_0}$ and $\{ Q^{t,(a,b,c)}_n(x) \}_{n\in \mathbb{N}_0}$, respectively. 

The following lemma concerning the integral of the (semiclassical) Jacobi weights is used in the construction of the stiffness matrix of the hierarchical basis introduced in \cref{sec:basis}.
\begin{lemma}[Normalization]
\label{lem:normalization}
For $a, b > -1$, $c \in \mathbb{R}$, and $t > 1$,
\begin{align}
p_{(a,b)} &\coloneqq \int_{-1}^1 (1-x)^a (1+x)^b \, \dx = 2^{a+b+1} \beta(a+1,b+1),\\
q_{t,(a,b,c)} &\coloneqq \int_0^1 x^a (1-x)^b (t-x)^c \, \dx = t^c \beta(1+a,1+b) {_2}F_1\left(\begin{matrix}1+a,-c\\ 2+a+b\end{matrix};1/t\right),
\end{align}
where $\beta$ is the Beta function \cite[Sec.~5.12]{dlmf} and ${_2}F_1$ is the Gauss hypergeometric function \cite[Sec.~15.2]{dlmf}. 
\end{lemma}

\subsection{Generalized Zernike and Zernike annular polynomials}

We denote the generalized Zernike polynomials by $Z^{(a)}_{n,m,j}(x,y)$. These are two-dimensional multivariate orthogonal polynomials in the Cartesian coordinates $(x,y)$ defined on the unit disk. Here $a$ is the Zernike weight parameter for the orthogonality weight, $n$ denotes the polynomial degree, $m$ denotes the Fourier mode, and $j$ denotes the Fourier sign. For $n$ odd,  $m \in \{ 1, 3, \ldots, n\}$, and for $n$ even, $m \in \{0, 2, \ldots, n\}$. If $m = 0$ then $j = 1$, otherwise $j \in  \{0,1\}$. Throughout this work, we denote the polar coordinates by $(r, \theta)$ where $r^2 = x^2 + y^2$ and $\theta = \mathrm{atan}(y/x)$. We define the generalized Zernike polynomials as
\begin{align}
Z^{(a)}_{n,m,j}(x,y) \coloneqq Y_{m,j}(x,y) P_{(n-m)/2}^{(a,m)}(2r^2-1),
\label{eq:zernike-def}
\end{align}
where   
\begin{align}
Y_{m,j}(x,y) \coloneqq r^m \times
\begin{cases}
\cos(m \theta) & \text{if $m \in \mathbb{N}_0$ and $j=1$},\\
\sin(m \theta) & \text{if $m \in \mathbb{N}$ and $j=0$},
\end{cases}
\end{align} 
are the harmonic polynomials, orthogonal on the surface of the disk. The generalized Zernike polynomials satisfy
\begin{align}
\iint_{\Omega_0} Z^{(a)}_{n,m,j}(x,y) Z^{(a)}_{\nu,\mu,\zeta}(x,y) (1-r^2)^a \dx\dy = \frac{\pi_m}{2^{m+a+2}} \delta_{n\nu} \delta_{m \mu} \delta_{j\zeta},
\end{align}
where $\pi_m = 2\pi$ if $m=0$ and $\pi_m = \pi$ for $m \in \mathbb{N}$.

The generalized Zernike annular polynomials are the extension of the Zernike polynomials to the annulus domain and are used  in the construction of gyroscopic polynomials \cite{Ellison2023}. They are denoted by $Z^{\rho,(a,b)}_{n,m,j}(x,y)$, where $0 < \rho < 1$, with the same relationship between $n$, $m$ and $j$ as for the Zernike polynomials. The non-generalized family (orthogonal with respect to the unweighted $L^2$-norm) was first introduced by Tatian \cite{Tatian1974} and Mahajan \cite{Mahajan1981}. We define the generalized family as:
\begin{align}
Z^{\rho,(a,b)}_{n,m,j}(x, y) \coloneqq Y_{m,j}(x,y) Q_{(n-m)/2}^{t,(a,b,m)}\left(\frac{1-r^2}{1 - \rho^2}\right).
\label{def:2Dannuli}
\end{align}
Throughout this work we denote $\tau \coloneqq t(1-r^2)$ where $t = (1-\rho^2)^{-1}$. Note that
\begin{align}
1-r^2 &= t^{-1}\tau, \;\; r^2 = t^{-1}(t - \tau), \;\; \text{and} \;\; r^2 - \rho^2 = t^{-1}(1 - \tau).
\label{eq:tau}
\end{align}
Utilizing \cref{eq:tau}, one finds that $Z^{\rho, (a,b)}_{n,m,j}$ satisfy \cite{Papadopoulos2023}
\begin{align}
\iint_{\Omega_\rho} Z^{\rho, (a,b)}_{n,m,j}(x,y) Z^{\rho, (a,b)}_{\nu,\mu,\zeta}(x,y) (1-r^2)^a(r^2-\rho^2)^b  \dx\dy = \frac{\pi_m}{2t^{a+b+m+1}} \delta_{n\nu} \delta_{m \mu} \delta_{j\zeta}.
\end{align}

Out of the spectral methods they considered for disks, Boyd and Yu \cite{Boyd2011} noted that Zernike polynomials often offer the best approximation per degree of freedom and a similar observation was made for the annulus \cite[Sec.~6]{Papadopoulos2023}. In order to compute with these polynomials quickly, we heavily rely on new methods for quasi-optimal $\mathcal{O}(N_p^2 \log N_p)$ complexity analysis (expansion) and synthesis (evaluation) operators introduced by Slevinsky \cite{Slevinsky2019} and Gutleb et al.~\cite{Gutleb2023} and further studied in \cite[Sec.~4.2]{Papadopoulos2023}. In a nutshell, for the analysis, one expands a function in a Chebyshev--Fourier series (for which a fast transform exists) and utilizes fast transforms to convert these expansion coefficients to those of the Zernike (annular) expansion. The synthesis operator is the reverse process.

As many operators decompose across Fourier modes, it is useful to consider each Fourier mode of the Zernike (annular) polynomials separately. Hence, we define the quasimatrix ${\bf Z}^{\rho,(a,b,c)}_{m,j}(x,y)$ as
\begin{align}
{\bf Z}^{\rho,(a,b)}_{m,j}(x,y) \coloneqq
\left(
Z^{\rho,(a,b)}_{m,m,j}(x,y) \, | \, Z^{\rho,(a,b)}_{m+2,m,j}(x,y)\, | \, Z^{\rho,(a,b)}_{m+4,m,j}(x,y)\, | \, \cdots 
\right).
\end{align}
The quasimatrix ${\bf Z}^{(a)}_{m,j}(x,y)$ is defined analogously. 

\subsection{Raising and Laplacian matrices}
In the next section, we will show that the (weighted) mass and stiffness matrices may be computed via the raising matrices for (semiclassical) Jacobi polynomials and the Laplacian matrices for Zernike (annular) polynomials.

\begin{remark}[Fast assembly]
\label{rem:quadrature}
In the next section, we will show how to compute the entries of the stiffness and mass matrices without any need for quadrature. Moreover, the entries of the load vector may also be computed in quasi-optimal complexity. This is achieved by considering hierarchies of raising matrices for (semiclassical) Jacobi polynomial families that can be computed in optimal complexity as detailed in \cite[Sec.~4.1]{Papadopoulos2023}, see also \cite{Gutleb2023, Slevinsky2019}. The right-hand side is expanded in Zernike annular polynomials by leveraging a fast DCT $\otimes$ FFT transform to compute the coefficients of a Chebyshev $\otimes$ Fourier expansion. These are then converted to the coefficients of a Zernike annular expansion via the hierarchy of raising matrices. How we then use the raising matrices to compute the entries of the FEM matrices and vectors is the main focus of \cref{sec:basis}.

Using raising matrices rather than quadrature to assemble the FEM matrices results in a fast and stable FEM framework.
\end{remark}

We now define the hierarchy of raising matrices.

\begin{definition}[Raising matrices]
\label{def:raising-matrices}
For $m \in \mathbb{N}_0$, we denote the raising matrix for weighted Jacobi and weighted semiclassical Jacobi polynomials by  $R^{(0,m)}_{\mathrm{a}, (1,m)}$ and $R^{t,(0,0,m)}_{\mathrm{ab}, (1,1,m)}$, respectively, where
\begin{align}
(1-x) {\bf P}^{(1,m)}(x) &= {\bf P}^{(0,m)}(x) R^{(0,m)}_{\mathrm{a}, (1,m)},\\
x(1-x) {\bf Q}^{t,(1,1,m)}(x) &= {\bf Q}^{t,(0,0,m)}(x) R^{t,(0,0,m)}_{\mathrm{ab}, (1,1,m)}.
\end{align}
$R^{(0,m)}_{\mathrm{a}, (1,m)}$ and $R^{t,(0,0,m)}_{\mathrm{ab}, (1,1,m)}$ are lower triangular matrices with lower bandwidths one and two, respectively.
\end{definition}
Recently a fast QR factorization technique was introduced that allows one to compute the semiclassical Jacobi hierarchy of raising matrices $\{R^{t,(0,0,m)}_{\mathrm{ab}, (1,1,m)}\}_{m \in \{0,1,\dots,N_p\}}$ in $\mathcal{O}(N_p^2)$ complexity \cite[Sec.~3.1]{Papadopoulos2023}. Thanks to explicit expressions for $R^{(0,m)}_{\mathrm{a}, (1,m)}$, one may compute the Jacobi hierarchy in the same complexity.

\begin{definition}[Laplacian matrices]
\label{def:laplacian-matrices}
For $m \in \mathbb{N}_0$, we denote the Laplacian matrix for weighted Zernike and weighted Zernike annular polynomials by  $D_m$ and $D^\rho_m$, respectively, where
\begin{align}
\Delta[(1-r^2) {\bf Z}^{(1)}_{m,j}](x,y) &= {\bf Z}^{(0)}_{m,j}(x,y) D_m,\\
\Delta[(1-r^2) (r^2-\rho^2) {\bf Z}^{(1,1)}_{m,j}](x,y) &= {\bf Z}^{(0,0)}_{m,j}(x,y) D^\rho_m.
\end{align}
$D_m$ and $D_m^\rho$ are diagonal and tridiagonal matrices, respectively.
\end{definition}
The same techniques that allow us to compute the raising matrices in optimal complexity, also allow us to compute the hierarchies of Laplacian matrices in optimal $\mathcal{O}(N_p^2)$ complexity \cite[Sec.~3.4]{Papadopoulos2023}.

\section[FEM Basis]{The FEM basis: hat and bubble functions}
\label{sec:basis}
In this section we construct the hat and bubble functions that form the continuous hierarchical FEM basis $\Phi$ for disk and annulus domains.

\begin{definition}[Radial affine transformation]
Consider the disk and annular cells $K_0 = \Omega_{0, \rho}$, $\rho>0$, $K_1 = \Omega_{\rho_1, \rho_2}$, $0 < \rho_1 < \rho_2$. Then we define the radial affine transformation of the Zernike (annular) polynomials as
\begin{align}
{\bf Z}^{K_{0}, (a)}(x,y) &\coloneqq {\bf Z}^{(a)}\left(\frac{x}{\rho}, \frac{y}{\rho}\right) \;\;\text{and}\;\;
{\bf Z}^{K_{1}, (a,b)}(x,y) \coloneqq {\bf Z}^{\rho_1/\rho_2,(a,b)}\left(\frac{x}{\rho_2}, \frac{y}{\rho_2}\right).
\end{align}
\end{definition} 
 
\begin{definition}[Bubble functions]
\label{def:bubble}
Consider the disk and annular cells $K_0 = \Omega_{0,\rho}$, $\rho > 0$, and $K_1 = \Omega_{\rho_1,\rho_2}$, $0 < \rho_1 < \rho_2$.  The disk and annulus bubble functions (otherwise known as internal shape functions) are denoted by $B^{K_0}_{n,m,j}(x,y)$ and $B^{K_1}_{n,m,j}(x,y)$, respectively, where
\begin{align}
B^{K_0}_{n,m,j}(x,y) &\coloneqq (1-(r/\rho)^2)Z^{K_0, (1)}_{n,m,j}(x,y), \\
B^{K_1}_{n,m,j}(x,y) &\coloneqq (1-(r/\rho_2)^2)((r/\rho_2)^2 - (\rho_1/\rho_2)^2) Z^{K_1,(1,1)}_{n,m,j}(x,y). 
\end{align}
\end{definition}
Note that the disk bubble functions vanish at $r=\rho$ and the annulus bubble functions vanish at $r = \rho_1$ and $r=\rho_2$.

We now distinguish between hat functions that are supported on two adjacent annulus elements and those supported on the disk element and the adjacent annulus element.

\begin{definition}[Disk-annulus hat functions]
\label{def:hat1}
Consider the disk and annular cells $K_0 = \Omega_{0,\rho_1}$ and $K_1 = \Omega_{\rho_1,\rho_2}$, $0 < \rho_1 < \rho_2$.
The disk-annulus hat functions (otherwise known as external shape functions) are defined as follows:
\begin{align*}
H^{K_0,K_1}_{m,j}(x,y) \coloneqq 
\begin{cases}
\kappa_m Z^{K_0, (0)}_{m,m,j}(x,y)  & \text{in} \;\; K_0,\\
(1-(r/\rho_2)^2) Z^{K_1, (1,0)}_{m,m,j}(x,y)  & \text{in} \;\; K_1.
\end{cases}
\end{align*}
The coefficient $\kappa_m \coloneqq (1-(\rho_1/\rho_2)^2) (\rho_1/\rho_2)^m$ ensures the continuity of $H^{K_0, K_1}_{m,j}(x,y)$ at $r = \rho_1$. 
\end{definition}
The disk-annulus hat functions are only supported on the disk cell and the adjacent annular cell, vanishing at $r=\rho_2$. 

\begin{definition}[Annulus-annulus hat functions]
\label{def:hat2}
Consider the annular cells $K_1 = \Omega_{\rho_1,\rho_2}$ and $K_2 = \Omega_{\rho_2,\rho_3}$, $0 < \rho_1 < \rho_2 < \rho_3$. The annulus-annulus hat functions (otherwise known as external shape functions) are defined as follows
\begin{align*}
H^{K_1, K_2}_{m,j}(x,y) \coloneqq 
\begin{cases}
 \gamma_m ((r/\rho_2)^2 - (\rho_1/\rho_2)^2) Z^{K_1, (0,1)}_{m,m,j}(x,y)  & \text{in} \;\; K_1,\\
(1-(r/\rho_3)^2) Z^{K_2, (1,0)}_{m,m,j}(x,y)& \text{in} \;\; K_2.
\end{cases}
\end{align*}
The coefficient $\gamma_m \coloneqq   (1-(\rho_2/\rho_3)^2) (\rho_2/\rho_3)^m (1-(\rho_1/\rho_2)^2)^{-1}$ ensures the continuity of $H^{K_1, K_2}_{m,j}(x,y)$ at $r = \rho_2$.
\end{definition}
The annulus-annulus hat functions are only supported on the two annular cells, vanishing at $r=\rho_1$ and $r=\rho_3$. In \cref{fig:hats-bubble}, we plot a one-dimensional slice at $\theta = 0$ of the bubble and hat functions with the Fourier mode and sign $(m, j) = (0,1)$ and $(m,j) = (1,1)$ on a two-cell mesh for increasing degree $N_p$. We emphasize that bubble functions are only ever supported on one cell and hat functions are supported on a maximum of two cells.

\begin{figure}[h!]
\centering
\includegraphics[height=5cm]{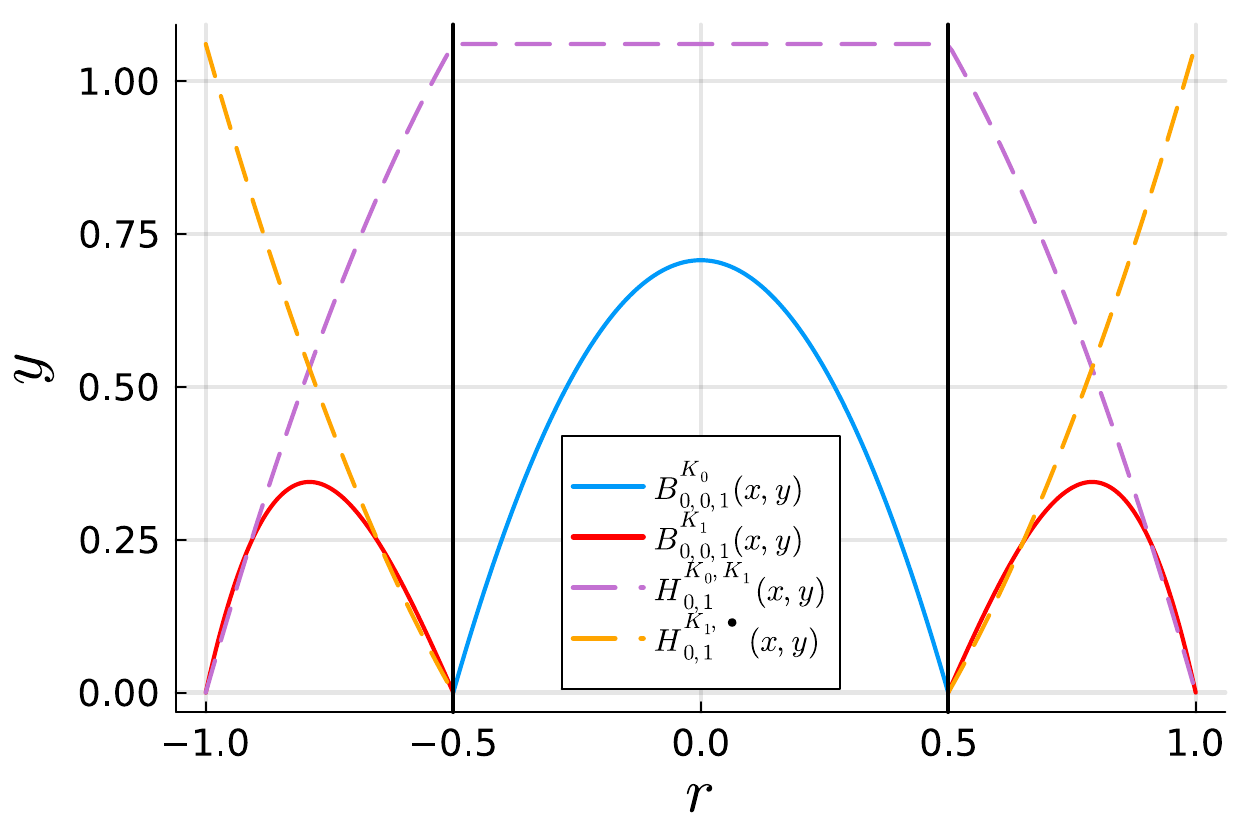}
\includegraphics[height=5cm]{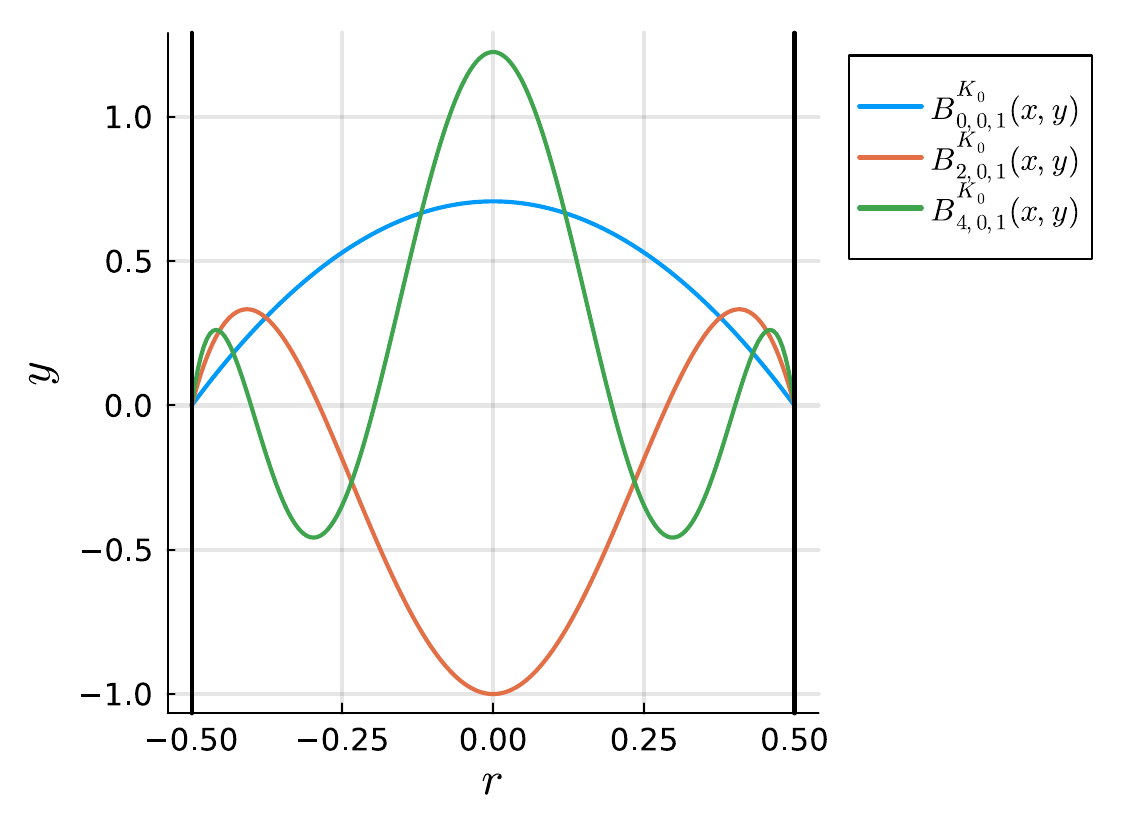} \\
\includegraphics[height=5cm]{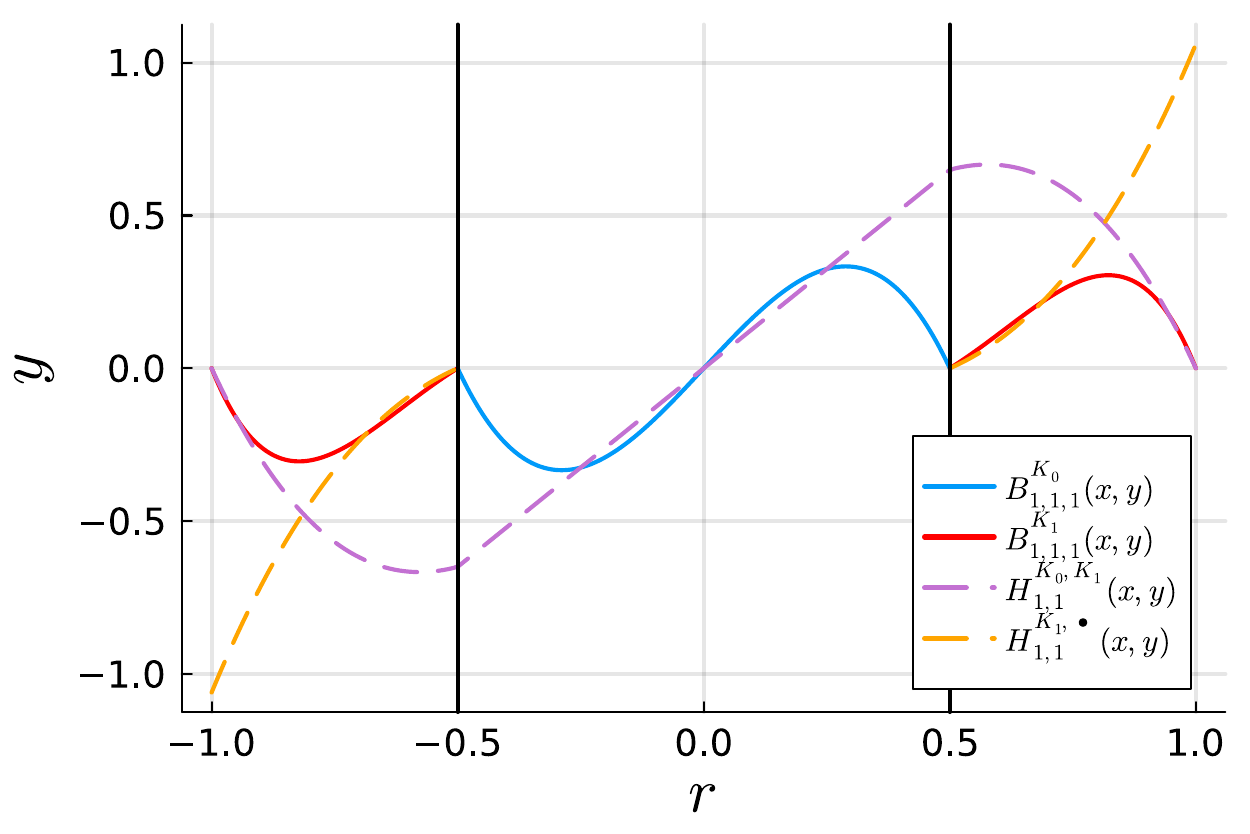}
\includegraphics[height=5cm]{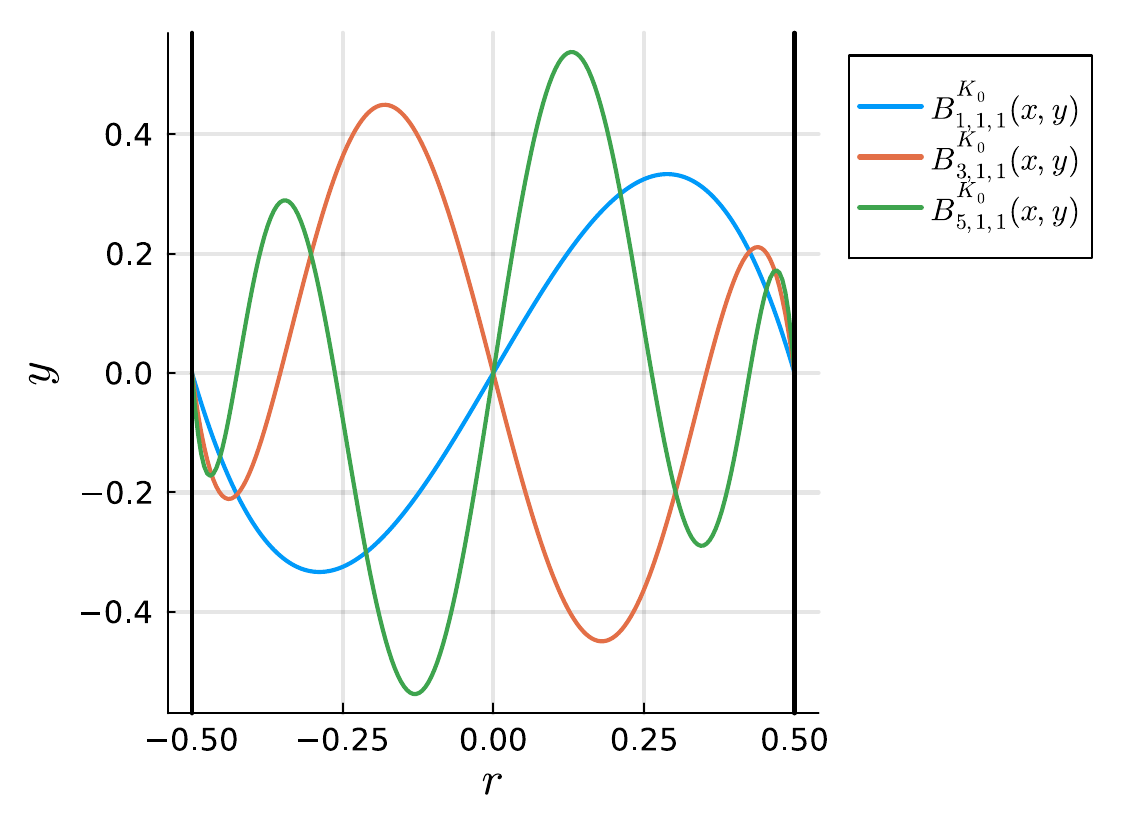}
\caption{Consider the mesh $\mathcal{T}_h = \{ \bar K_0, \bar K_1 \}$, $K_0 = \Omega_{0,1/2}$, $K_1 = \Omega_{1/2,1}$. (Left column) A one-dimensional slice at $\theta=0$ is plotted for the hat and lowest degree bubble functions with the Fourier mode and sign $(m, j)=(0,1)$ (top left) and $(m,j)=(1,1)$ (bottom left) across the two cells. The solid black vertical lines indicate the edges of the cells. The bubble and hat functions are plotted with a solid and dashed line, respectively. (Right column) A one-dimensional slice at $\theta=0$ of the first three bubble functions with $(m,j)=(0,1)$ (top right) and $(m,j) = (1,1)$ (bottom right) on the inner disk cell.
}
\label{fig:hats-bubble}
\end{figure}

We denote the continuous hierarchical basis quasimatrix restricted to the Fourier mode $(m,j)$ by ${\bm \Phi}_{m,j}$. We order the basis functions such that the basis functions of the same degree are grouped together. Hence, the hat functions across all the cells appear first and the bubble basis functions appear after.
Consider the mesh $\mathcal{T}_h = \{\bar K_j\}_{j=0}^{N_h-1}$ where $K_j = \Omega_{\rho_j, \rho_{j+1}}$, $0 = \rho_0 < \rho_1 < \cdots < \rho_{N_h}$. Let
\begin{align}
{\bf H}^{\mathcal{T}_h}_{m,j} &\coloneqq 
\begin{pmatrix}
H^{K_0, K_1}_{m,j} \;\;  \cdots \;\; H^{K_{N_h-2}, K_{N_h-1}}_{m,j} \;\; H^{ K_{N_h-1}, \bullet}_{m,j}
\end{pmatrix} \;\; (\text{$N_h$ hat functions}),\\
{\bf B}^{\mathcal{T}_h}_{n,m,j} &\coloneqq
\begin{pmatrix}
B^{K_0}_{n,m,j} \;\; B^{K_1}_{n,m,j} \;\; \cdots \;\; B^{K_{N_h-1}}_{n,m,j}
\end{pmatrix} \;\; (\text{$N_h$ bubble functions per degree $n$}).
\end{align}
We use the superscript $^\bullet$ in a hat function that is only defined on one cell (such as the hat function at the boundary). ${\bf H}^{\mathcal{T}_h}_{m,j}(x,y)$ and ${\bf B}^{\mathcal{T}_h}_{n,m,j}(x,y)$ are the quasimatrices of the hat and bubble functions, respectively, on the mesh $\mathcal{T}_h$ restricted to the Fourier mode $(m,j)$ and polynomial degree $n$. The hierarchical basis quasimatrix defined on the mesh $\mathcal{T}_h$, restricted to the Fourier mode $(m,j)$ is
\begin{align}
{\bm \Phi}^{\mathcal{T}_h}_{m,j} \coloneqq
\begin{pmatrix}
{\bf H}^{\mathcal{T}_h}_{m,j} & {\bf B}^{\mathcal{T}_h}_{m,m,j} & {\bf B}^{\mathcal{T}_h}_{m+2,m,j} &\cdots
\end{pmatrix}.
\label{def:Phi-mj}
\end{align}

As is standard in FEM, we derive a number results for the local assembly of a matrix with respect to a reference element. The global assembly of the matrices is then deduced in the classical manner. With this in mind, we define the hierarchical basis quasimatrices on the unit disk domain $\Omega_0$ and the annulus domain $\Omega_\rho$ as:
\begin{align}
{\bm \Phi}^{\Omega_0}_{m,j}(x,y) &\coloneqq 
\begin{pmatrix}
H^{\Omega_0, \bullet}_{m,j}(x,y) & B^{\Omega_0}_{m,m,j}(x,y) & B^{\Omega_0}_{m+2,m,j}(x,y) & \cdots
\end{pmatrix},\\
{\bm \Phi}^{\Omega_\rho}_{m,j}(x,y) &\coloneqq 
\begin{pmatrix}
H^{\bullet, \Omega_\rho}_{m,j}(x,y) & H^{\Omega_\rho, \bullet}_{m,j}(x,y) & B^{\Omega_\rho}_{m,m,j} & B^{\Omega_\rho}_{m+2,m,j}(x,y) & \cdots
\end{pmatrix},\\
{\bm \Phi}^{\Omega_0}(x,y) &\coloneqq 
\begin{pmatrix}
{\bm \Phi}^{\Omega_0}_{0,1}(x,y) & {\bm \Phi}^{\Omega_0}_{1,0}(x,y) & {\bm \Phi}^{\Omega_0}_{1,1}(x,y) & \cdots
\end{pmatrix}, \label{def:Phi-Omega-0} \\
{\bm \Phi}^{\Omega_\rho}(x,y) &\coloneqq 
\begin{pmatrix}
{\bm \Phi}^{\Omega_\rho}_{0,1}(x,y) & {\bm \Phi}^{\Omega_\rho}_{1,0}(x,y) & {\bm \Phi}^{\Omega_\rho}_{1,1}(x,y) & \cdots
\end{pmatrix}. \label{def:Phi-Omega-rho}
\end{align}
For clarity, let ${\bm \Phi}^{\Omega}_{m,j}$ be the $i$th entry in \cref{def:Phi-Omega-0} or \cref{def:Phi-Omega-rho}, then $m = \lfloor i /2 \rfloor$ and $j = i \, \mathrm{mod} \, 2$.

\begin{remark}[Homogeneous Dirichlet boundary condition]
A homogeneous Dirichlet boundary condition is enforced by dropping the hat functions in the basis that are nonzero on the boundary of the domain.
\end{remark}

A crucial ingredient for constructing the (weighted) mass matrices in later sections will be the following proposition, which connects our basis to  multivariate orthogonal polynomials with respect to a uniform weight:
\begin{proposition}[Raising operators]
\label{prop:raising}
Consider the unit disk domain $\Omega_0$ and the annulus domain $\Omega_\rho$, $\rho > 0$. Recall the definitions of the raising matrices  $R^{(0,m)}_{\mathrm{a}, (1,m)}$ and $R^{t,(0,0,m)}_{\mathrm{ab}, (1,1,m)}$ from \cref{def:raising-matrices}. Then
\begin{align}
{\bm \Phi}^{\Omega_0}_{m,j}(x,y) = {\bf Z}^{(0)}_{m,j}(x,y) R^{\Omega_0}_{m}  \;\; \text{and} \;\;
{\bm \Phi}^{\Omega_\rho}_{m,j}(x,y) = {\bf Z}^{\rho, (0, 0)}_{m,j}(x,y) R^{\Omega_\rho}_{m} ,
\end{align}
where
\begin{align}
R^{\Omega_0}_{m} \coloneqq 
\left(
\begin{array}{c|c}
\begin{array}{c} 1 \\ 0 \\ \vdots \end{array} & 
\begin{array}{c} \\ \raisebox{0.4\height}[0pt][0pt]{\Large $\frac{1}{2}R^{(0,m)}_{\mathrm{a}, (1,m)}$} \\ \end{array}
\end{array}
\right), 
R^{\Omega_\rho}_{m} \coloneqq \frac{1}{t}
\left(
\begin{array}{cc|c}
\begin{array}{c} r_{11} \\ r_{21} \\ 0 \\ \vdots \end{array} &
\begin{array}{c} r_{12} \\ r_{22} \\ 0 \\ \vdots \end{array} & 
\begin{array}{c} \\ \raisebox{0.2\height}[0pt][0pt]{\Large $\frac{1}{t} R^{t, (0,0,m)}_{\mathrm{ab}, (1,1,m)}$} \\ \\ \end{array}
\end{array}
\right),
\end{align}
such that $( r_{11} \;\; r_{21}  \;\; \cdots)^\top$  and $( r_{12} \;\; r_{22}  \;\; \cdots)^\top$ are the first columns of the lower bidiagonal matrices $R^{t,(0,0,m)}_{\mathrm{a},(1,0,m)}$ and $R^{t,(0,0,m)}_{\mathrm{b},(0,1,m)}$, respectively, where $x {\bf Q}^{t,(1,0,m)}(x) = {\bf Q}^{t,(0,0,m)}(x)R^{t,(0,0,m)}_{\mathrm{a},(1,0,m)}$ and $(1-x) {\bf Q}^{t,(0,1,m)}(x) = {\bf Q}^{t,(0,0,m)}(x)R^{t,(0,0,m)}_{\mathrm{b},(0,1,m)}$. 
\end{proposition}
\begin{proof}
We first consider the unit disk.

\noindent \textbf{(Disk).} Note that the first entries of ${\bm \Phi}^{\Omega_0}_{m,j}(x,y)$ and ${\bf Z}^{(0)}_{m,j}(x,y)$ are both ${Z}^{(0)}_{m,m,j}(x,y)$ which corresponds to the first column in $R^{\Omega_0}_{m}$. Now, for $\eta = 2r^2-1$,
\begin{align}
\begin{split}
{\bf B}^{\Omega_0}_{m,j}(x,y)&= \frac{1}{2} Y_{m,j}(x,y) (1-\eta)  {\bf P}^{(1,m)}(\eta) \\
& = \frac{1}{2} Y_{m,j}(x,y)   {\bf P}^{(0,m)}(\eta )R^{(0,m)}_{\mathrm{a}, (1,m)}
=   {\bf Z}^{(0)}_{m,j}(x,y) \frac{1}{2} R^{(0,m)}_{\mathrm{a}, (1,m)}.
\end{split}
\end{align}
The first and third equalities follow from \cref{eq:zernike-def} and the second equality follows from \cref{def:raising-matrices}.

\noindent \textbf{(Annulus).} Note that ${\bm \Phi}^{\Omega_\rho}_{m,j}= \left( H^{\bullet, \Omega_\rho}_{m,j} \;\; H^{\Omega_\rho, \bullet}_{m,j}  \;\; {\bf B}^{\Omega_\rho}_{m,j} \right)$. Now, for $\tau = t(1-r^2)$, $(x,y) \in \Omega_\rho$,
\begin{align}
\begin{split}
H^{\bullet, \Omega_\rho}_{m,j}(x,y) &= (1-r^2) Z^{\rho, (1,0)}_{m,j}(x,y)\\
& = t^{-1}Y_{m,j}(x,y) \tau {Q}_0^{t,(1,0,m)}(\tau)\\
&= t^{-1}Y_{m,j}(x,y) [r_{11} {Q}_0^{t,(0,0,m)}(\tau) + r_{21} {Q}_1^{t,(0,0,m)}(\tau)]\\
& =  t^{-1} [r_{11} {Z}^{\rho,(0,0)}_{m,m,j}(x,y) + r_{21} {Z}^{\rho,(0,0)}_{m+2,m,j}(x,y)].
\end{split}
\end{align}
The first equality follows by the definition of the hat function, the second and fourth equalities follow from \cref{def:2Dannuli} and the third equality follows from the definition of $r_{11}$ and $r_{12}$. Thus we recover the first column and the second column follows similarly. We recover the remaining columns as follows:
\begin{align}
\begin{split}
{\bf B}^{\Omega_\rho}_{m,j}(x,y) &= (1-r^2) (r^2-\rho^2) Z^{\rho, (1,1)}_{m,j}(x,y)\\
&=t^{-2} Y_{m,j}(x,y) \tau (1-\tau) {\bf Q}^{t,(1,1,m)}(\tau)\\
&=t^{-2} Y_{m,j}(x,y) {\bf Q}^{t,(0,0,m)}(\tau) R^{t, (0,0,m)}_{\mathrm{ab}, (1,1,m)}
= t^{-2} {\bf Z}^{\rho,(0,0)}_{m,j}  R^{t, (0,0,m)}_{\mathrm{ab}, (1,1,m)}.
\end{split}
\end{align}
\end{proof}
\begin{remark}
$R^{\Omega_0}_{m}$ is upper triangular with upper bandwidth one and $R^{\Omega_\rho}_{m}$ is almost upper triangular (the upper-triangular structure is disrupted by the $r_{21}$-entry) with upper bandwidth two.
\end{remark}

\subsection{Mass and stiffness matrices}
\label{sec:mass-and-stiffness}
This subsection focuses on computing the entries of the mass and stiffness matrices.  The entries are computed via the raising matrices defined in \cref{prop:raising}. In particular we emphasize that no quadrature is required. 

\begin{theorem}[Mass matrix]
\label{th:mass}
Consider the unit disk domain $\Omega_0$ and the annulus domain $\Omega_\rho$, $\rho > 0$. Define the mass matrices $M^{\Omega_0} = \langle ({\bm \Phi}^{\Omega_0})^\top,  {\bm \Phi}^{\Omega_0} \rangle_{L^2(\Omega_0)}$ and $M^{\Omega_\rho}  = \langle ({\bm \Phi}^{\Omega_\rho})^\top,  {\bm \Phi}^{\Omega_\rho} \rangle_{L^2(\Omega_\rho)}$. Then both  $M^{\Omega_0}$ and $M^{\Omega_\rho}$ are block-diagonal where the blocks correspond to each Fourier mode:
\begin{align}
\begin{split}
M^{\Omega_0} = 
\begin{pmatrix}
M^{\Omega_0}_{0,1} & & &\\
&M^{\Omega_0}_{1,0}&&\\
&&M^{\Omega_0}_{1,1}&\\
&&&\ddots
\end{pmatrix}, \;\;
M^{\Omega_\rho} = 
\begin{pmatrix}
M^{\Omega_\rho}_{0,1} & & &\\
&M^{\Omega_\rho}_{1,0}&&\\
&&M^{\Omega_\rho}_{1,1}&\\
&&&\ddots
\end{pmatrix}
\end{split}
\label{eq:mass-matrix-block-diagonal}
\end{align}
where
\begin{align}
M^{\Omega_0}_{m,j} = \langle ({\bm \Phi}^{\Omega_0}_{m,j})^\top,  {\bm \Phi}^{\Omega_0}_{m,j} \rangle_{L^2(\Omega_0)} &= \frac{\pi_m}{2^{m+2}}(R^{\Omega_0}_m)^\top R^{\Omega_0}_m, \label{eq:mass-disk}\\
M^{\Omega_\rho}_{m,j} = \langle ({\bm \Phi}^{\Omega_\rho}_{m,j})^\top,  {\bm \Phi}^{\Omega_\rho}_{m,j} \rangle_{L^2(\Omega_\rho)} &=  \frac{\pi_m}{2t^{m+1}} (R^{\Omega_\rho}_m)^\top R^{\Omega_\rho}_m.
\end{align}
Thus $M_{m,0} = M_{m,1}$.
\end{theorem}
\begin{proof}
If either $m \neq \mu$ or $j \neq \zeta$, then
\begin{align}
\begin{split}
\langle ({\bm \Phi}^{\Omega_0}_{m,j})^\top,  {\bm \Phi}^{\Omega_0}_{\mu,\zeta} \rangle_{L^2(\Omega_0)} = \langle ({\bm \Phi}^{\Omega_\rho}_{m,j})^\top,  {\bm \Phi}^{\Omega_\rho}_{\mu,\zeta} \rangle_{L^2(\Omega_\rho)} &= {\bf 0},
\end{split}\label{eq:different-fourier}
\end{align}
where $\vectt{0}$ denotes the infinite-dimensional matrix of zeroes.  \cref{eq:different-fourier} follows by substituting in the definitions of the basis functions and noting that $\int_0^{2\pi} \sin(m\theta + j \pi/2) \sin(\mu \theta + \zeta \pi/2) \, \mathrm{d}\theta = 0$ if either  $m \neq \mu$ or $j \neq \zeta$. In other words hat and bubble functions with different Fourier modes have a mass matrix entry of zero. This implies that the block diagonal structure in \cref{eq:mass-matrix-block-diagonal} holds.

For the disk cell, one finds that
\begin{align}
\begin{split}
&M^{\Omega_0}_{m,j} = \langle ({\bm \Phi}^{\Omega_0}_{m,j})^\top,  {\bm \Phi}^{\Omega_0}_{m,j} \rangle_{L^2(\Omega_0)}
=(R^{\Omega_0}_m)^\top \langle ({\bf Z}^{(0)}_{m,j})^\top,  {\bf Z}^{(0)}_{m,j} \rangle_{L^2(\Omega_0)}R^{\Omega_0}_m\\
&\indent =\left( \int_0^{2\pi} \sin^2(m\theta +j \pi/2)\, \mathrm{d}\theta \right) \\
&\indent \indent \times (R^{\Omega_0}_m)^\top \langle ({\bf P}^{(0,m)})^\top,  2^{-(m+2)} (1+\eta)^m {\bf P}^{(0,m)} \rangle_{L^2(-1,1)}R^{\Omega_0}_m\\
&\indent ={2^{-(m+2)}}{\pi_m}(R^{\Omega_0}_m)^\top R^{\Omega_0}_m.
\end{split}
\end{align}
The first equality follows by the definition of the mass matrix and the second equality from \cref{prop:raising}. The third equality follows from \cref{eq:zernike-def} and a change from Cartesian coordinates $(x,y)$ to polar coordinates $(r,\theta)$ followed a the second change of coordinates $\eta = 2r^2-1$. The final equality follows from a direct evaluation of the $\theta$-dependent integral and the orthogonality of the Jacobi polynomials ${\bf P}^{(0,m)}$. 

A similar calculation reveals that
\begin{align}
\begin{split}
&M^{\Omega_\rho}_{m,j} = \langle ({\bm \Phi}^{\Omega_\rho}_{m,j})^\top,  {\bm \Phi}^{\Omega_\rho}_{m,j} \rangle_{L^2(\Omega_\rho)}
=(R^{\Omega_\rho}_m)^\top \langle ({\bf Z}^{\rho, (0,0)}_{m,j})^\top,  {\bf Z}^{\rho, (0,0)}_{m,j} \rangle_{L^2(\Omega_\rho)}R^{\Omega_\rho}_m\\
&\indent =\left( \int_0^{2\pi} \sin^2(m\theta +j \pi/2)\, \mathrm{d}\theta \right) \\
&\indent \indent \times (R^{\Omega_\rho}_m)^\top  \langle ({\bf Q}^{t,(0,0,m)})^\top,  2^{-1} t^{-(m+1)} (t-\tau)^m {\bf Q}^{t,(0,0,m)} \rangle_{L^2(0,1)}R^{\Omega_\rho}_m\\
&\indent =\frac{\pi_m}{2t^{m+1}}(R^{\Omega_\rho}_m)^\top R^{\Omega_\rho}_m.
\end{split}
\end{align}
\end{proof}

\begin{remark}
$M^{\Omega_0}_{m,j}$ is tridiagonal and $M^{\Omega_\rho}_{m,j}$  has a $4 \times 4$ ``arrowhead'' followed by a pentadiagonal tail. Thus $M^{\Omega_\rho}_{m,j}$ is a $B^3$-Arrowhead matrix with block-bandwidths $(2,2)$ and sub-block-bandwidth $1$ \cite[Def.~4.1]{Olver2023}. Hence, the global mass matrix is block diagonal where each submatrix is a $B^3$-Arrowhead matrix with block-bandwidths $(2,2)$ and sub-block-bandwidth $1$.
\end{remark}

\begin{remark}
Mass matrix entries corresponding to the $L^2$-inner product of bubble functions centred on different cells are equal to zero since their supports have zero measurable overlap. 
\end{remark}

The following lemma concerning the Laplacian applied to harmonic polynomial will be required to compute the entries of the stiffness matrix.
\begin{lemma}
\label{lem:weighted-harmonic}
Let $\rho>0$ and recall that $Y_{m,j}(x,y) \coloneqq r^m \sin(m\theta + j \pi/2)$. Then the following holds:
\begin{align}
\Delta [(1-r^2) Y_{m,j}] = -\Delta [(r^2-\rho^2) Y_{m,j}] = -4(m+1) Y_{m,j}. \label{eq:harmonic}
\end{align}
\end{lemma}
\begin{proof}
\cref{eq:harmonic} follows by a direction calculation after applying the polar coordinate version of the Laplacian $\Delta = \partial^2_{rr} + \partial_r/r +\partial^2_{\theta\theta}/r^2$.
\end{proof}

\begin{theorem}[Stiffness matrix]
Consider the unit disk domain $\Omega_0$ and the annulus domain $\Omega_\rho$, $\rho > 0$. Recall the definition of the Laplacian matrices $D_m$ and $D^\rho_m$ in \cref{def:laplacian-matrices} and the normalization constants $p_{(a,b)}$ and $q_{t,(a,b,c)}$ from \cref{lem:normalization}.  Define the stiffness matrices $A^{\Omega_0} = \langle (\nabla {\bm \Phi}^{\Omega_0})^\top,  \nabla {\bm \Phi}^{\Omega_0} \rangle_{L^2(\Omega_0)}$ and $A^{\Omega_\rho}  = \langle (\nabla {\bm \Phi}^{\Omega_\rho})^\top,  \nabla {\bm \Phi}^{\Omega_\rho} \rangle_{L^2(\Omega_\rho)}$. Then both  $A^{\Omega_0}$ and $A^{\Omega_\rho}$ are block-diagonal where the blocks correspond to each Fourier mode:
\begin{align}
\begin{split}
A^{\Omega_0} = 
\begin{pmatrix}
A^{\Omega_0}_{0,1} & & &\\
&A^{\Omega_0}_{1,0}&&\\
&&A^{\Omega_0}_{1,1}&\\
&&&\ddots
\end{pmatrix}, \;\;
A^{\Omega_\rho} = 
\begin{pmatrix}
A^{\Omega_\rho}_{0,1} & & &\\
&A^{\Omega_\rho}_{1,0}&&\\
&&A^{\Omega_\rho}_{1,1}&\\
&&&\ddots
\end{pmatrix}
\end{split}
\label{eq:stiffness-matrix-block-diagonal}
\end{align}
where
\begin{align}
A^{\Omega_0}_{m,j} = \langle (\nabla {\bm \Phi}^{\Omega_0}_{m,j})^\top,  \nabla{\bm \Phi}^{\Omega_0}_{m,j} \rangle_{L^2(\Omega_0)} 
=
\begin{pmatrix}
  \frac{m \pi}{p_{(0,m)}} &  \hspace*{-25mm}  \begin{matrix} 0 & \cdots &  \end{matrix} \\
  \begin{matrix} 0 \\ \vdots \\ \end{matrix}  &
  \begin{matrix}
  \hspace*{-\arraycolsep}
  \phantom{x_{11}} & \phantom{x_{11}} & \phantom{x_{11}}
  \hspace*{-\arraycolsep}
  \\
  & \raisebox{-0.2\height}[0pt][0pt]{\LARGE$-\frac{\pi_m}{2^{m+3}} D_m$} & \\
  & &
  \end{matrix}
\end{pmatrix}, \label{eq:stiffness-disk}
\end{align}
and
\begin{align}
A^{\Omega_\rho}_{m,j} =  \langle (\nabla {\bm \Phi}^{\Omega_\rho}_{m,j})^\top,  \nabla{\bm \Phi}^{\Omega_\rho}_{m,j} \rangle_{L^2(\Omega_\rho)} 
=
\begin{pmatrix}
   \begin{matrix} a_m & b_m  \\ b_m & c_m & \end{matrix} 
   & \hspace*{-20mm} \begin{matrix} d_m & 0 & \cdots  \\ e_m & 0 & \cdots \end{matrix} \\
   \hspace*{-3mm}\begin{matrix}  d_m & e_m \\ 0 & 0 \\ \vdots & \vdots \end{matrix}  &
  \begin{matrix}
  \hspace*{-\arraycolsep}
  \phantom{x_{11}} & \phantom{x_{11}} & \phantom{x_{11}}
  \hspace*{-\arraycolsep}
  \\
  & \raisebox{-0.2\height}[0pt][0pt]{\hspace*{-10mm}\LARGE$-\frac{\pi_m}{2t^{m+3}} D^\rho_m$} & \\
  & &
  \end{matrix}
\end{pmatrix}, \label{eq:stiffness-annulus}
\end{align}
such that
\begin{align}
a_m &=  \pi_m \frac{2-\rho^{2m}(m(2+m) -2m(2+m)\rho^2 + (2-m(2+m))\rho^4)}{(2+m)q_{t,(1,0,m)}},\\
b_m &= -\frac{2\pi(1-\rho^{4+2m})}{(2+m)q_{t,(1,0,m)}^{1/2}q_{t,(0,1,m)}^{1/2}},\\
c_m &= \pi_m \frac{2-2\rho^{4+2m}+m(2+m)t^{-2}}{(2+m)q_{t,(0,1,m)}},\\
d_m &= 2\pi_m(m+1) q_{t,(1,1,m)}^{1/2} q_{t,(1,0,m)}^{-1/2} t^{-(m+3)},\\
e_m & = -2\pi_m(m+1) q_{t,(1,1,m)}^{1/2} q_{t,(0,1,m)}^{-1/2} t^{-(m+3)}.
\end{align}
Thus $A_{m,0} = A_{m,1}$.
\end{theorem}
\begin{proof}
If either $m \neq \mu$ or $j \neq \zeta$, then
\begin{align}
\begin{split}
\langle (\nabla {\bm \Phi}^{\Omega_0}_{m,j})^\top,  \nabla{\bm \Phi}^{\Omega_0}_{\mu,\zeta} \rangle_{L^2(\Omega_0)} = \langle (\nabla{\bm \Phi}^{\Omega_\rho}_{m,j})^\top, \nabla {\bm \Phi}^{\Omega_\rho}_{\mu,\zeta} \rangle_{L^2(\Omega_\rho)} &= {\bf 0},
\end{split}\label{eq:different-fourier-stiffness}
\end{align}
where $\vectt{0}$ denotes the infinite-dimensional matrix of zeroes. \cref{eq:different-fourier-stiffness} follows by substituting in the definitions of the basis functions and noting that $\int_0^{2\pi} \sin(m\theta + j \pi/2) \sin(\mu \theta + \zeta \pi/2) \, \mathrm{d}\theta = 0$ if either  $m \neq \mu$ or $j \neq \zeta$. In other words hat and bubble functions with different Fourier modes have a stiffness matrix entry of zero.  This implies that \cref{eq:stiffness-matrix-block-diagonal} holds.

First consider the disk cell.

\noindent \textbf{(Disk).} The (1,1) entry in \cref{eq:stiffness-disk} follows by a direct calculation:
\begin{align}
\begin{split}
\langle \nabla Z^{(0)}_{m,m,j}, \nabla Z^{(0)}_{m,m,j} \rangle_{L^2(\Omega_0)}
&= p_{(0,m)}^{-1} \iint_{\Omega_0} \nabla Y_{m,j}(x,y) \cdot \nabla Y_{m,j}(x,y) \, \dx\dy\\
&  = 2 \pi m^2 p_{(0,m)}^{-1} \int_{0}^1 r^{2m-1} \,\dr = m \pi p_{(0,m)}^{-1}.
\end{split}
\end{align}
The first equality follows from \cref{eq:zernike-def} and noting that $P^{(0,m)}_0(x) =  p_{(0,m)}^{-1/2}$. The second equality follows by utilizing the gradient in polar coordinates, i.e.~$\nabla = (\partial_r \;\; \partial_\theta/r)^\top$ and a change from Cartesian coordinates to polar coordinates.

The remainder of ${\bm \Phi}^{\Omega_0}_{m,j}$ consists of bubble functions which vanish on the boundary of the cell. Thus one may perform an integration by parts and the boundary term vanishes, i.e.
\begin{align}
\begin{split}
\langle (\nabla {\bf B}^{\Omega_0}_{m,j})^\top, \nabla {\bf B}^{\Omega_0}_{m,j} \rangle_{L^2(\Omega_0)} 
&= - \langle  ({\bf B}^{\Omega_0}_{m,j})^\top, \Delta {\bf B}^{\Omega_0}_{m,j} \rangle_{L^2(\Omega_0)} \\
&= -\langle (1-r^2) ({\bf Z}^{(1)}_{m,j})^\top, {\bf Z}^{(1)}_{m,j} \rangle_{L^2(\Omega_0)} D_m 
= - \frac{\pi_m}{2^{m+3}}  D_m.
\end{split}
\end{align}
The final equality follows from the orthogonality of ${\bf Z}^{(1)}_{m,j}$ where the normalization constant is calculated via the definition \cref{eq:zernike-def}. It remains to show that the off-diagonal entries are zero. This follows as:
\begin{align}
\begin{split}
\langle \nabla Z^{(0)}_{m,m,j} , \nabla {\bf B}^{\Omega_0}_{m,j} \rangle_{L^2(\Omega_0)} 
&= -\langle \Delta Z^{(0)}_{m,m,j}, {\bf B}^{\Omega_0}_{m,j} \rangle_{L^2(\Omega_0)} \\
&=- p_{(0,m)}^{-1/2} \langle \Delta Y_{m,j}, {\bf B}^{\Omega_0}_{m,j} \rangle_{L^2(\Omega_0)} = 0.
\end{split}
\end{align}
The final equality follows since, by definition, $\Delta Y_{m,j}(x,y) = 0$.

\noindent \textbf{(Annulus).}  We first consider the interaction with the bubble functions with themselves. As the bubble functions vanish on the boundary of the element, one may perform an integration by parts and the boundary term vanishes. Hence,
\begin{align}
\begin{split}
&\langle (\nabla {\bf B}^{\Omega_\rho}_{m,j})^\top, \nabla {\bf B}^{\Omega_\rho}_{m,j} \rangle_{L^2(\Omega_\rho)} 
 = - \langle  ({\bf B}^{\Omega_\rho}_{m,j})^\top, \Delta {\bf B}^{\Omega_\rho}_{m,j} \rangle_{L^2(\Omega_\rho)} \\
&\indent = -\langle (1-r^2)(r^2-\rho)^2({\bf Z}^{\rho,(1,1)}_{m,j})^\top, {\bf Z}^{\rho,(1,1)}_{m,j} \rangle_{L^2(\Omega_\rho)} D^\rho_m 
 = -  \frac{\pi_m}{2t^{m+3}}   D^\rho_m.
\end{split}
\end{align}
We now compute $d_m$ in the (1,3)-entry and the trailing vector of zeroes from the (1,4)-entry in the first row. Note that these entries correspond to
\begin{align}
\begin{split}
\langle \nabla [(1-r^2) Z^{\rho, (1,0)}_{m,m,j}], \nabla {\bf B}^{\Omega_\rho}_{m,j} \rangle_{L^2(\Omega_\rho)}
= -\langle \Delta [(1-r^2) Z^{\rho, (1,0)}_{m,m,j}], {\bf B}^{\Omega_\rho}_{m,j} \rangle_{L^2(\Omega_\rho)},
\end{split}
\end{align}
where the equality follows by an integration by parts. Then, by utilizing \cref{lem:weighted-harmonic} and \cref{def:2Dannuli}, we see that
\begin{align}
\begin{split}
&-\langle \Delta [(1-r^2) Z^{\rho, (1,0)}_{m,m,j}], {\bf B}^{\Omega_\rho}_{m,j} \rangle_{L^2(\Omega_\rho)}
 = 4(m+1)  \langle  Z^{\rho, (1,0)}_{m,m,j}, {\bf B}^{\Omega_\rho}_{m,j} \rangle_{L^2(\Omega_\rho)}\\
&\indent= 2\pi_m(m+1) t^{-(m+3)} \frac{q_{t,(1,1,m)}^{1/2}}{q_{t,(1,0,m)}^{1/2}} \\
& \indent \indent \times \int_0^1 \tau (1-\tau)(t-\tau)^m Q_0^{t,(1,1,m)}(\tau) {\bf Q}^{t,(1,1,m)}(\tau) \, \mathrm{d}\tau\\
&\indent= 2\pi_m(m+1) t^{-(m+3)} \frac{q_{t,(1,1,m)}^{1/2}}{q_{t,(1,0,m)}^{1/2}} \begin{pmatrix} 1 &0 & 0 & \cdots \end{pmatrix}.
\end{split}
\end{align}
Thus we recover the value of $d_m$ and the trailing zeroes. The value of $e_m$ and the subsequent trail of zeroes in the second row follow in an almost identical fashion. It remains to verify the values of $a_m$, $b_m$, and $c_m$. Note that by \cref{def:2Dannuli} and \cref{lem:weighted-harmonic}
\begin{align}
\begin{split}
a_m &= \langle \nabla [(1-r^2) Z^{\rho, (1,0)}_{m,m,j}], \nabla [(1-r^2) Z^{\rho, (1,0)}_{m,m,j}] \rangle_{L^2(\Omega_\rho)}\\
& =q_{t,(1,0,m)}^{-1}  \langle \nabla [(1-r^2) Y_{m,j}], \nabla [(1-r^2) Y_{m,j}] \rangle_{L^2(\Omega_\rho)}\\
&=q_{t,(1,0,m)}^{-1} \iint_{\Omega_\rho} m^2 (r^{2m-2} + r^{2m+2}) + (2m+1) r^{2m+2} \sin^2(m\theta + j \pi/2)\, \dx\dy\\
&=q_{t,(1,0,m)}^{-1} \left[2 \pi \int_{\rho}^1 m^2 (r^{2m-1} + r^{2m+3}) \, \dr + \pi_m \int_{\rho}^1 (2m+1) r^{2m+3} \, \dr \right]. \label{eq:am}
\end{split}
\end{align}
The third  equality in \cref{eq:am} followed by utilizing the gradient in polar coordinates, i.e.~$\nabla = (\partial_r \;\; \partial_\theta/r)^\top$. The value of $a_m$ follows by computing the final integral in \cref{eq:am}. The values of $b_m$ and $c_m$ follow similarly.
\end{proof}
\begin{remark}
$A^{\Omega_0}_{m,j}$ is diagonal and $A^{\Omega_\rho}_{m,j}$ has a $3 \times 3$ ``arrowhead'' followed by a tridiagonal tail. Thus $A^{\Omega_\rho}_{m,j}$ is a $B^3$-Arrowhead matrix with block-bandwidths $(1,1)$ and sub-block-bandwidth $1$ \cite[Def.~4.1]{Olver2023}. Hence, the global stiffness matrix is block diagonal where each submatrix is a $B^3$-Arrowhead matrix with block-bandwidths $(1,1)$ and sub-block-bandwidth $1$.
\end{remark}

\subsection{Variable Helmholtz coefficients}

In this subsection we demonstrate how to handle variable coefficients.
\subsubsection{Rotationally invariant coefficients}

\label{sec:assembly}
The hierarchical basis can discretize a rotationally invariant Helmholtz coefficient $\lambda(r^2)$ efficiently. Moreover, the sparsity of the induced weighted mass matrix $M_\lambda$ is correlated with the number of terms in a Chebyshev expansion required to resolve the coefficient.

We first consider the following lemma:
\begin{lemma}
\label{lem:jacobi-matrix-r2}
Consider the disk and annular cells $K_0 = \Omega_{0, \rho}$, $\rho>0$, $K_1 = \Omega_{\rho_1, \rho_2}$, $0 < \rho_1 < \rho_2$. Then
\begin{align}
r^2 {\bf Z}_{m,j}^{K_0, (a)}(x,y) &=\frac{\rho^2}{2} {\bf Z}_{m,j}^{K_0, (a)}(x,y) (I+X_{(a,m)}), \label{eq:r2-1}\\
r^2 {\bf Z}_{m,j}^{K_1,(a,b)}(x,y) &=  {\rho_2^2}{\bf Z}_{m,j}^{K_1,(a,b)}(x,y) (I-t^{-1}X_{t,(a,b,m)}), \label{eq:r2-2}
\end{align}
where $I$ is the identity matrix.
\end{lemma}
\begin{proof}
We prove  \cref{eq:r2-2} and note that  \cref{eq:r2-1} follows similarly. Consider $\tau = t(1-r^2)$. Then, by utilizing \cref{eq:tau},
\begin{align}
\begin{split}
r^2 {\bf Z}_{m,j}^{\rho,(a,b)}(x,y) 
&= r^2 Y_{m,j}(x,y) {\bf Q}^{t,(a,b,m)}(\tau) \\
&=    Y_{m,j}(x,y) (1-t^{-1}\tau) {\bf Q}^{t,(a,b,m)}(\tau) \\
&=    Y_{m,j}(x,y) {\bf Q}^{t,(a,b,m)}(\tau) (I - t^{-1} X_{t,(a,b,m)}) \\
&=   {\bf Z}_{m,j}^{\rho,(a,b)}(x,y)  (I - t^{-1} X_{t,(a,b,m)}).
\end{split}
\end{align}
Thus \cref{eq:r2-2} holds when $\rho_1 = \rho$ and $\rho_2 = 1$. The result follows for a general annular cell $K_1$ with a scaling argument.
\end{proof}

Leveraging \cref{lem:jacobi-matrix-r2} we now describe how one discretizes a rotationally invariant Helmholtz coefficient $\lambda(r^2)$. The Helmholtz coefficient is expanded over each cell in the mesh independently. 

\begin{definition}
\label{def:scaled-ChebyshevT}
We define $\{T_n\}_{n \in \mathbb{N}_0}$ as Chebyshev polynomials of the first kind \cite[Sec.~18.3]{dlmf} and $T^{[a,b]}_n(x) \coloneqq T_n((2x-a-b)/(b-a))$. In other words $\{T^{[a,b]}_n\}_{n \in \mathbb{N}_0}$ are Chebyshev polynomials of the first kind scaled to the interval $[a,b]$.
\end{definition}

\begin{theorem}[Weighted mass matrix: disk]
\label{th:assembly:disk}
Consider the unit disk domain $\Omega_0$ and let $\{T^{[0,1]}_n\}_{n \in \mathbb{N}_0}$ be the Chebyshev polynomials scaled to the interval $[0,1]$ as defined in \cref{def:scaled-ChebyshevT}.
\begin{enumerate}
\itemsep=0pt
\item Consider the expansion: $\lambda(r^2)|_{\Omega_0} = \sum_{n=0}^\infty \lambda_n T^{[0,1]}_n(r^2)$;
\item Let $\Lambda_m = \sum_{n=0}^\infty \lambda_n T^{[0,1]}_n((I+X_{(0,m)})/2)$.
\end{enumerate}
Then
\begin{align}
\langle ({\bm \Phi}^{\Omega_0}_{m,j})^\top , \lambda(r^2) {\bm \Phi}^{\Omega_0}_{m,j} \rangle_{L^2(\Omega_0)} 
= \frac{\pi_m}{2^{m+2}} (R^{\Omega_0}_m)^\top \Lambda_m R^{\Omega_0}_m.
\end{align}
\end{theorem}
\begin{proof}
Note that
\begin{align}
\begin{split}
&\langle ({\bm \Phi}^{\Omega_0}_{m,j})^\top , \lambda(r^2) {\bm \Phi}^{\Omega_0}_{m,j} \rangle_{L^2(\Omega_0)} 
= (R^{\Omega_0}_m)^\top\langle ({\bf Z}^{(0)}_{m,j})^\top , (\sum_{n=0}^\infty \lambda_n T^{[0,1]}_n(r^2)) {\bf Z}^{(0)}_{m,j} \rangle_{L^2(\Omega_0)} R^{\Omega_0}_m\\
&\indent =(R^{\Omega_0}_m)^\top \langle ({\bf Z}^{(0)}_{m,j})^\top  {\bf Z}^{(0)}_{m,j} \rangle_{L^2(\Omega_0)} \Lambda_m R^{\Omega_0}_m = \frac{\pi_m}{2^{m+2}} (R^{\Omega_0}_m)^\top \Lambda_m R^{\Omega_0}_m.
\end{split}
\end{align}
The first equality holds thanks to \cref{prop:raising} and the second equality holds thanks to \cref{lem:jacobi-matrix-r2}.
\end{proof}

\begin{theorem}[Weighted mass matrix: annulus]
Fix the annulus domain $\Omega_\rho$, $\rho>0$ and let $\{T^{[\rho,1]}_n\}_{n \in \mathbb{N}_0}$ be the Chebyshev polynomials scaled to the interval $[\rho,1]$ as defined in \cref{def:scaled-ChebyshevT}.
\begin{enumerate}
\itemsep=0pt
\item Consider the expansion: $\lambda(r^2)|_{\Omega_\rho} = \sum_{n=0}^\infty \lambda_n T^{[\rho,1]}_n(r^2)$;
\item Let $\Lambda_m = \sum_{n=0}^\infty \lambda_n T^{[\rho,1]}_n(\rho^2 (I-t^{-1}X_{t,(0,0,m)}))$.
\end{enumerate}
Then
\begin{align}
\langle ({\bm \Phi}^{\Omega_\rho}_{m,j})^\top , \lambda(r^2) ({\bm \Phi}^{\Omega_\rho}_{m,j}) \rangle_{L^2(\Omega_\rho)} 
= \frac{\pi_m}{2t^{m+1}} (R^{\Omega_\rho}_m)^\top \Lambda_m R^{\Omega_0}_m.
\end{align}
\end{theorem}
\begin{proof}
The result follows analogously to the proof of \cref{th:assembly:disk}
\end{proof}

\begin{remark}
In practice we truncate the scaled Chebyshev expansion at some degree $N$ and utilize the Clenshaw algorithm to compute the matrices $\Lambda_m$. 
\end{remark}

\subsubsection{Rotationally anisotropic coefficients}
\label{sec:non-separable}

If $\lambda$ is rotationally anisotropic (it cannot be written as $\lambda(r)$) but it can be represented by a piecewise low-order polynomial in Cartesian coordinates, we may still recover a sparse discretization. However, a rotationally anisotropic coefficient means the PDE operator is now non-separable. In other words, the operator does not decouple across Fourier modes and thus the resulting linear system is not block diagonal.

We confine the discussion to discretizing $\lambda(x,y) = x$ on an annular cell. Discretizing on a disk cell follows similarly. For a more complex coefficient, one expands $\lambda(x,y)$ in a tensor-product Chebyshev expansion in $x$ and $y$ and the linear system is assembled akin to \cref{sec:assembly}.

We first partially quote a proposition from \cite{Papadopoulos2023}.

\begin{proposition}[Proposition 4.8 in \cite{Papadopoulos2023}, $x$-Jacobi matrix]
\label{prop:jacobi-matrix-x}
Let $t = (1-\rho^2)^{-1}$ and $R^{t,(a,b,c+1)}_{(a,b,c)}$ denote the raising matrix for the orthonormal semiclassical Jacobi families \newline ${\bf Q}^{t,(a,b,c)} = {\bf Q}^{t,(a,b,c+1)}  R^{t,(a,b,c+1)}_{(a,b,c)}$. Then
\begin{align}
x {\bf Z}_{0,1}^{\rho,(a,b)}(x,y) &= {\bf Z}_{1,1}^{\rho,(a,b)}(x,y) R^{t,(a,b,1)}_{(a,b,0)},\\
x {\bf Z}_{1,0}^{\rho,(a,b)}(x,y) &= \frac{1}{2} {\bf Z}_{2,0}^{\rho,(a,b)}(x,y) R^{t,(a,b,2)}_{(a,b,1)},
\end{align}
and for, $m \geq 1$, $(m,j) \neq (1,0)$,
\begin{align}
x {\bf Z}_{m,j}^{\rho,(a,b)}(x,y) =  \frac{1}{2} \left[t^{-1} {\bf Z}_{m-1,j}^{\rho,(a,b)}(x,y) (R^{t,(a,b,m)}_{(a,b,m-1)})^\top + {\bf Z}_{m+1,j}^{\rho,(a,b)}(x,y) R^{t,(a,b,m+1)}_{(a,b,m)}\right].
\label{eq:x1}
\end{align}
\end{proposition}

Inspired by \cref{prop:jacobi-matrix-x}, we define the $x$-Jacobi matrix $X^{\Omega_\rho}$ as
\begin{align}
X^{\Omega_\rho}  =
\begin{pmatrix}
& & \frac{1}{2t} (R^{t,(0,0,1)}_{(0,0,0)})^\top & & \\
& & & \frac{1}{2t} (R^{t,(0,0,2)}_{(0,0,1)})^\top &\\
R^{t,(0,0,1)}_{(0,0,0)} & & & & \ddots \\
& \frac{1}{2} R^{t,(0,0,2)}_{(0,0,1)} & & &\\
& & \ddots & &
\end{pmatrix}
\end{align}
such that
\begin{align}
x 
\begin{pmatrix}
{\bf Z}^{\rho,(0,0)}_{0,1} & {\bf Z}^{\rho,(0,0)}_{1,0} & {\bf Z}^{\rho,(0,0)}_{1,1} & \cdots
\end{pmatrix}
=
\begin{pmatrix}
{\bf Z}^{\rho,(0,0)}_{0,1} & {\bf Z}^{\rho,(0,0)}_{1,0} & {\bf Z}^{\rho,(0,0)}_{1,1} & \cdots
\end{pmatrix}X^{\Omega_\rho}.
\end{align} 

\begin{proposition}
\label{prop:x-Jacobi}
Let $R^{\Omega_\rho}_m$ denote the raising operators defined in \cref{prop:raising}. Define
\begin{align}
R^{\Omega_\rho}
\coloneqq \begin{pmatrix}
R^{\Omega_\rho}_0 & & \\
& R^{\Omega_\rho}_1  &\\
& &  \ddots
\end{pmatrix},
\;\;
D^{\Omega_\rho} \coloneqq 
\begin{pmatrix}
\mathrm{diag}(\frac{\pi_m}{2^2}) & &\\
& \mathrm{diag}(\frac{\pi_m}{2^4})   &\\
& & \ddots
\end{pmatrix},
\end{align}
such that $\Phi^{\Omega_\rho}(x,y)= {\bf Z}^{\rho, (0,0)}(x,y) R^{\Omega_\rho}$ and $\langle ({\bf Z}^{\rho, (0,0)})^\top, {\bf Z}^{\rho, (0,0)} \rangle_{L^2(\Omega_\rho)} = D^{\Omega_\rho}$. Then 
\begin{align}
\langle (\Phi^{\Omega_\rho})^\top, x \Phi^{\Omega_\rho} \rangle_{L^2(\Omega_\rho)}
= (R^{\Omega_\rho})^\top D^{\Omega_\rho} X^{\Omega_\rho} R^{\Omega_\rho}.
\end{align}
\end{proposition}
\begin{proof}
We note that
\begin{align}
\langle (\Phi^{\Omega_\rho})^\top, x \Phi^{\Omega_\rho} \rangle_{L^2(\Omega_\rho)}
&= (R^{\Omega_\rho})^\top \langle  ({\bf Z}^{\rho, (0,0)})^\top, x  {\bf Z}^{\rho, (0,0)} \rangle_{L^2(\Omega_\rho)} R^{\Omega_\rho}\\
&=(R^{\Omega_\rho})^\top \langle  ({\bf Z}^{\rho, (0,0)})^\top,  {\bf Z}^{\rho, (0,0)} \rangle_{L^2(\Omega_\rho)} X^{\Omega_\rho} R^{\Omega_\rho}\\
&= (R^{\Omega_\rho})^\top D^{\Omega_\rho} X^{\Omega_\rho} R^{\Omega_\rho}.
\end{align}
\end{proof}
The result in \cref{prop:x-Jacobi} then allows us to assemble the FEM matrices associated with anisotropic coefficients as further demonstrated in \cref{sec:examples:non-separable}.

\subsection{The load vector}
\label{sec:load-vector}
In order to construct the load vector ${\bf b}$ in \cref{eq:helmholtz-LA}, we are required to test the data against a test function $\langle f,v\rangle_{L^2(\Omega)}$ for all $v \in H^1_0(\Omega)$ where the test function has been expanded in the basis $\Phi$. We choose to expand the right-hand side in the discontinuous Zernike (annular) polynomial basis orthogonal with respect to the Lebesgue measure,  which can be achieved in quasi-optimal complexity.  More precisely,
consider the mesh $\mathcal{T}_h = \{\bar K_j\}_{j=0}^{N_h-1}$ where $K_j = \Omega_{\rho_j, \rho_{j+1}}$, $0 = \rho_0 < \rho_1 < \cdots < \rho_{N_h}$ and the quasimatrix:
\begin{align}
{\bm \Psi}^{\mathcal{T}_h}_{m,j} \coloneqq
\begin{pmatrix}
{\bf Z}^{K_0, (0)}_{m,j} & {\bf Z}^{K_1, (0,0)}_{m,j} & \cdots &{\bf Z}^{K_{N_h-1}, (0,0)}_{m,j}
\end{pmatrix},
\end{align}
such that  ${\bm \Psi}^{\mathcal{T}_h} \coloneqq \begin{pmatrix} {\bm \Psi}^{\mathcal{T}_h}_{0,1} & {\bm \Psi}^{\mathcal{T}_h}_{1,0} & \cdots \end{pmatrix}$. Then we expand $f(x,y)$ in \cref{eq:helmholtz} as $f(x,y) = {\bm \Psi}^{\mathcal{T}_h}(x,y) {\bf f}$ and fix ${\bf b} = \langle ({\bm \Phi}^{\mathcal{T}_h})^\top, {\bm \Psi}^{\mathcal{T}_h}\rangle_{L^2(\Omega)} {\bf f}$. The entries in the matrix $G_{\Phi, \Psi} \coloneqq \langle ({\bm \Phi}^{\mathcal{T}_h})^\top, {\bm \Psi}^{\mathcal{T}_h}\rangle_{L^2(\Omega)}$ may be computed thanks to the following proposition.

\begin{proposition}[Load vector]
Consider the unit disk domain $\Omega_0$ and the annulus domain $\Omega_\rho$, $\rho > 0$. Entries in $G_{\Phi, \Psi}$ associated with basis functions on different cells in the mesh are zero. Similarly, if either $m \neq \mu$ or $j \neq \zeta$, then
\begin{align}
\begin{split}
\langle ({\bm \Phi}^{\Omega_0}_{m,j})^\top,  {\bf Z}^{\Omega_0, (0)}_{\mu,\zeta} \rangle_{L^2(\Omega_0)} = \langle ({\bm \Phi}^{\Omega_\rho}_{m,j})^\top,  {\bf Z}^{\Omega_\rho, (0,0)}_{\mu,\zeta} \rangle_{L^2(\Omega_\rho)} &= {\bf 0},
\end{split}\label{eq:different-fourier:load-vector}
\end{align}
where $\vectt{0}$ denotes the infinite-dimensional matrix of zeroes. Moreover,
\begin{align}
\langle ({\bm \Phi}^{\Omega_0}_{m,j})^\top,  {\bf Z}^{\Omega_0, (0)}_{m,j} \rangle_{L^2(\Omega_0)} &= \frac{\pi_m}{2^{m+2}}(R^{\Omega_0}_m)^\top, \label{eq:load-vector-disk}\\
\langle ({\bm \Phi}^{\Omega_\rho}_{m,j})^\top,  {\bf Z}^{\Omega_\rho, (0,0)}_{m,j} \rangle_{L^2(\Omega_\rho)} &=  \frac{\pi_m}{2t^{m+1}} (R^{\Omega_\rho}_m)^\top.
\end{align}
\end{proposition}
\begin{proof}
The proof of this result is very similar to the proof of \cref{th:mass}.
\end{proof}

\begin{remark}
A more na\"ive approach for computing the load vector expands $f(x, y)$ directly in the continuous hierarchical FEM basis $\Phi$ and does a matrix-vector product with the mass matrices in \cref{th:mass}. However, if $f(x,y)$ has a radial discontinuity, then an expansion in a continuous basis will result in a poor approximation. Moreover, as discussed in \cref{sec:synthesis}, the analysis operator for the basis $\Phi$ is increasingly ill-conditioned as $m \to \infty$. 
\end{remark}

\subsection{Truncation}
\label{sec:truncation}
Until this section, the discussion revolved around computing operators which are infinite-dimensional. Hence \cref{eq:helmholtz} was being approximated exactly. However, in order to ensure the computations are numerically tractable, we must discretize the operators by truncating the stiffness and (weighted) mass matrices as well as the load vector. Consider a mesh $\mathcal{T}_h$ for the disk domain $\Omega$. The continuous hierarchical basis quasimatrix ${\bm \Phi}^{\mathcal{T}_h}$ consists of the Fourier mode restrictions such that
\begin{align}
{\bm \Phi}^{\mathcal{T}_h} \coloneqq 
\begin{pmatrix}
{\bm \Phi}^{\mathcal{T}_h} _{0,1} & {\bm \Phi}^{\mathcal{T}_h}_{1,0} & {\bm \Phi}^{\mathcal{T}_h}_{1,1} & \cdots 
\end{pmatrix},
\end{align}
where ${\bm \Phi}^{\mathcal{T}_h} _{m,j}$ is as defined in \cref{def:Phi-mj}. We denote a \emph{truncation} of the quasimatrix at degree $N_p$ on each element by ${\bm \Phi}^{\mathcal{T}_h, N_p}$. Suppose that $N_p$ is even. Then
\begin{align}
{\bm \Phi}^{\mathcal{T}_h, N_p} \coloneqq 
\begin{pmatrix}
{\bm \Phi}^{\mathcal{T}_h, N_p} _{0,1} & {\bm \Phi}^{\mathcal{T}_h, N_p}_{1,0} & {\bm \Phi}^{\mathcal{T}_h, N_p}_{1,1} & \cdots & {\bm \Phi}^{\mathcal{T}_h, N_p}_{N_p,0} & {\bm \Phi}^{\mathcal{T}_h, N_p}_{N_p,1} 
\end{pmatrix},
\end{align}
where ${\bm \Phi}^{\mathcal{T}_h, N_p}_{m,j}$ denotes the finite-dimensional quasimatrix with any polynomial with degree $n>N_p$ removed. Thus ${\bm \Phi}^{\mathcal{T}_h, N_p}_{N_p,j}$ and ${\bm \Phi}^{\mathcal{T}_h, N_p}_{N_p-1,j}$ have one entry and ${\bm \Phi}^{\mathcal{T}_h, N_p}_{N_p-2,j}$ and ${\bm \Phi}^{\mathcal{T}_h, N_p}_{N_p-3,j}$ have two entries. Moreover, for $m < N_p -3$, then ${\bm \Phi}^{\mathcal{T}_h, N_p}_{m,j}$ has  $N_h (N_p-m)/2+2$ entries if $m$ is even and $N_h (N_p-3-m)/2+2$ entries if $m$ is odd.

The truncated (weighted) mass and stiffness matrices are subsequently formed from the truncated quasimatrices. For instance 
\begin{align}
M^{\Phi, N_p} = \langle (\Phi^{\mathcal{T}_h, N_p})^\top, \Phi^{\mathcal{T}_h, N_p} \rangle_{L^2(\Omega)}
= \begin{pmatrix}
M^{\Phi, N_p}_{0,1} &   & &\\
 & M^{\Phi, N_p}_{1,0} && \\
&  &  \ddots & \\
&  & &   M^{\Phi, N_p}_{N_p,1}
\end{pmatrix},
\end{align}
such that $M^{\Phi, N_p} \in  \mathbb{R}^{N \times N}$ where $N = \frac{1}{2} (N_h N_p + 2)(N_p-1) + 2$.

\subsection{Synthesis}
\label{sec:synthesis}
In order to inspect and plot the solutions, one requires a fast synthesis (evaluation) operator, i.e.~given a vector ${\bf u}$, we wish to evaluate ${\bm \Phi}(x_i,y_j){\bf u}$ for a (non-Cartesian) grid of points $\{(x_i,y_j)\}_{ij}$ in quasi-optimal complexity $\mathcal{O}(N_h N_p^2 \log N_p)$. We focus on applying the synthesis operator on a single annular cell $\Omega_\rho$, but we note that this technique generalizes to multiple cells. In particular, the synthesis operator may be applied on each cell in the mesh independently and in parallel.

Suppose we wish to evaluate an expansion of $u(x,y) = \Phi^{\Omega_\rho}(x,y) {\bf u}$ on the Zernike annular grid. Then by leveraging \cref{prop:raising}, we may construct a raising operator matrix $R^{\Omega_\rho}$ (consisting of the Fourier mode submatrices $R^{\Omega_\rho}_m$) such that
\begin{align}
u(x,y) = {\bm \Phi}^{\Omega_\rho}(x,y) {\bf u} = {\bf Z}^{\rho, (0,0)}(x,y) R^{\Omega_\rho} {\bf u} 
= {\bf Z}^{\rho, (0,0)}(x,y) \tilde{\bf u}.
\end{align}
Then one applies the quasi-optimal synthesis operator that exists for ${\bf Z}^{\rho, (0,0)}$ with the coefficient vector $\tilde{\bf u}$ \cite{Papadopoulos2023, Gutleb2023}. A synthesis operator for the disk cell follows analogously.

\begin{remark}
The analysis operator is the reverse of the synthesis operator. Given a known function $f(x,y)$ on $\Omega_\rho$, the goal is to find the coefficient vector ${\bf f}$, in quasi-optimal complexity, such that $f(x,y) = {\bm \Phi}^{\Omega_\rho}(x,y) {\bf f}$. The analysis operator may be deduced by first expanding $f(x,y) = {\bf Z}^{\rho, (0,0)}(x,y) \tilde{\bf f}$ and then inverting the raising operator matrix ${\bf R}^{\Omega_\rho}_m$ for each Fourier mode to deduce the corresponding coefficient vector for ${\bm \Phi}^{\Omega_\rho}_{m,j}$. However, we note that a square truncation ${\bf R}^{\Omega_\rho}_m$ becomes  increasingly ill-conditioned as $m \to \infty$. Hence, we advise against directly expanding a known function in the continuous hierarchical basis. Indeed, this is typically never required. One only ever expands the right-hand side in the well-conditioned ${\bf Z}^{\rho, (0,0)}(x,y)$ basis and computes the load vector as in \cref{sec:load-vector}.
\end{remark}

\subsection{Optimal complexity factorization and solves}
\label{sec:factorization}
As shown in \cref{sec:mass-and-stiffness} the mass, $M$, and stiffness, $A$, matrices are block diagonal where each submatrix on the diagonal corresponds to a different Fourier mode of the basis. Moreover, the sparsity pattern of the submatrices are that of a $B^3$-Arrowhead matrix \cite[Def.~4.1]{Olver2023}. Symmetric positive-definite $B^3$-Arrowhead matrices admit a reverse Cholesky factorization (a Cholesky factorization initialized from the bottom right corner of the matrix rather than the top left) with sparse factors with zero fill-in \cite[Cor.~4.3]{Olver2023}. Thanks to the sparsity of the reverse Cholesky factors, they may be computed and inverted in optimal linear complexity $\mathcal{O}(N_h (N_p - m))$ where the Fourier mode submatrix has size $\lfloor N_h (N_p - m)/2 \rfloor \times \lfloor N_h (N_p - m)/2 \rfloor$. The linear complexity Cholesky factorization may be applied to any positive-definite addition of the mass and stiffness matrix, e.g.~$A+ \kappa^2 M$ for any $\kappa \in \mathbb{R}$. For more details on the optimal complexity reverse Cholesky factorization for the matrices that arise here, we refer the reader to \cite{Olver2023}.

The linear system in \cref{eq:helmholtz-LA} may be indefinite for more general choices of $\lambda$. In the examples considered in \cref{sec:examples}, we have reported some success with a UL factorization (with no pivoting) and the factors remain sparse and computable in linear complexity. We do not expect this to be stable for a general choice of the Helmholtz coefficient $\lambda$. A stable alternative is a QL factorization which, for fixed $N_h$, achieves linear complexity as $N_p \to \infty$ but  it is currently unknown how to achieve linear complexity   as $N_h \to \infty$  due to fill-in.

For the factorizations to accurately capture the inverses of the matrices in the linear systems, we require each submatrix on the block-diagonal to be sufficiently well-conditioned. The conditioning of $A$ and $M$ is largely influenced by the continuity coefficients $\kappa_m$ and $\gamma_m$ in Definitions \labelcref{def:hat1} and \labelcref{def:hat2} for the hat functions. These coefficients degrade the conditioning as $m \to \infty$. Consider a mesh $\mathcal{T}_h = \{\bar K_j\}_{j=0}^{N_h-1}$ where $K_j = \Omega_{\rho_j, \rho_{j+1}}$, $0 = \rho_0 < \rho_1 < \cdots < \rho_{N_h}$. Let $\mathrm{ratio}(\mathcal{T}_h) = \min_{j \in \{1,\dots,N_h-1\}} \rho_j / \rho_{j+1}$. Then the smallest singular value of the submatrices is roughly $\mathcal{O}(\mathrm{ratio}(\mathcal{T}_h)^m)$. $\mathrm{ratio}(\mathcal{T}_h)$ is a measure of the thickness of the annuli cells in the mesh. The smaller the diameter of the cells, the closer the value of $\mathrm{ratio}(\mathcal{T}_h)$ is to one and the better the conditioning. Thus the conditioning is improved by \emph{increasing} the number of cells in the mesh. Consider a maximum truncation degree of $N_p$. In practice, provided one picks a mesh such that $\text{ratio}(\mathcal{T}_h)^{N_p} > 10^{-8}$ (double 64-bit precision), then utilizing a robust factorization (such as reverse Cholesky or UL) means that the ill-conditioning will not cause numerical pollution in the solutions.

\begin{remark}[Static condensation]
The hierarchical ordering of the basis functions, together with the sparsity structure of the mass and stiffness matrices, mean that the induced Helmholtz linear systems are amenable to preconditioning via static condensation. Essentially one constructs the Schur complement induced by considering the top left block consisting of the rows and columns associated with the hat functions. Then the hat and bubble degrees of freedom may be solved for independently. For more details we refer the reader to \cite[Ch.~3.2]{Schwab1998}.
\end{remark}

\section[3D cylindrical domains]{Quasi-optimal solves in 3D cylindrical domains}
\label{sec:ADI}

We construct a hierarchical FEM basis for a three-dimensional cylindrical domain  with multiple elements not just in the $x,y$-axis but also the $z$-axis. This is achieved  by considering the tensor-product space of the continuous FEM hierarchical basis for the disk with the univariate continuous hierarchical $p$-FEM basis for the interval as defined in \cite[Sec.~2.1]{Olver2023}, see also \cite[Ch.~2.5.2]{szabo2011introduction} and \cite[Ch.~3.1]{Schwab1998}. Moreover, if the equation considered is the screened Poisson equation, then we prove that there exists a quasi-optimal $\mathcal{O}(N_h N_p^3 \log (N_h^{1/4} N_p))$ complexity solve. The setup complexity is $\mathcal{O}(N_h N_p^3\log^2 N_p)$. To clarify, here we use $N_p$ to denote the truncation degree of each polynomial basis factor of the tensor-product space, such that the tensor-product basis contains polynomials of maximum degree $2N_p$. $N_h$ denotes the number of three-dimensional cells in the mesh. The total degrees of freedom is $\frac{1}{2} (N_h^{1/2}  N_p + 2)(N_p-1) + 2)((N_p+1)N_h^{1/2} -1) = \mathcal{O}(N_h N_p^3)$. 

Without a loss of generality, let ${\bm \Omega} = \Omega_0 \times (-1,1) \subset \mathbb{R}^3$ be a cylindrical domain. Let $\mathcal{T}_h$ denote a mesh for ${\bm \Omega}$ into cylindrical and tube cells as exemplified in \cref{fig:3d-mesh}.
\begin{figure}[h!]
\centering
\includegraphics[width =0.15 \textwidth]{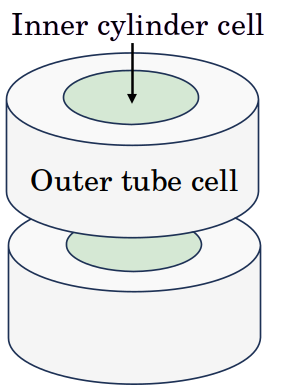}
\caption{A four cell mesh for the cylindrical domain ${\bm \Omega}$.}
\label{fig:3d-mesh}
\end{figure}
Fix a Helmholtz coefficient $\lambda(r) \geq 0$ a.e.~and consider the screened Poisson equation, find $u \in H^1_0(\Omega)$ that satisfies \cref{eq:helmholtz}. We pick the basis  $u(x,y,z) = {\bm \Phi}(x,y) U {\bf Q}(z)^\top$ where $U$ is the matrix of expansion coefficients and ${\bf Q}(z)$ is the quasimatrix of the univariate continuous hierarchical $p$-FEM basis consisting of weighted Jacobi polynomials and piecewise linear hat functions defined in \cite{Olver2023}. Note that ${\bf Q}(-1) = {\bf Q}(1) = {\bf 0}$. In quasimatrix notation, the three-dimensional screened Poisson equation may be rewritten as a generalized Sylvester equation, i.e.~find the coefficient matrix $U$ that satisfies:
\begin{align}
\left(M_\lambda^{\Phi} + A^\Phi \right) U M^Q + M^\Phi U A^Q  = G_{\Phi, \Psi} F G_{Q, P}^\top. \label{eq:adi:1}
\end{align}
Here $M^{\Phi}$, $M_\lambda^{\Phi}$, and $A^\Phi$ are the mass, weighted mass, and stiffness matrices, respectively, for the hierarchical basis $\Phi$ on the disk. Whereas $M^Q$ and $A^Q$ are the mass and stiffness matrices of the univariate basis. $F$ denotes the matrix of expansion coefficients for the right-hand side $f(x,y,z) = {\bm \Psi}(x,y) F {\bf P}(z)^\top$ where ${\bf P}(z)$ denotes the quasimatrix of the discontinuous hierarchical basis (consisting of piecewise Legendre polynomials) \cite{Olver2023}. Moreover, $G_{\Phi,\Psi} = \langle{\bm \Phi}^\top, {\bm \Psi}\rangle_{L^2(\Omega_0)}$ and $G_{Q,P} = \langle {\bf Q}, {\bf P}\rangle_{L^2(-1,1)}$. We note that  $A^Q$, $M^Q$, and $G_{Q,P}$ are sparse \cite{Olver2023}. Hence all the operator matrices appearing in \cref{eq:adi:1} are sparse.  We now truncate the tensor-product factor bases to degree $N_p$ to recover the finite-dimensional matrices $M_\lambda^{\Phi, N_p}$, $M^{\Phi, N_p}$, $A^{\Phi, N_p}$, $A^{Q, N_p}$, $M^{Q,N_p}$, $G^{N_p}_{\Phi, \Psi}$, and $G^{N_p}_{Q,P}$. For the remainder of this subsection, we drop the superscript $N_p$ for readability.

The ADI algorithm is an iterative algorithm for finding the matrix $X$ that solves the Sylvester equation $AX - XB = F$ \cite{Fortunato2020}.  It requires the two following assumptions to hold:
\begin{enumerate}[label={P}\arabic*.]
\itemsep=0pt
\item $A$ and $B$ are symmetric matrices; \label{as:P1}
\item There exist real disjoint nonempty intervals $[\mu_a, \mu_b]$ and $[\mu_c, \mu_d]$ such that $\sigma(A) \subset [\mu_a, \mu_b]$ and $\sigma(B) \subset [\mu_c, \mu_d]$, where $\sigma$ denotes the spectrum of a matrix. \label{as:P2}
\end{enumerate}
The algorithm proceeds iteratively. First one fixes the initial matrix $X_0 = 0$. Then, iteratively for $\ell \in \{1,2,\dots,\ell_{\mathrm{max}}\}$, we compute
\begin{align}
 X_{\ell-1/2} &= (F - (A-p_\ell I) X_{\ell-1})(B - p_\ell I)^{-1}  \label{eq:adi:7},\\ 
X_{\ell} &= (A-q_\ell I)^{-1}(F - X_{\ell-1/2}(B-q_\ell I)), \label{eq:adi:8}
\end{align}
where the value of the final iterate is $\ell_{\mathrm{max}} = \lceil \log(16 \gamma) \log(4/\epsilon)/\pi^2 \rceil$ with $\gamma = |\mu_c-\mu_a||\mu_d-\mu_b|/(|\mu_c-\mu_b||\mu_d-\mu_a|)$. The ADI shifts $p_\ell$ and $q_\ell$ have explicit formulae depending on $\gamma$ \cite[Eq.~(2.4)]{Fortunato2020}. Notably, we have that $p_\ell > 0$ and $q_\ell < 0$ for all $\ell \in \{1,\dots,\ell_{\mathrm{max}}\}$.

The first step to obtaining quasi-optimal solves is to notice that, due to the block diagonal (Fourier mode decoupling) nature of $M^\Phi$ and $M^\Phi_\lambda + A^\Phi$, \cref{eq:adi:1} also admits a Fourier mode decoupling. Hence, for $m \in \{0,1,\dots,N_p\}$, $j \in \{0,1\}$, $(m,j) \neq (0,0)$, we instead consider the equations:
\begin{align}
K^\Phi_{m,j} U_{m,j} M^Q + M^\Phi_{m,j} U_{m,j} A^Q  = H_{m,j}, \label{eq:adi:5}
\end{align}
where $K^\Phi_{m,j}  \coloneqq [M_\lambda^{\Phi}]_{m,j} + A^\Phi_{m,j}$ and $H_{m,j} \coloneqq G_{\Phi_{m,j}, \Psi_{m,j}} F_{m,j} G_{Q, P}^\top$. The following theorem reveals how to leverage the ADI algorithm to solve \cref{eq:adi:5} for each Fourier mode.

\begin{theorem}
\label{th:ADI}
Consider the cylindrical domain ${\bm \Omega} = \Omega_0 \times (-1,1)$ and a mesh $\mathcal{T}_h$ as depicted in \cref{fig:3d-mesh}. Consider the reverse Cholesky factorizations of the truncated matrices (not indicated) $L^\top L = A^Q$ and $V^\top V = K^\Phi_{m,j}$.

Suppose that $\sigma(L^{-\top} M^Q L^{-1}) \subset [\mu_a, \mu_b]$ and $\sigma(-V^{-\top} M^\Phi_{m,j} V^{-1}) \subset [\mu_c, \mu_d]$. Pick an ADI tolerance $\epsilon$. Let $\gamma = |\mu_c-\mu_a||\mu_d-\mu_b|/(|\mu_c-\mu_b||\mu_d-\mu_a|)$ and fix $\ell_{\mathrm{max}} = \lceil \log(16 \gamma) \log(4/\epsilon)/\pi^2 \rceil$. Assign $W_0 = {\bf 0}$ and for $\ell \in \{1,2,\dots,\ell_{\mathrm{max}}\}$ compute:
\begin{align}
W_{\ell-1/2}  &= [H_{m,j} - (M_{m,j}^\Phi - p_\ell K^\Phi_{m,j})W_{\ell-1}](-M^Q - p_\ell A^Q)^{-1} \label{eq:adi2:3}\\ 
W_{\ell} &= (M^\Phi_{m,j} - q_\ell K^\Phi_{m,j})^{-1} [ H_{m,j} - W_{\ell-1/2} (-M^Q - q_\ell A^Q)], \label{eq:adi2:4}
\end{align}
Then $U_{m,j} \approx  W_{\ell_\mathrm{max}} (A^Q)^{-1}$. More precisely,
\begin{align}
\| V(U_{m,j} - W_{\ell_\mathrm{max}}  (A^Q)^{-1}) L^\top \|_2 \leq \epsilon \|V U_{m,j} L^\top \|_2.
\label{eq:adi2:9}
\end{align}
\end{theorem}
\begin{proof}
The result follows by a direct application of Lemma 5.1 in \cite{Olver2023}.
\end{proof}

Thanks to \cref{th:ADI}, we may compute $U_{m,j}$ via the ADI algorithm. The remainder of this sections focuses on showing that the solution may be computed in quasi-optimal $\mathcal{O}(N_h N_p^3 \log (N_h^{1/2} N_p) \log \epsilon^{-1})$ flops. The first step is to show that asymptotically $\ell_{\text{max}} = \mathcal{O}(\log (N_h^{1/2} N_p) \log \epsilon^{-1})$.

\begin{lemma}[Inverse inequality]
\label{lem:inverse}
Consider a disk or annulus domain $\Omega_{a,b}$, $0 \leq a < b$. Suppose that $\pi_p$ denotes a degree $p$ multivariate polynomial. Let $h = b-a$. Then there exists a $c > 0$ such that the following inverse inequality holds:
\begin{align}
\| \nabla \pi_p \|_{L^2(\Omega_{a,b})} \leq c h^{-1} p^2 \| \pi_p \|_{L^2(\Omega_{a,b})}. 
\label{eq:inverse}
\end{align}
\end{lemma}
\begin{proof}
The case of a disk domain, $a=0$, was shown in \cite[Ex.~4.25]{Cangiani2022}. Then the extension to the annulus domain follows by leveraging a rescaling argument in the radial direction akin to the techniques introduced in \cite[Ch.~4.5]{Brenner2008}.
\end{proof} 

\begin{proposition}[Spectrum]
\label{prop:eigs}
Suppose the conditions of \cref{th:ADI} hold and $\lambda \in L^\infty(\Omega)$ with $\lambda(r) \geq 0$ a.e.~in $\Omega$. Suppose the mesh is quasi-uniform such that $N_h \sim h^{-2}$, where $h$ denotes the minimum diameter of the cells in the mesh \cite[Def.~4.4.13]{Brenner2008}. Consider the reverse Cholesky factorizations of the truncated matrices (not indicated) $L^\top L = A^Q$ and $V^\top V = K^\Phi_{m,j}$. Then there exist constants $0 < c \leq C < \infty$, independent of $N_p$ and $N_h$, such that
\begin{align}
\sigma(L^{-\top} M^Q L^{-1}) &\subseteq \left[\frac{1}{12 N_h N_p^{4}}, \frac{4}{\pi^2} \right], \label{spec:1}\\ 
\sigma(V^{-\top} M^\Phi_{m,j} V^{-1}) &\subseteq [c (N_h N_p^{4} + \|\lambda\|_{L^\infty(\Omega)})^{-1}, C]. \label{spec:2}
\end{align}
It follows that asymptotically $\ell_{\mathrm{max}} = \mathcal{O}(\log (N_h^{1/4} N_p) \log \epsilon^{-1})$ where $\ell_{\mathrm{max}}$ is the value of final iterate in the ADI algorithm and $\epsilon$ is the ADI tolerance.
\end{proposition}
\begin{proof}
\cref{spec:1} was shown to hold in \cite[Lem.~5.2]{Olver2023} and \cref{spec:2} may be derived by utilizing \cref{lem:inverse} and following the proof of \cite[Lem.~5.2]{Olver2023} with some small modifications.
\end{proof}

\begin{corollary}
Suppose the conditions of \cref{prop:eigs} hold. Then \cref{eq:adi:1} may be solved with $\mathcal{O}(N_h N_p^3 \log(N_h^{1/4}N_p) \log \epsilon^{-1})$ flops via the ADI algorithm, where $\epsilon$ is the ADI tolerance for each Fourier mode subsolve.
\end{corollary}
\begin{proof}
We first note that the eigenvalues $\mu_a$ and $\mu_b$  may be approximated in  $\mathcal{O}(N_h^{1/2}N_p)$ flops via an inverse iteration (with shifts). The same is true for $\mu_c$ and $\mu_d$ but in $\mathcal{O}(N_h^{1/2}N_p^2)$ flops. The result then follows by deconstructing each component of the setup and execution of the ADI algorithm in \cref{th:ADI}. The most expensive part of the algorithm is the $\ell_{\mathrm{max}}$ solves of \cref{eq:adi2:3} and \cref{eq:adi2:4} for each Fourier mode. Thanks to the optimal complexity reverse Cholesky factorization discussed in \cref{sec:factorization}, we have that each solve requires $\mathcal{O}(N_h N_p^2)$ flops.  \cref{prop:eigs} reveals that $\mathcal{O}(\ell_{\mathrm{max}}) = \mathcal{O}(\log (N_h^{1/4} N_p) \log \epsilon^{-1})$. Thus the $\ell_{\mathrm{max}}$ solves requires $\mathcal{O}(N_h N_p^2 \log (N_h^{1/4} N_p) \log \epsilon^{-1})$ flops. This must be conducted over $\mathcal{O}(N_p)$ Fourier modes resulting in an  \newline $\mathcal{O}(N_h N_p^3 \log (N_h^{1/4} N_p) \log \epsilon^{-1})$ complexity solve.
\end{proof}

\begin{figure}[h!]
\centering
\begin{minipage}{255mm}
 \begin{tikzpicture}[%
every node/.style={draw=black, thick, anchor=west},
grow via three points={one child at (-0.2,-0.9) and
    two children at (0.0,-0.9) and (0.0,-1.8)},
edge from parent path={(\tikzparentnode.220) |- (\tikzchildnode.west)}]
\node {Setup}
child { node {Compute entries of $A^Q$, $M^Q$, and $G_{Q,P}$: $\mathcal{O}(N_h N_p)$}}
child { node {Compute entries of $M^\Phi$, $K_\lambda^\Phi$, and $G_{\Phi, \Psi}$: $\mathcal{O}(N_h N_p^2)$}}
child { node {Compute $F$}
	child { node {Expansion in ${{\bm \Psi}}(x,y)$: $\mathcal{O}(N_h^{1/2} N_p^2 \log N_p)$ applied $\mathcal{O}(N_h^{1/2} N_p)$ times}}
	child { node {Expansion in $\mathbf{P}(z)$: $\mathcal{O}(N_h^{1/2} N_p \log^2 N_p)$  applied $\mathcal{O}(N_h^{1/2} N_p^2)$ times}} 
};
\hspace*{114mm}
\node {Solve}
child { node {Decompose across Fourier modes: $\mathcal{O}(N_p)$}
child { node {$\ell_{\text{max}}$ iterations of ADI: $\mathcal{O}(\log (N_h^{1/4} N_p \log \epsilon^{-1}))$}
child { node {Solving \cref{eq:adi2:3} and \cref{eq:adi2:4}:  $\mathcal{O}(N_h N_p^2)$}
}}
};
\end{tikzpicture}
\end{minipage}
\caption{Complexity flowchart of the setup and solve for the 3D screened Poisson equation of \cref{sec:examples:ADI}. The overall complexity of the setup is $\mathcal{O}(N_h N_p^3\log^2 N_p)$ and the solve is $\mathcal{O}(N_h N_p^3 \log (N_h^{1/4} N_p) \log \epsilon^{-1})$.}\label{fig:adi-complexity}
\end{figure}
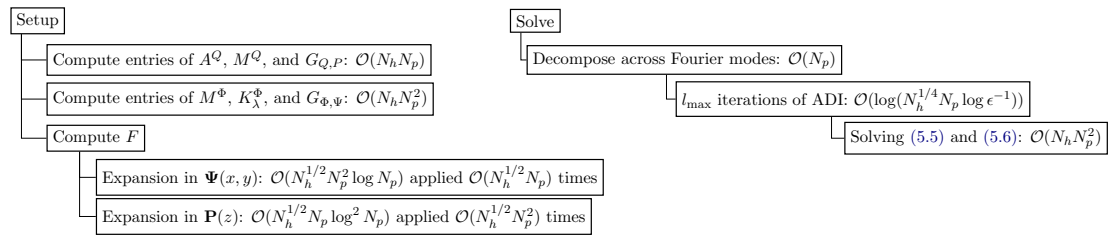

\section{Examples}
\label{sec:examples}

In  this section we utilize the hierarchical FEM basis introduced in this paper for disk and annulus domains to approximately solve a number of equations. When measuring the errors, we measure the $\ell^\infty$-norm on a heavily over-resolved Zernike (annular) grid on each cell in the mesh.

\textbf{Data availability:} The numerical experiments were conducted in Julia and heavily rely on a number of packages \cite{ClassicalPoly.jl2023,SemiPoly.jl2023,MulPoly.jl2023,RadialPoly.jl2023,PiecewisePoly.jl2023,FastTransforms.jl2023}.  For reproducibility, an implementation of the hierarchical basis as well as scripts to generate the plots and solutions can be found at SparseDiskFEM.jl \cite{SparseDiskFEM.jl2023}.  The version used in this paper is archived on Zenodo \cite{sparsediskfem.jl-zenodo}.

\subsection{Plane wave with discontinuous coefficients and data}
\label{sec:examples:plane-wave}

The first example we consider is a plane wave problem with radial discontinuities in both the right-hand side $f(x,y)$ and the Helmholtz coefficient $\lambda(r)$. Consider $\rho$, $\lambda_0$, $\lambda_1 > 0$ and define:
\begin{align}
\tilde{u}(r)
=
\frac{1}{4} \times \begin{cases}
 \lambda_0 r^2 + (\lambda_1 - \lambda_0) \rho^2 - \lambda_1 + 2(\lambda_0 - \lambda_1)\rho^2 \log(\rho)  & \text{if} \;\; 0\leq r \leq \rho, \\
\lambda_1 r^2 - \lambda_1 +2 (\lambda_0 - \lambda_1) \rho^2\log(r) & \text{if} \;\; \rho< r \leq 1.
\end{cases}
\label{eq:tildeu}
\end{align}
In this example we choose the Helmholtz coefficient:
\begin{align}
\lambda(r) = \Delta \tilde{u}(r)
=
\begin{cases}
\lambda_0 & \text{if} \;\; 0\leq r \leq \rho, \\
\lambda_1 & \text{if} \;\; \rho< r \leq 1.
\end{cases}
\end{align}
Note that $\tilde{u}(r)$ is twice differentiable (but the second derivative is not continuous) and $\tilde{u}(1) = 0$. Consider a unit disk domain $\Omega_0$ and fix the parameters as $\rho=1/2$, $\lambda_0 = 10^{-2}$, $\lambda_1 = 50$.  Let $u_e(x,y) = \sin(50x)\tilde{u}(r)$ and consider the right-hand side $f(x,y) = [-\frac{1}{50} \Delta + \lambda(r)] u_e(x,y)$. Our goal is to recover the exact known solution $u_e(x,y)$ by approximately finding $u \in H^1_0(\Omega_0)$ that satisfies the screened Poisson equation:
\begin{align}
\frac{1}{50} \langle \nabla v, \nabla u \rangle_{L^2(\Omega_0)} +  \langle  v, \lambda u \rangle_{L^2(\Omega_0)} = \langle v, f \rangle_{L^2(\Omega_0)}.
\label{eq:helmholtz-plane-wave}
\end{align}
We mesh the unit disk domain with $\mathcal{T}_h = \{ 0 \leq r \leq 1/2\} \cup \{ 2^{-(j+1)/9} \leq r \leq 2^{-j/9}\}_{j \in \{0, 1,\dots,9\}}$ culminating in $N_h = 10$ cells. We then compute the load vector from $f(x,y)$ as described in \cref{sec:load-vector} and compute the stiffness and weighted mass matrices as described in Sections \labelcref{sec:mass-and-stiffness} and \labelcref{sec:assembly}. The resultant matrix is block diagonal where the blocks correspond to the Fourier mode decoupling and have a $B^3$-Arrowhead matrix structure. Hence, we solve for each block individually via a reverse Cholesky factorization (as described in \cref{sec:factorization}) for an optimal complexity solve: $\mathcal{O}(N_h N_p^2)$.

We plot the right-hand side $f$ and the approximated solution $u$ of \cref{eq:helmholtz-plane-wave} on the whole domain $\Omega_0$ as well as a slice at $\theta = 0.6168$ in \cref{fig:plane-wave-plots}. Note the discontinuity of the right-hand side at $r=1/2$. We examine the convergence of the discretization with a fixed mesh but as $N_p \to \infty$ in \cref{fig:plane-wave-convergence}. After an initial plateau, we observe spectral convergence as we simultaneously increase $N_p$ on each element.

\begin{figure}[h!]
\centering
\includegraphics[width =0.49 \textwidth]{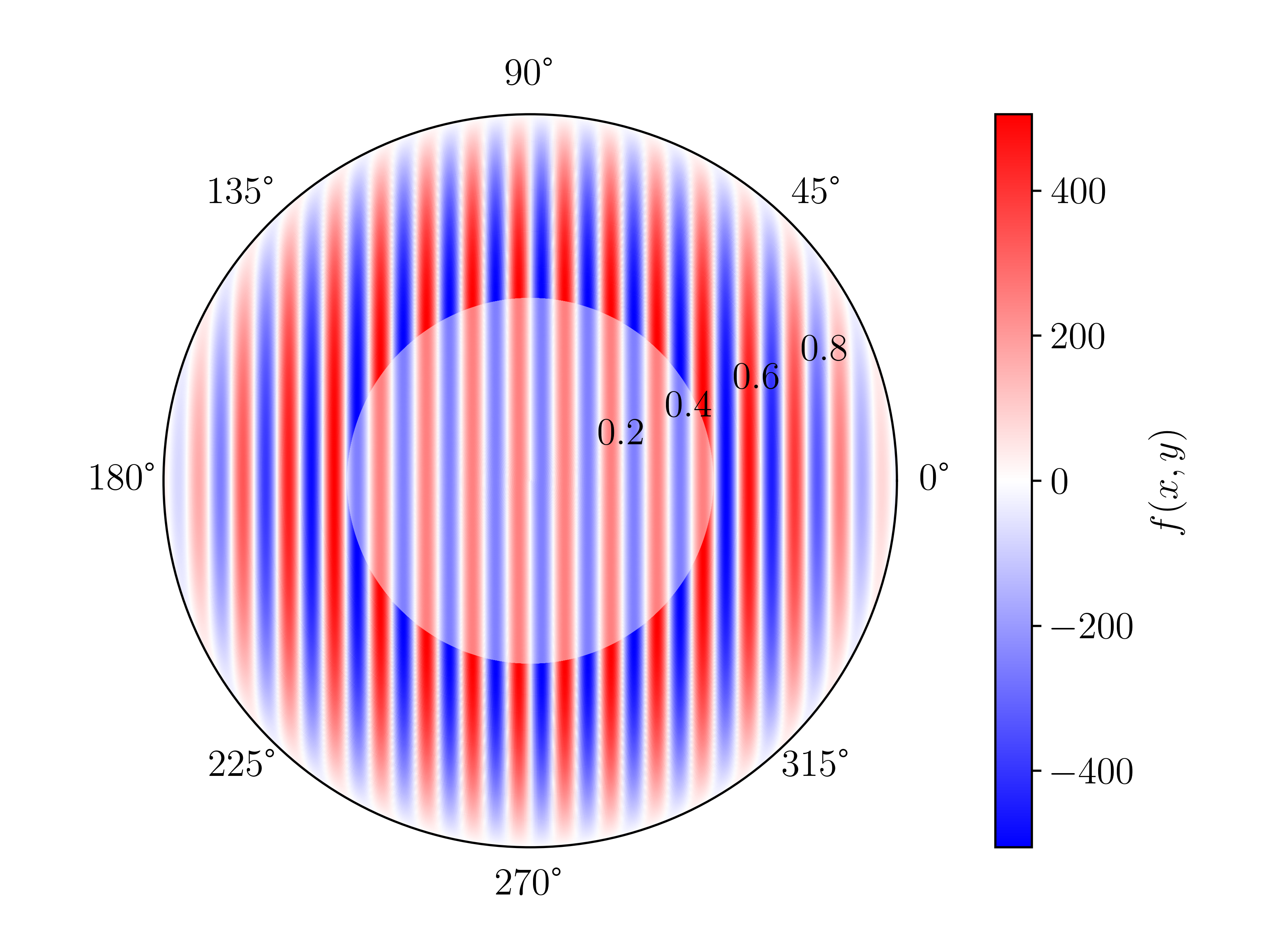} 
\includegraphics[width =0.49 \textwidth]{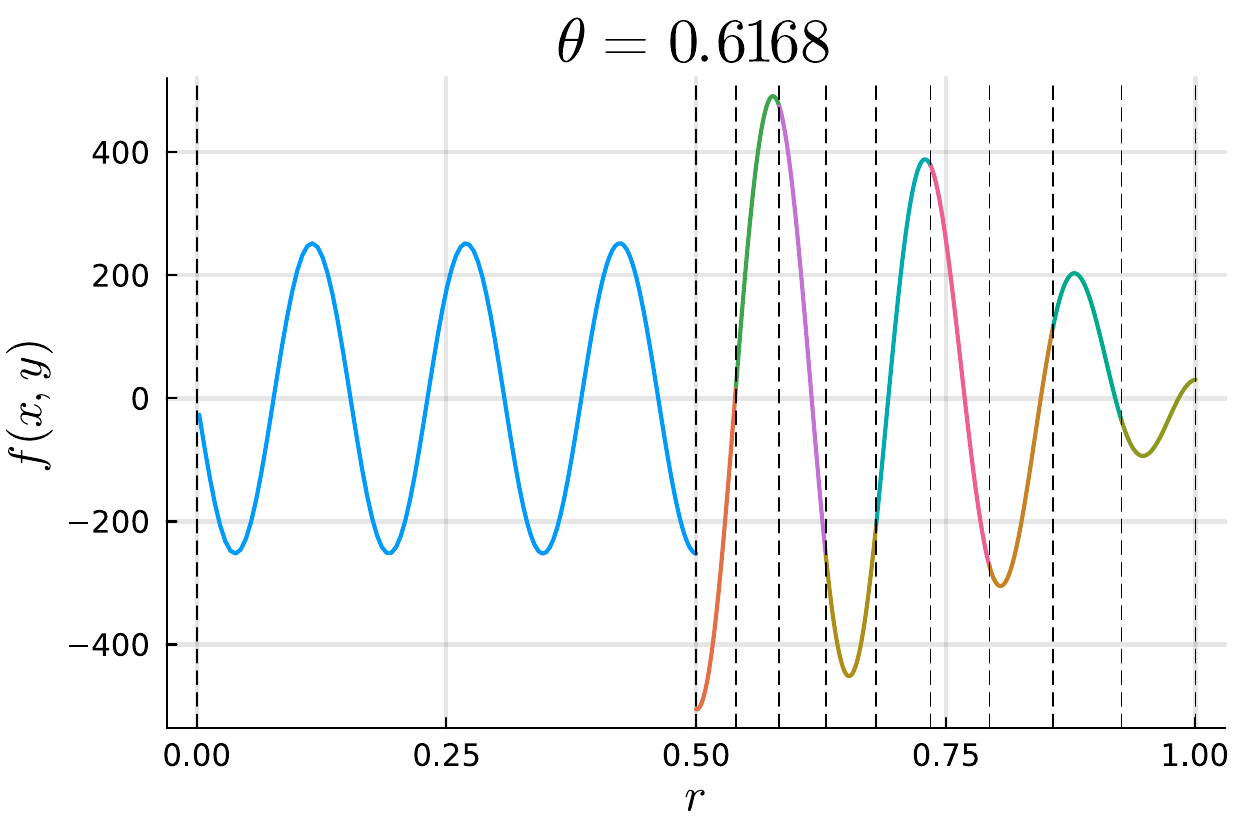} \\
\includegraphics[width =0.49 \textwidth]{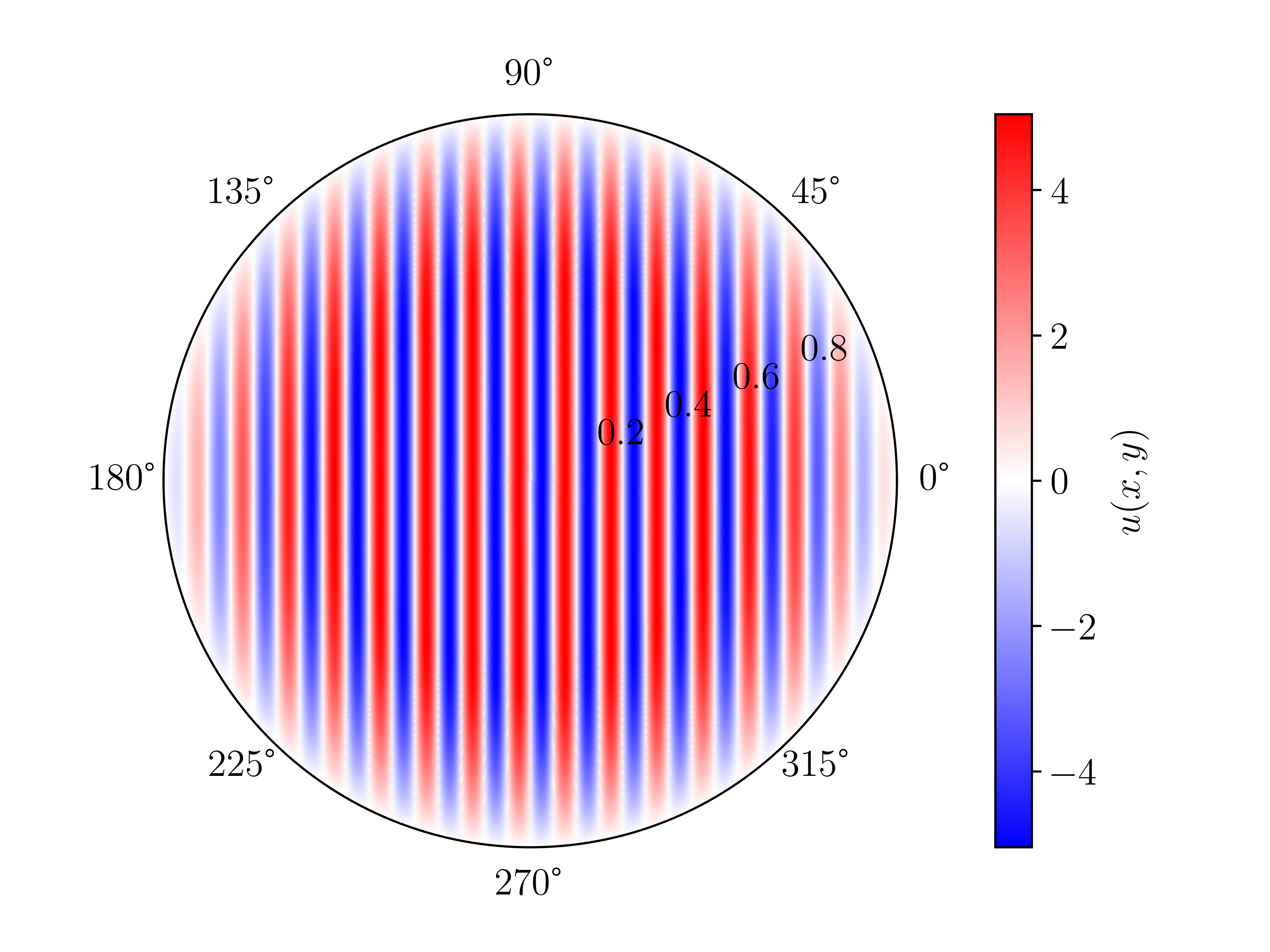} 
\includegraphics[width =0.49 \textwidth]{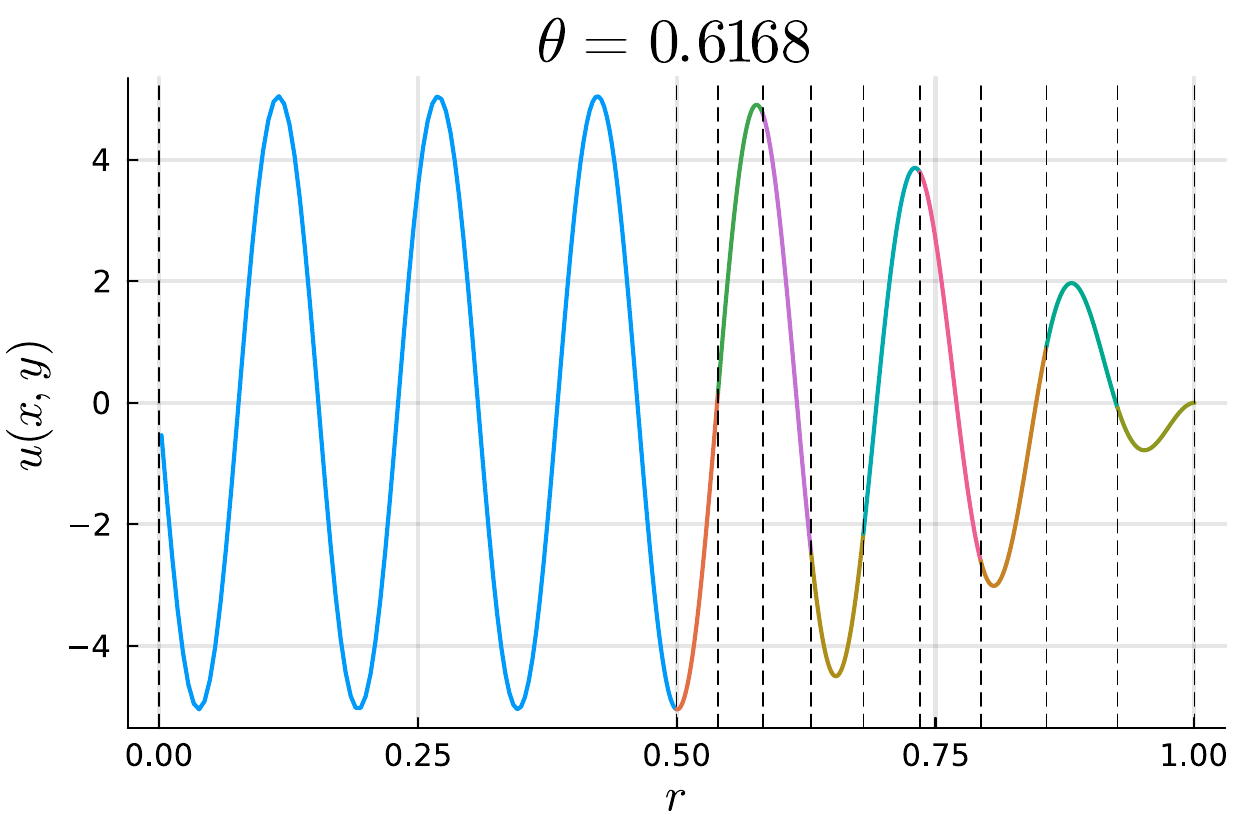}
\caption{Plots of the right-hand side $f(x,y)$ (top) and solution $u(x,y)$ (bottom) of the plane wave problem of \cref{sec:examples:plane-wave}. The black vertical lines in the slice plots indicate the edges of the cells in the mesh. The right-hand side has a radial discontinuity at $r=1/2$. Nevertheless the sparse $hp$-FEM accurately captures the jump and enforces the necessary continuity in the solution.}
\label{fig:plane-wave-plots}
\end{figure}

\begin{figure}[h!]
\centering
\includegraphics[width =0.49 \textwidth]{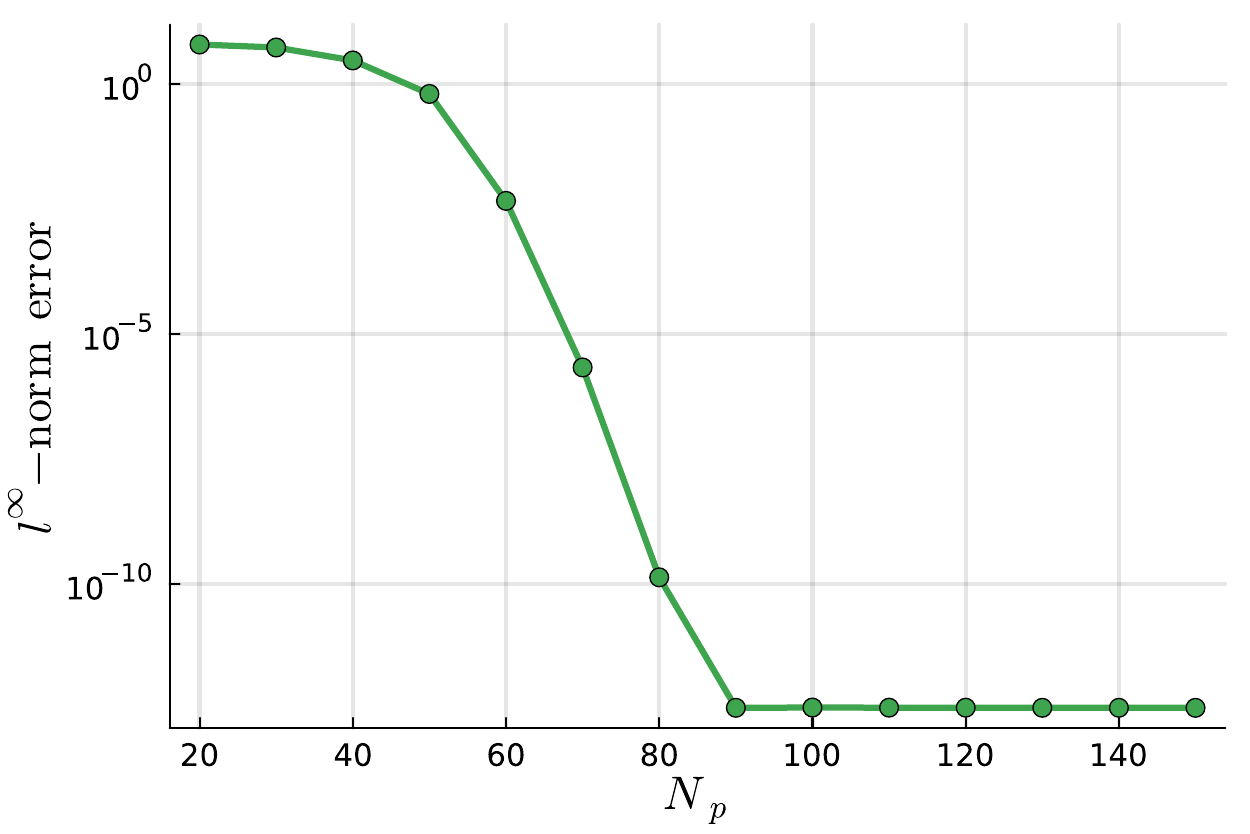} 
\caption{A semi-log convergence plot of the $\ell^\infty$-norm error of the hierarchical basis for the plane wave problem of \cref{sec:examples:plane-wave} with increasing polynomial degree $N_p$ on each of the 10 cells in the mesh. The plot indicates spectral convergence when $N_p > 50$ despite the radial discontinuities in $\lambda(r)$ and $f(x,y)$.}
\label{fig:plane-wave-convergence}
\end{figure}

\subsection{High frequency with a discontinuous Helmholtz coefficient}
\label{sec:examples:high-frequency}

In this example we consider the indefinite Helmholtz equation \cref{eq:helmholtz} on the unit disk domain $\Omega_0$. We pick a Helmholtz coefficient and a right-hand side with radial discontinuities at $r=1/2$:
\begin{align}
\lambda(r)
=
\begin{cases}
-80^2 & \text{if} \;\; 0 \leq r \leq 1/2,\\
-90^2 & \text{if} \;\; 1/2 < r \leq 1,
\end{cases}
\;\; \text{and} \;\;
f(x,y)
=
\begin{cases}
 2\sin(200x) & \text{if} \;\; 0 \leq r \leq 1/2,\\
\sin(100y) & \text{if} \;\; 1/2 < r \leq 1.
\end{cases}
\label{eq:high-freq}
\end{align}
We mesh the unit disk domain with $\mathcal{T}_h = \{ 0 \leq r \leq 1/2\} \cup \{ 2^{-(j+1)/11} \leq r \leq 2^{-j/11}\}_{j \in \{0, 1,\dots,11\}}$ culminating in $N_h = 12$ cells. We compute the entries of the matrices in the indefinite linear system \cref{eq:helmholtz-LA} as in the previous example. The resultant matrix is block diagonal where the blocks correspond to the Fourier mode decoupling. We solve for each block individually via a UL factorization with no pivoting for an optimal complexity solve: $\mathcal{O}(N_h N_p^2)$. We provide the spy plots of the $(m,j)=(175,1)$ Fourier mode submatrix of $A+M_\lambda$ and the UL factors in \cref{fig:high-frequency-spy} when $N_p = 200$. 

\begin{figure}[h!]
\centering
\subfloat[$(m,j)=(175,1)$ submatrix]{\includegraphics[width =0.32 \textwidth]{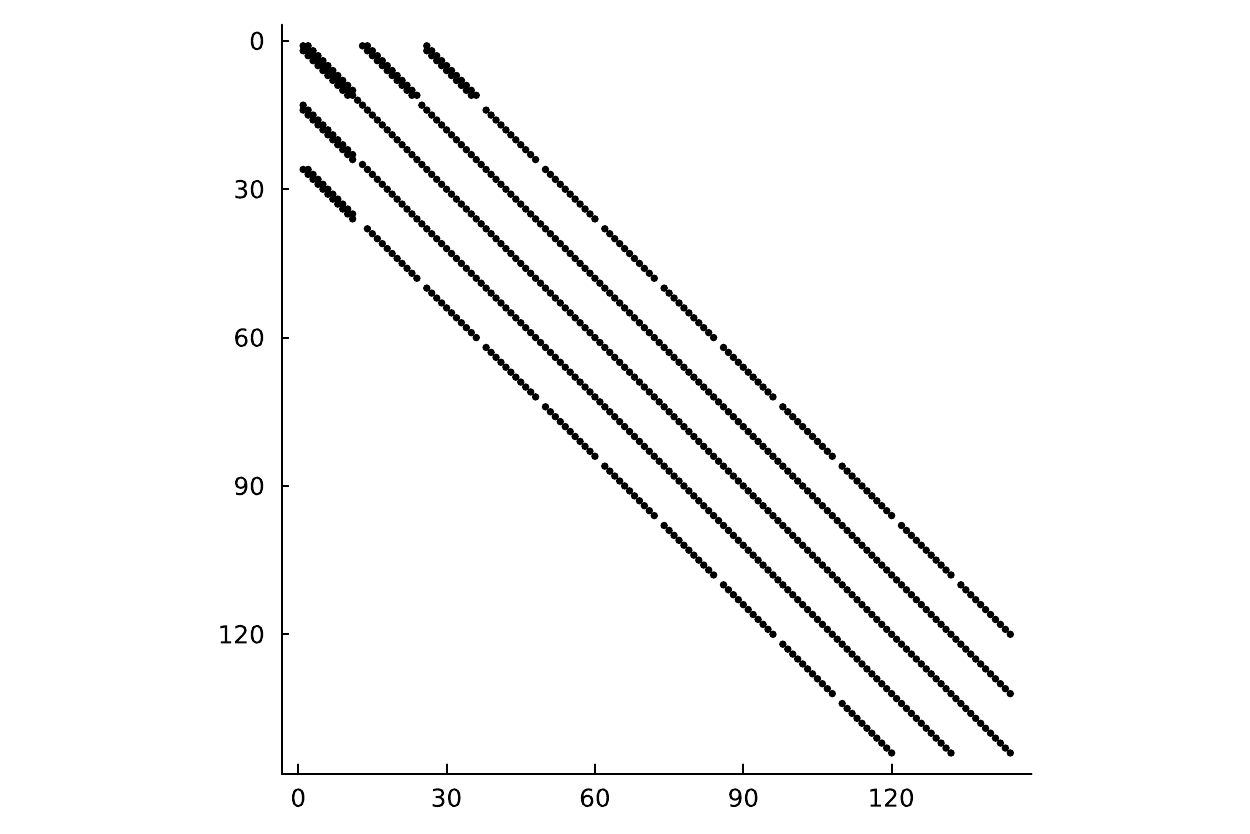}} 
\subfloat[reverse $L$ factor]{\includegraphics[width =0.32 \textwidth]{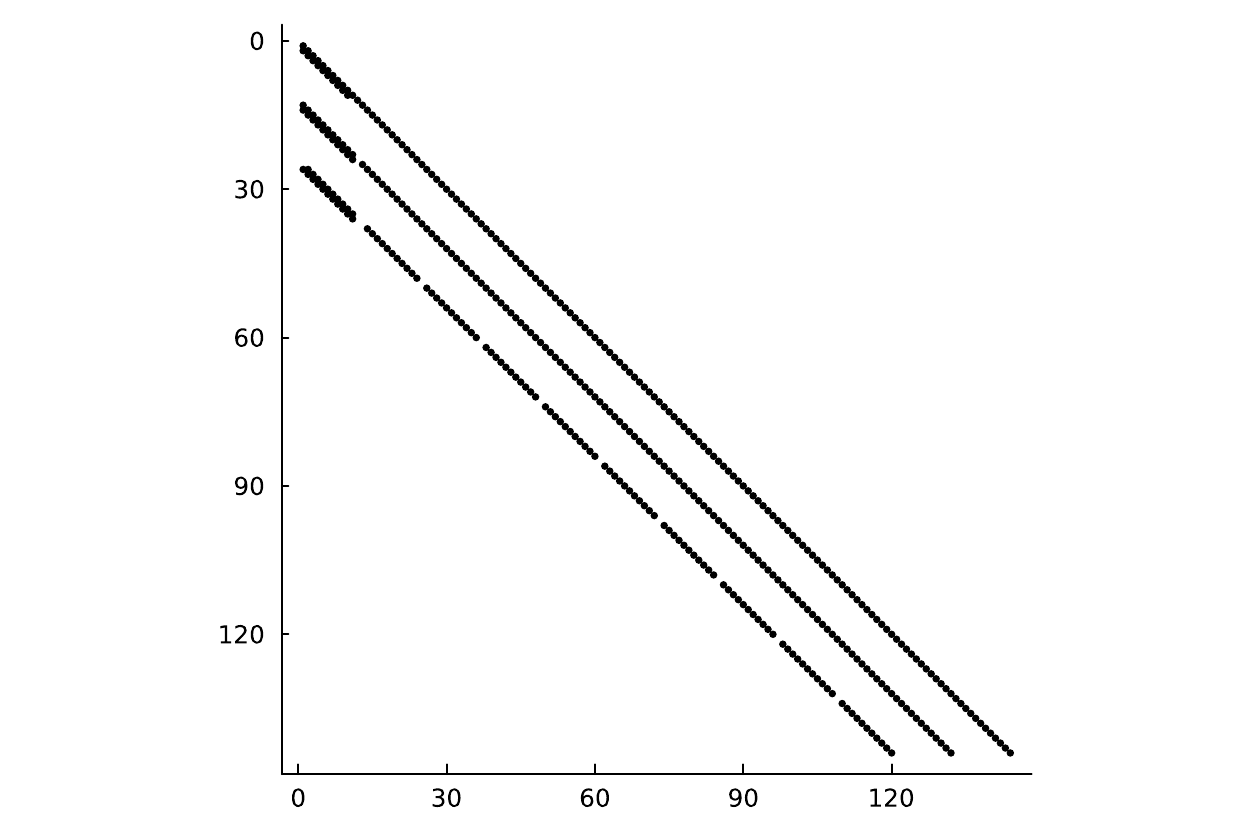}}
\subfloat[reverse $U$ factor]{\includegraphics[width =0.32 \textwidth]{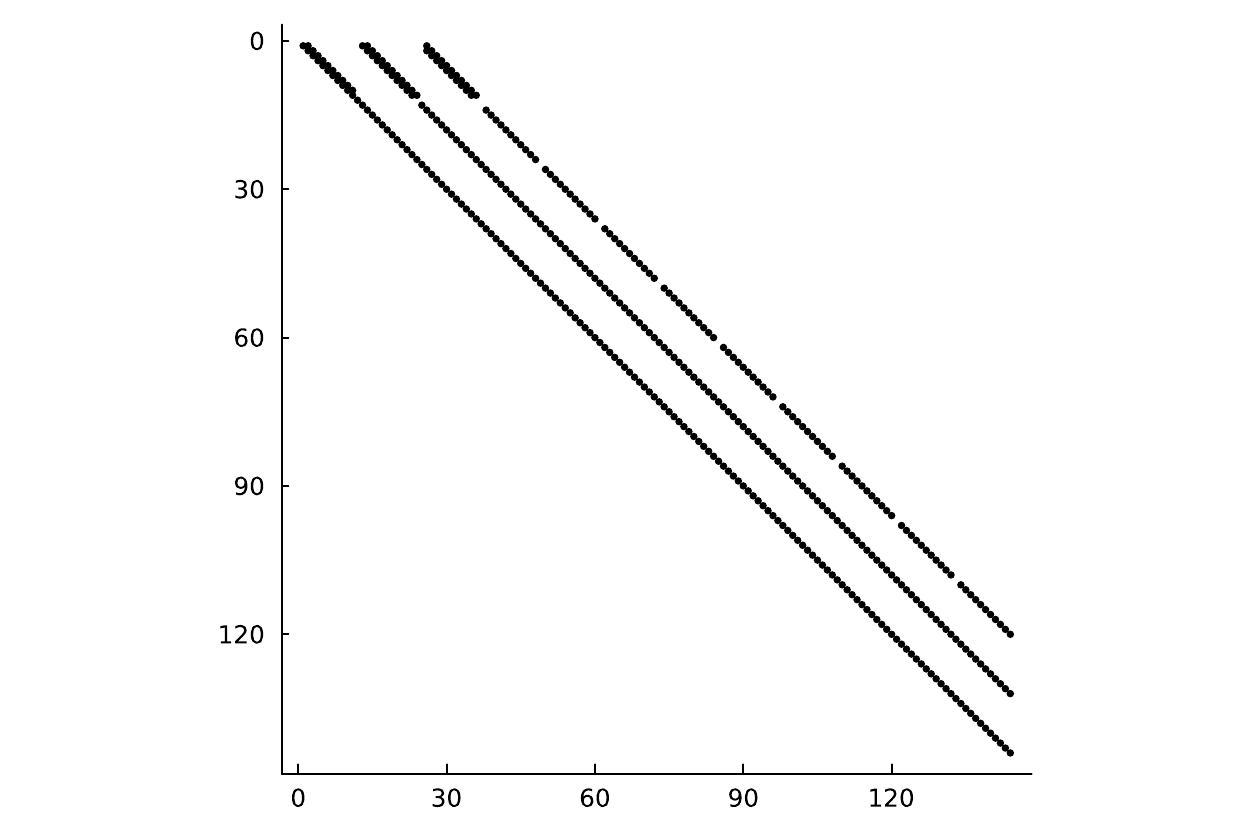}}
\caption{Spy plots of the $(m,j)=(175,1)$ Fourier mode submatrix of ${A+M_\lambda}$ and its UL factors (computed without pivoting) that arise in \cref{sec:examples:high-frequency}. The mesh contains 12 cells and the polynomial order is $N_p=200$ on each element. The UL factors are sparse, have no fill-in, and may be computed in linear complexity.}
\label{fig:high-frequency-spy}
\end{figure}

We do not have a closed-form expression for the exact solution of this problem. Hence, we measure the error against two over-resolved reference solutions. The first reference solution is computed via the sparse $hp$-FEM of this work. The second reference solution is computed via the SEM introduced in \cite[Sec.~6]{Papadopoulos2023} which discretizes the strong form of the equation. In this reference solution, the domain is meshed into the two cells $\Omega_0 = \Omega_{0,1/2} \cup \Omega_{1/2,1}$. A Zernike polynomial discretization is used in the inner disk cell and a Chebyshev--Fourier series discretization is used in the outer annular cell. Boundary conditions and continuity across the cells are enforced via a tau-method \cite[Sec.~6]{Papadopoulos2023}, see also \cite{Burns2020, Ortiz1969}. 

\begin{figure}[h!]
\centering
\includegraphics[width =0.49 \textwidth]{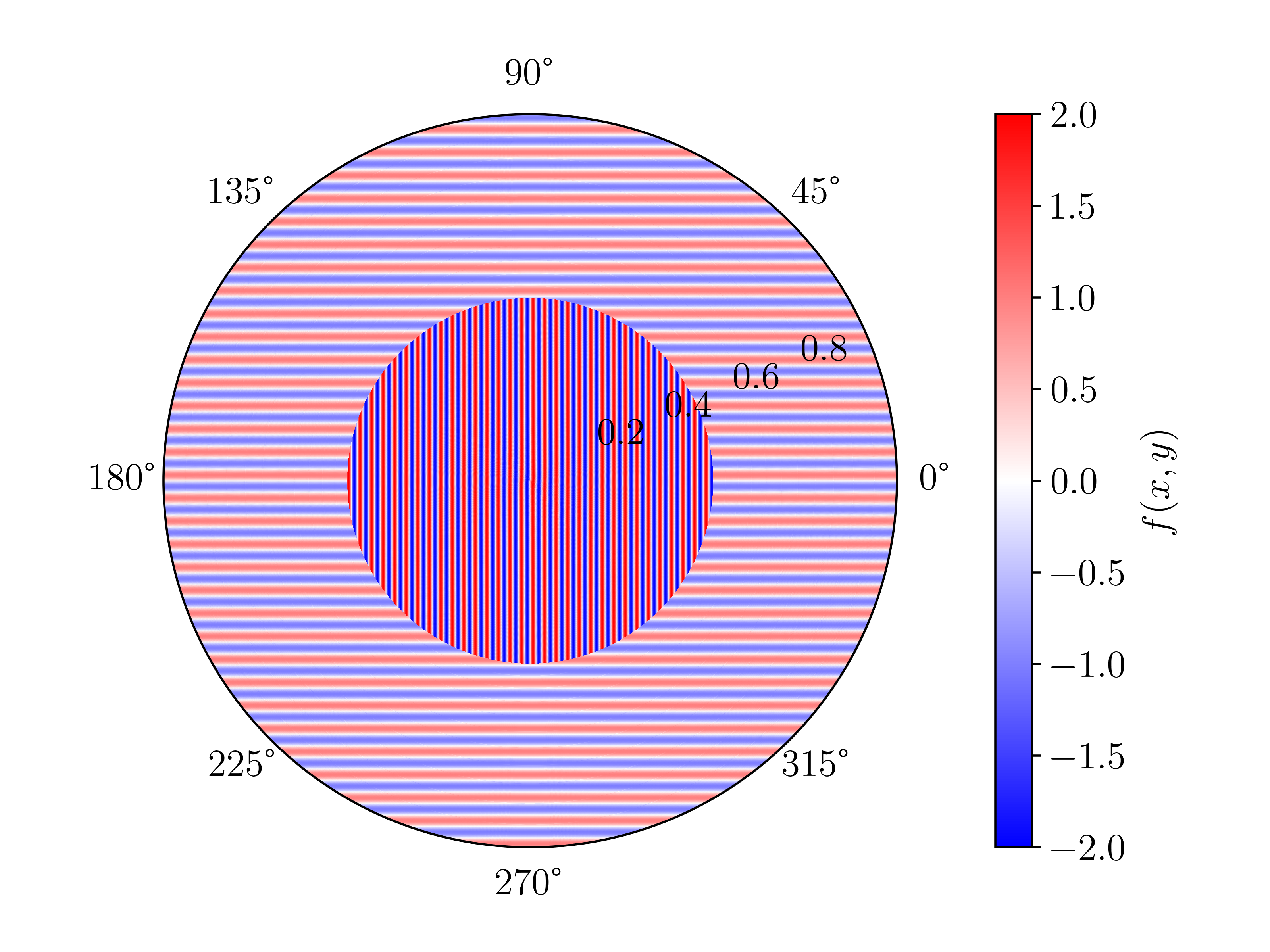} 
\includegraphics[width =0.49 \textwidth]{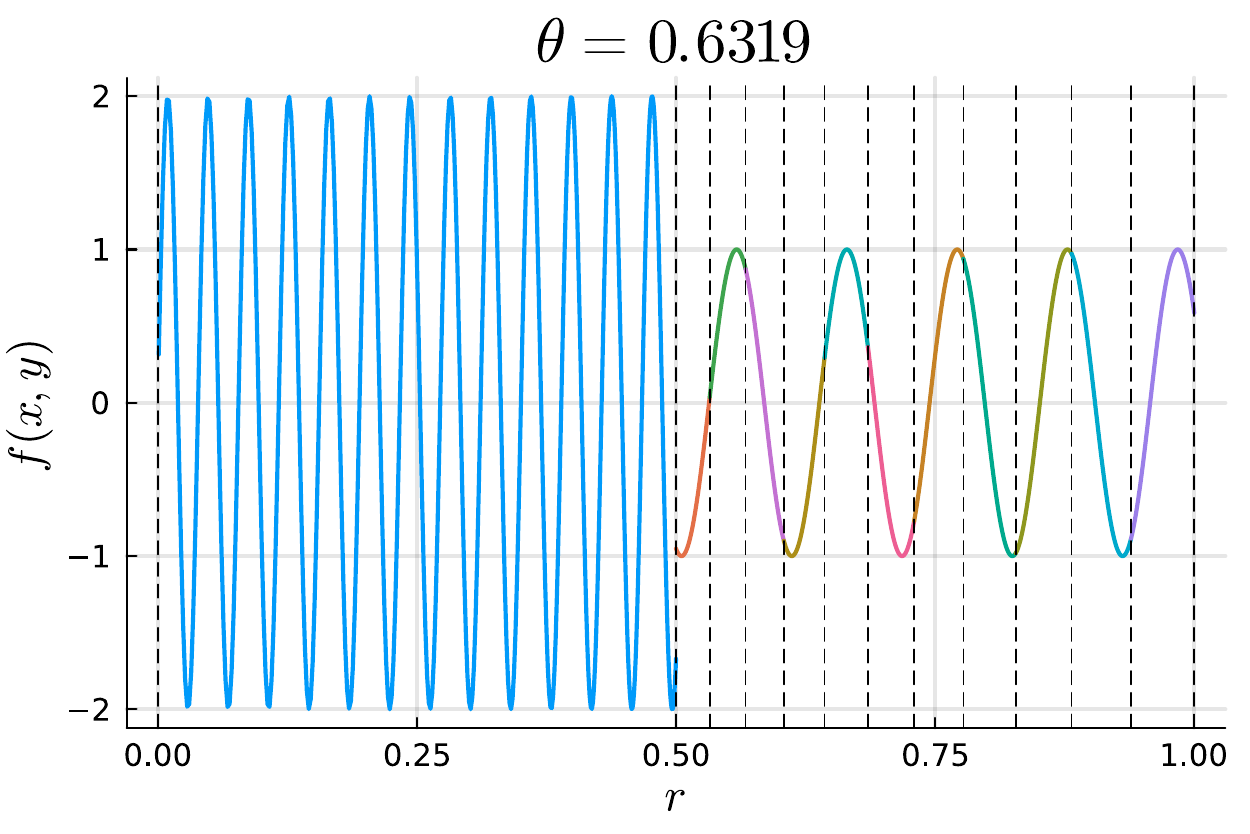} \\
\includegraphics[width =0.49 \textwidth]{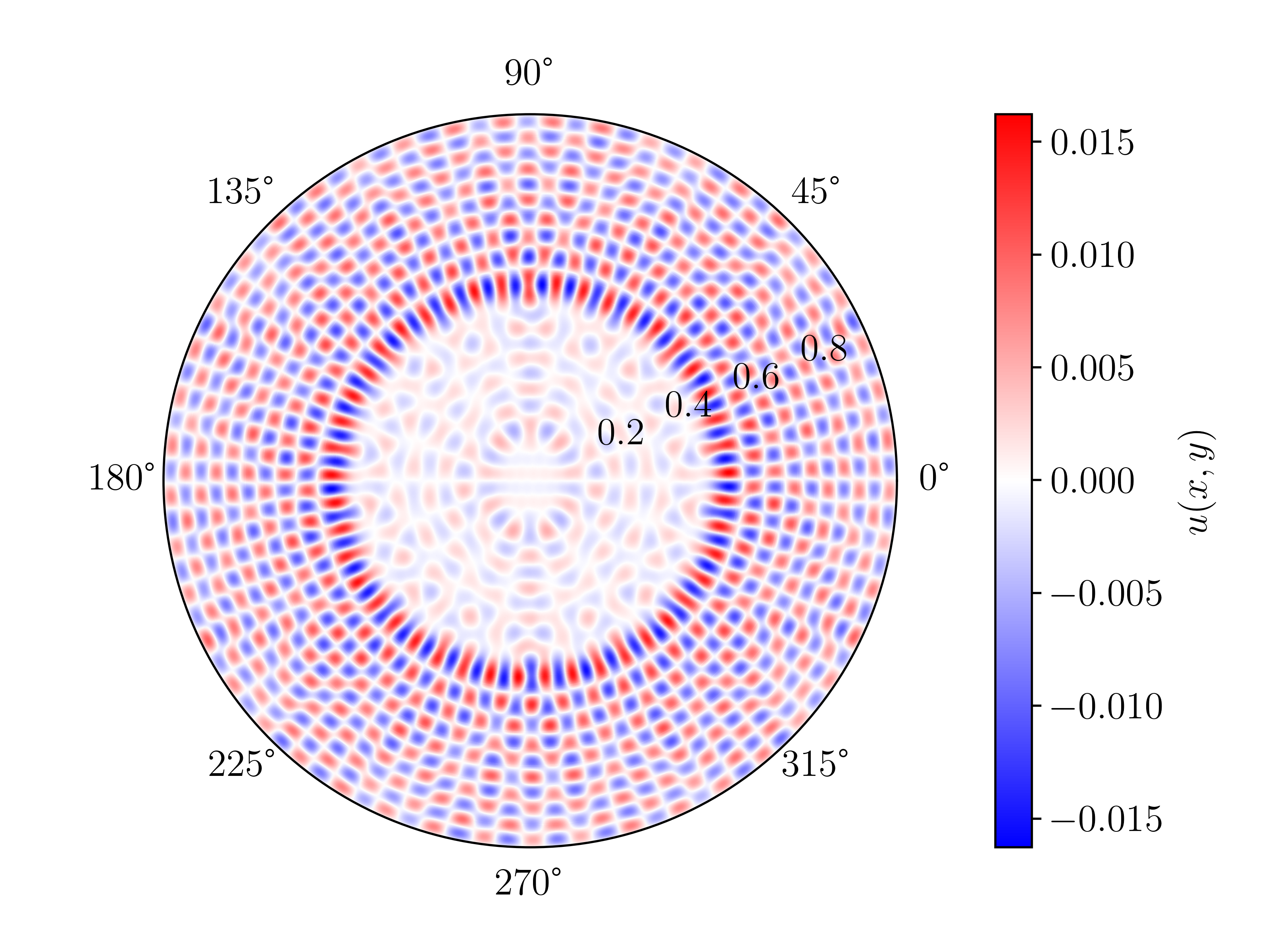} 
\includegraphics[width =0.49 \textwidth]{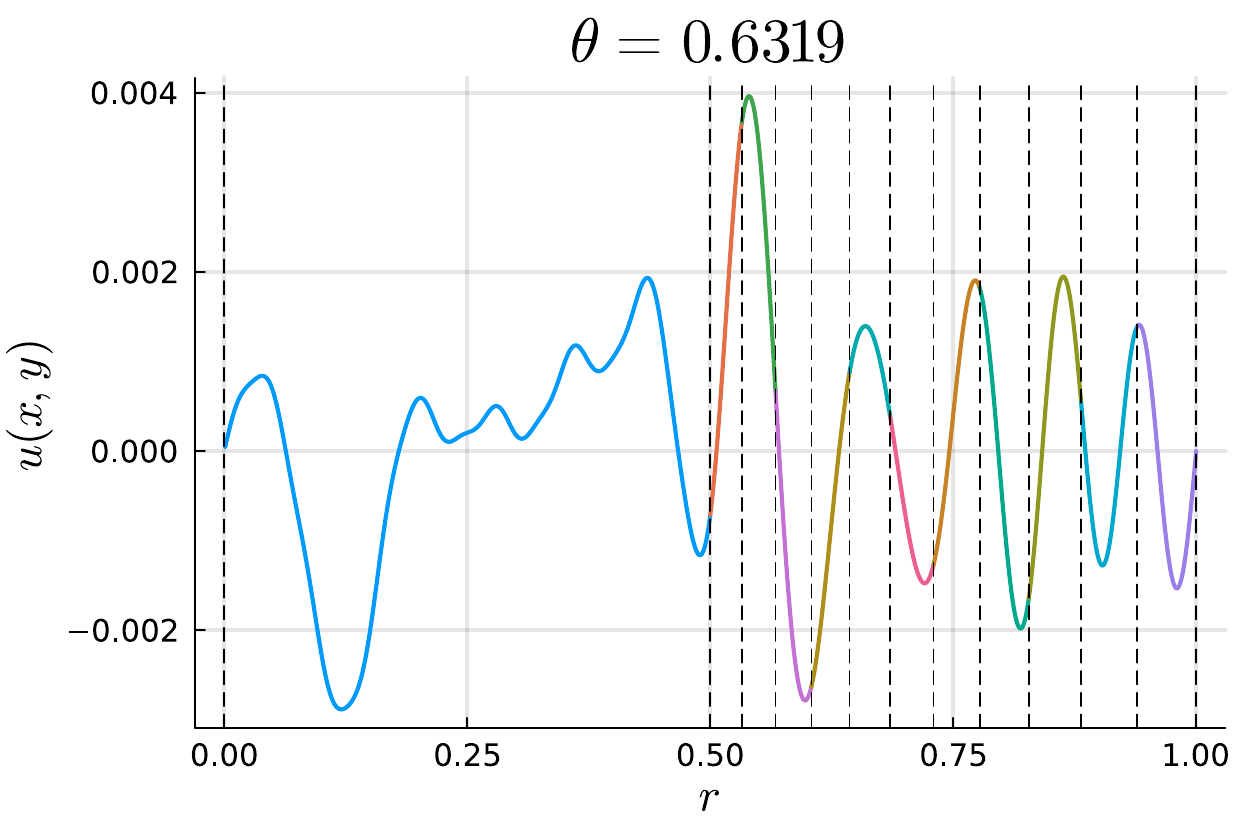}
\caption{Plots of the right-hand side $f(x,y)$ (top) and solution $u(x,y)$ (bottom) of the high frequency problem in \cref{sec:examples:high-frequency} with the Helmholtz coefficient and right-hand side as given in \cref{eq:high-freq}. The black vertical lines in the slice plots indicate the edges of the cells in the mesh. The right-hand side and Helmholtz coefficient have a radial discontinuity at $r=1/2$. Nevertheless the sparse $hp$-FEM accurately captures the jump and enforces the necessary continuity in the solution.}
\label{fig:high-frequency-plots}
\end{figure}

We plot the right-hand side $f$ and the approximated solution $u$ on the whole domain $\Omega_0$ as well as a slice at $\theta = 0.6319$ in \cref{fig:high-frequency-plots}. The negative Helmholtz coefficient causes oscillations to occur in the solution which are normally very difficult to capture. Moreover, one can spot the change in the behaviour of the solution as one crosses the radial discontinuity barrier at $r=1/2$. We plot the convergence of the sparse $hp$-FEM in \cref{fig:high-frequency-convergence}. After an initial period where the error decreases slowly, we observe spectral convergence for $N_p > 100$. Convergence is reached at $N_p = 160$. The error with respect to the reference SEM solution stagnates at $\mathcal{O}(10^{-12})$ with is due to the discretization error in the reference solution.

\begin{figure}[h!]
\centering
\includegraphics[width =0.49 \textwidth]{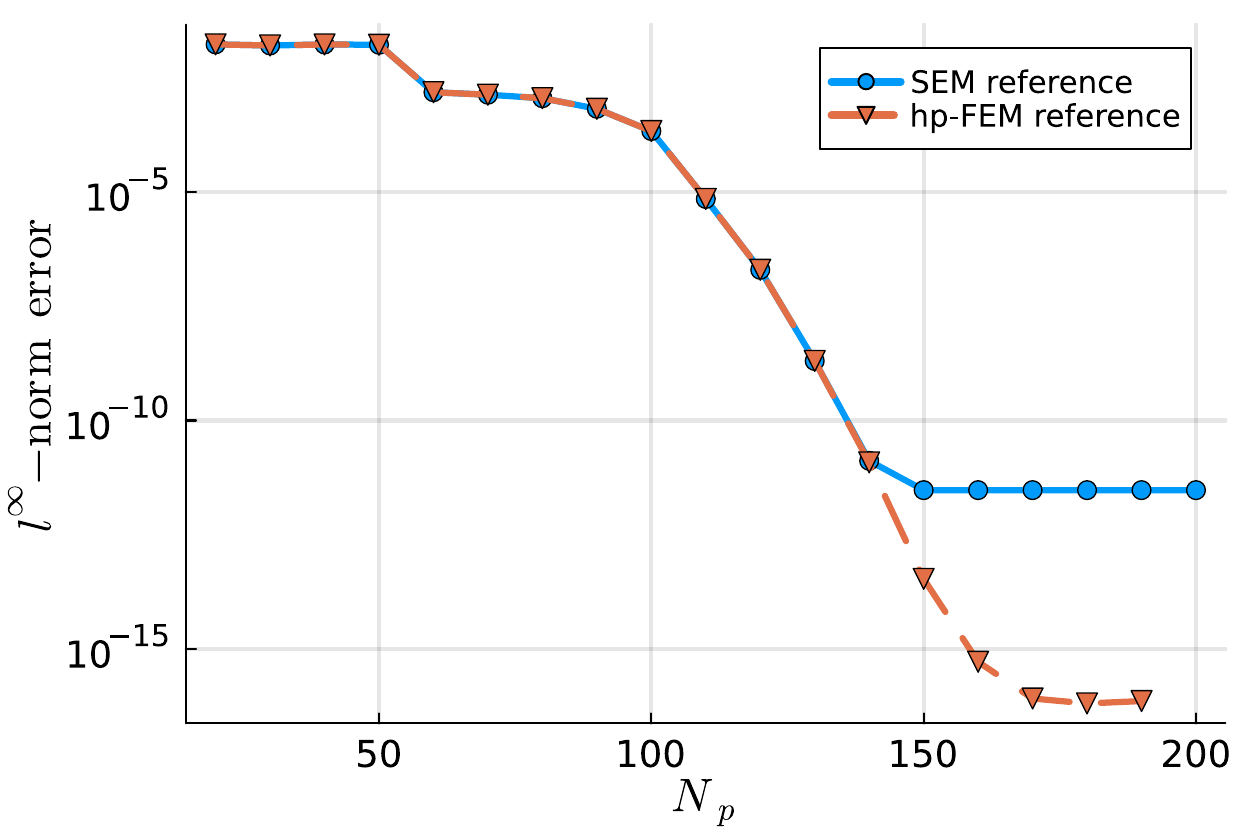} 
\caption{A semi-log convergence plot of the $\ell^\infty$-norm error of the continuous hierarchical basis for the high frequency problem in \cref{sec:examples:high-frequency} with increasing polynomial degree $N_p$ on each of the 12 cells in the mesh. After an initial plateau until $N_p = 100$, the plot indicates spectral convergence despite the radial discontinuities in the problem data and the high-frequency oscillations of the solution. The larger error plateau with respect to the reference SEM solution is due to the discretization error in the reference solution.}
\label{fig:high-frequency-convergence}
\end{figure}

\subsection{Time-dependent Schr\"odinger equation}
\label{sec:examples:schrodinger}

The time-dependent Schr\"odinger equation is one of the fundamental equations in quantum mechanics and describes how the wave function of a quantum system changes in response to the energy of the system. We consider the following form of the equation:

\begin{align}
\im \hbar \partial_t u(x,y,t) = \left(- \frac{\hbar^2}{2m} \Delta + V(r^2) \right) u(x,y,t), \;\; u(x,y,0) = u^{(0)}(x,y), \label{eq:sch:1}
\end{align}
where $t > 0$ and $(x,y) \in \mathbb{R}^2$. Here $\im$ is the imaginary unit, $\im^2=-1$, $\hbar$ is the reduced Planck constant, $m$ is the mass of the particle, and $V$ is the potential of the environment.  We assume $V$ depends on $r^2$ and is stationary. From here on, we choose the normalization constants $\hbar = 1$, $m=1/2$ and let $\lambda(r^2) = V(r^2)$. Thus \cref{eq:sch:1} reduces to
\begin{align}
\im \partial_t u(x,y,t) = \left(-\Delta + \lambda(r^2) \right) u(x,y,t), \;\; u(x,y,0) = u^{(0)}(x,y). \label{eq:sch:2}
\end{align}
\cref{eq:sch:2} is unitary and has the solution $u(x,y,t) = \E^{-\im t (-\Delta + \lambda)} u^{(0)}(x,y)$ which implies that 
\begin{align}
\|u(\cdot,\cdot,t)\|_{L^2(\mathbb{R}^2)} = \|u^{(0)}\|_{L^2(\mathbb{R}^2)} \;\; \text{for any} \;\; t > 0.  \label{eq:sch:3}
\end{align}
A favourable property for any temporal discretization of \cref{eq:sch:2} is that the energy is conserved, i.e.~\cref{eq:sch:3} holds. Provided the spatial discretization leads to a symmetric linear system, then the Crank--Nicolson method is the simplest temporal discretization that preserves energy and to which we restrict our investigations. High-order temporal discretizations will be considered in future work \cite{Hochbruck2010, Bader2016}.

Closed form expressions for the solutions of \cref{eq:sch:2} are difficult to find. When $\lambda(r^2) = r^2$ in \cref{eq:sch:2}, the problem is known as the quantum harmonic oscillator. Here, the eigenfunctions of the operator $(-\Delta + r^2)$ are known and take the form \cite[Eq.~(17)]{Wiss2015}:
\begin{align}
(-\Delta + r^2) \psi_{n,m}(x,y) = E_{n,m} \psi_{n,m}(x,y),
\label{sec:examples:sch:eig} 
\end{align}
where $\psi_{n,m}(x,y) = H_{n}(x) H_{m}(y) \exp(-(x^2+y^2)/2)$ and $E_{n,m} = 2(n+m+1)$. Here $H_n$, $n \in \mathbb{N}_0$, denote the orthonormalized Hermite polynomials \cite[Sec.~18.3]{dlmf}. Thus if $u^{(0)}(x,y) = \psi_{n,m}(x,y)$, then the solution of \cref{eq:sch:2} is $u(x,y,t) = \E^{-\im E_{n,m} t} \psi_{n,m}(x,y)$.

As the domain in \cref{eq:sch:2} is $\mathbb{R}^2$, we truncate the domain to the disk $\bar \Omega = \{ 0 \leq r \leq 50\}$. We mesh the domain with the 16 cells $\mathcal{T}_h =  \{ 0 \leq r \leq 50 (6/5)^{-15} \} \cup \bigcup_{j = 0}^{14} \{ 50 (6/5)^{j+1} \leq r \leq 50 (6/5)^{j} \} $. We discretize \cref{eq:sch:2} in the time variable with the Crank--Nicolson method with the uniform time step $\delta t$. We consider the final time $T = 2\pi/E_{20,21}$ which corresponds to one full period of oscillation of the solution, $u(x,y,T) = u(x,y,0)$. Rewritten in quasimatrix form, then at each time step $k = 0,1,2,\dots$, the time-stepping problem reduces to solving:
\begin{align}
(2M + \im \delta t (A + M_{r^2})) {\bf u}^{(k+1)} = (2M - \im \delta t (A + M_{r^2})) {\bf u}^{(k)}.
\label{eq:sch:4}
\end{align}
We consider the initial state $u^{(0)}(x,y) = \psi_{20,21}(x,y)$. We discretize in space with the continuous hierarchical FEM basis with truncation degree $N_p=100$ on each cell. This problem and discretization preserves the Fourier mode decoupling and, therefore, the matrices in \cref{eq:sch:4} are block diagonal with $2N_p+1$ blocks which may be decoupled into $2N_p+1$ independent linear systems. Each block permits a complex-valued UL factorization with no pivoting. Moreover, the complex-valued UL factors may be computed in linear complexity and are sparse with zero fill-in as observed in \cref{fig:schrodinger-spy}.

\begin{figure}[h!]
\centering
\subfloat[$(m,j)=(70,1)$ submatrix]{\includegraphics[width =0.32 \textwidth]{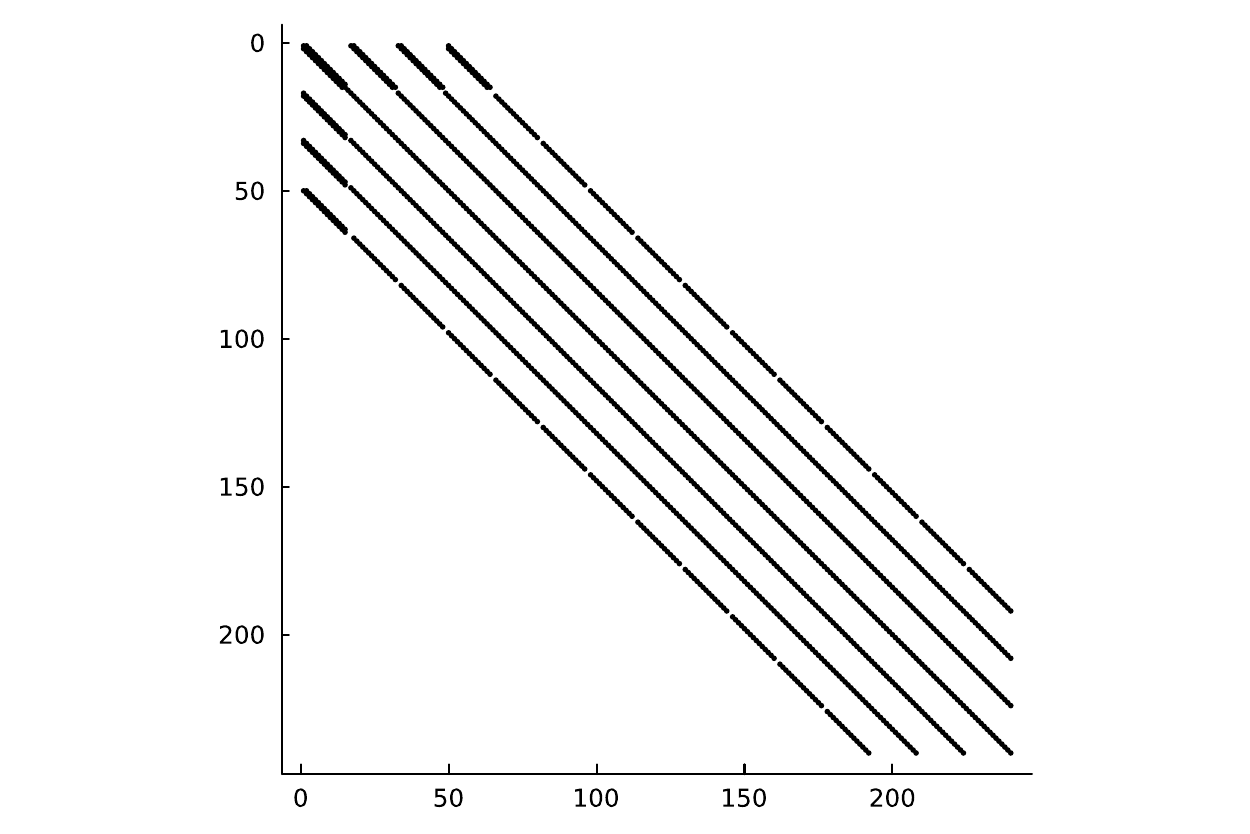}} 
\subfloat[reverse $L$ factor]{\includegraphics[width =0.32 \textwidth]{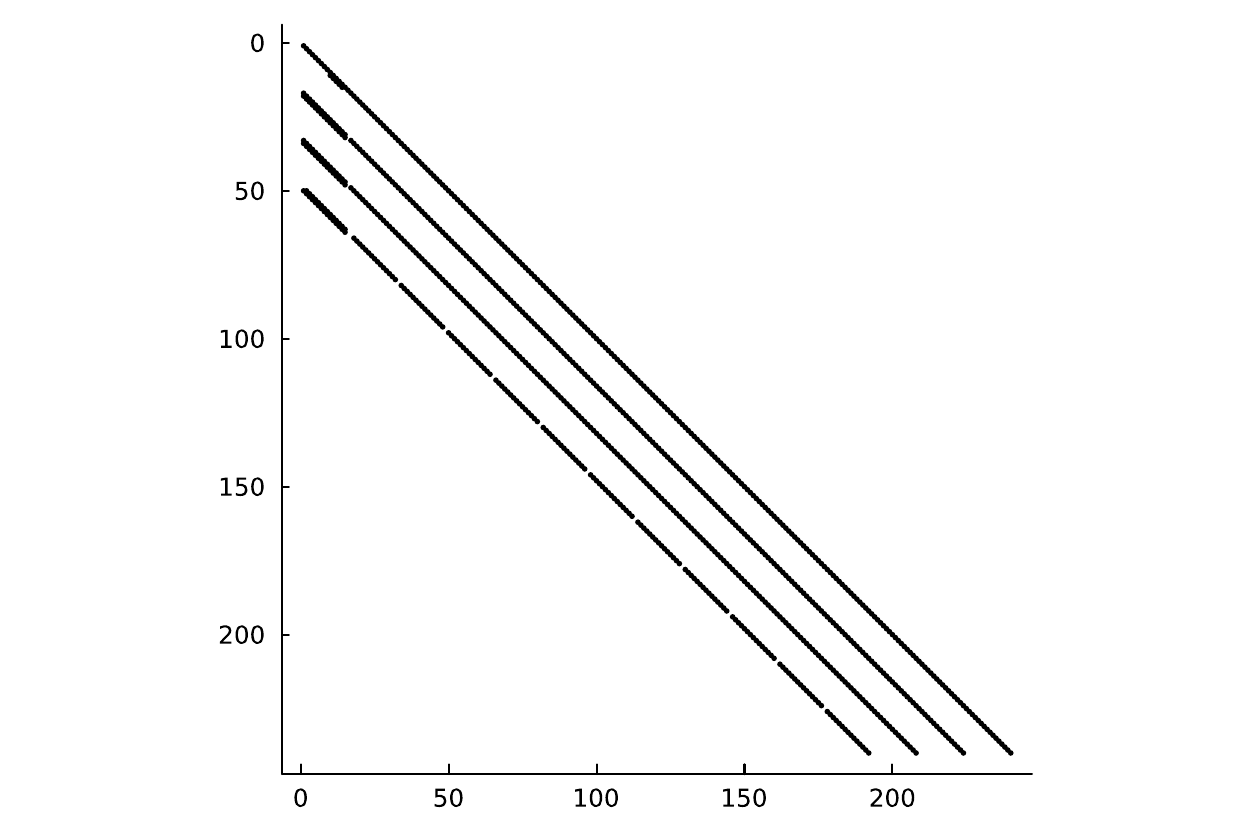}}
\subfloat[reverse $U$ factor]{\includegraphics[width =0.32 \textwidth]{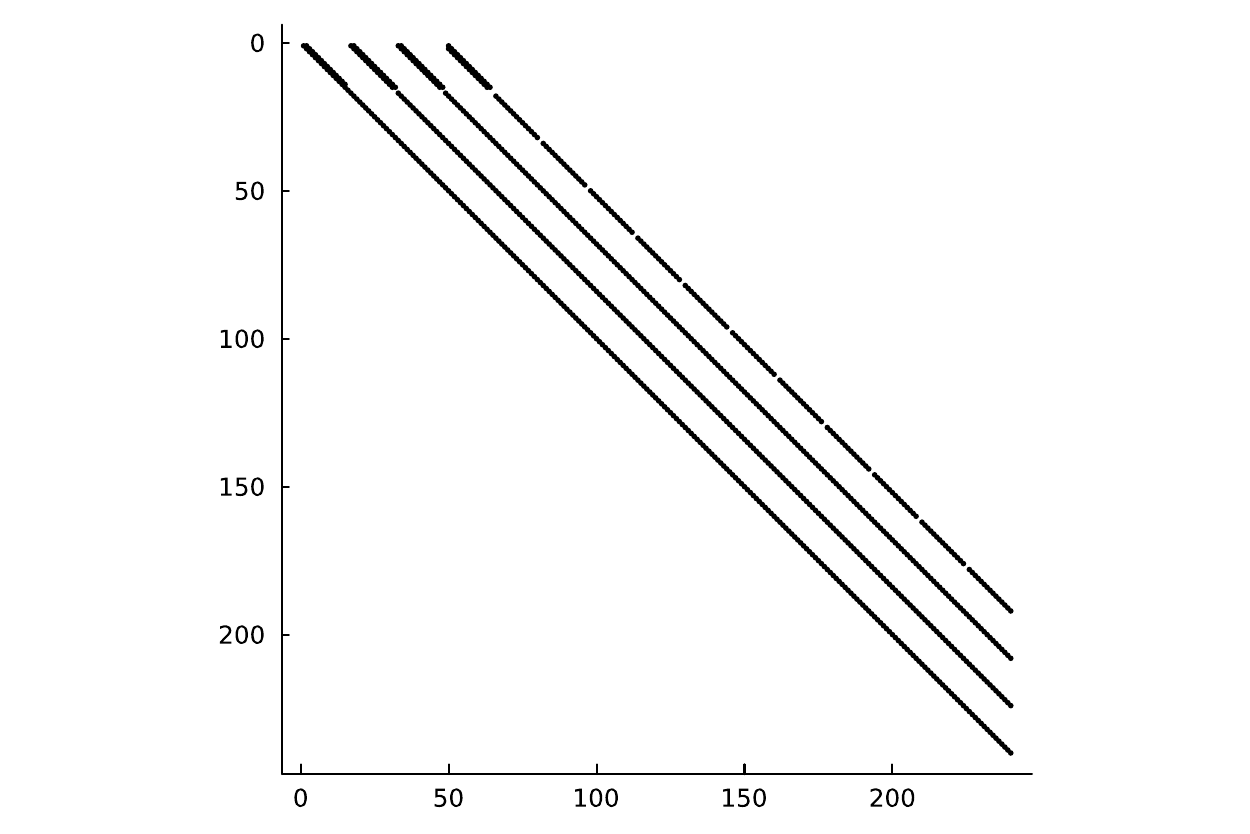}}
\caption{Absolute value spy plots of the $(m,j)=(70,1)$ Fourier mode complex-valued submatrix of $2M + \im \delta t (A + M_{r^2})$ and the complex-valued UL factors (computed without pivoting). The mesh contains 16 cells and the polynomial order is $N_p=100$ on each element. The UL factors are sparse, have no fill-in, and may be computed in linear complexity.}
\label{fig:schrodinger-spy}
\end{figure}

The initial state is plotted in \cref{fig:schrodinger-plots} and the spatial discretization yields an $\ell^\infty$-norm error of $7.94 \times 10^{-15}$. As discussed one must truncate the domain sufficiently large in order to sufficiently emulate an unbounded domain. However, since the initial state exponentially decays as $r \to \infty$, the initial state evaluates to below (double) machine precision for $r>9$. A one-cell discretization would struggle to sufficiently capture the oscillations of the initial state close to the origin. Our investigations revealed that a one-cell Zernike ${\bf Z}^{(0)}(x,y)$ discretization required a truncation degree $N_p=700$ to resolve the initial state to an $\ell^\infty$-norm error of $5 \times 10^{-15}$ on the domain $\Omega = \{ 0 \leq r \leq 50\}$.

A convergence plot is given in \cref{fig:schrodinger-convergence} for decreasing step size $\delta t$ where we measure the $\ell^\infty$-norm error at the final time step $T = 2\pi/E_{20,21}$ (which always had the largest error across all the time steps). We observe the expected $\mathcal{O}(\delta t^2)$ convergence. \cref{fig:schrodinger-convergence} displays a plot of the difference in the $L^2(\Omega)$-norm between the approximate solution at time step iterate $k$ and the discretization of the initial state for the finest temporal discretization where $\delta t = 5.75 \times 10^{-5}$. In other words we plot $|\| u^{(k)} \|_{L^2(\Omega)} - \| u^{(0)} \|_{L^2(\Omega)}|$ for $k = 1,2,\dots,1300$. We see that there is some loss of energy due to the floating point error. At each time step $|\| u^{(k+1)} \|_{L^2(\Omega)} - \| u^{(k)} \|_{L^2(\Omega)}|  \approx 10^{-13}$ and $1300 \times 10^{-13} \approx 10^{-10}$.

\begin{figure}[h!]
\centering
\includegraphics[width =0.32 \textwidth]{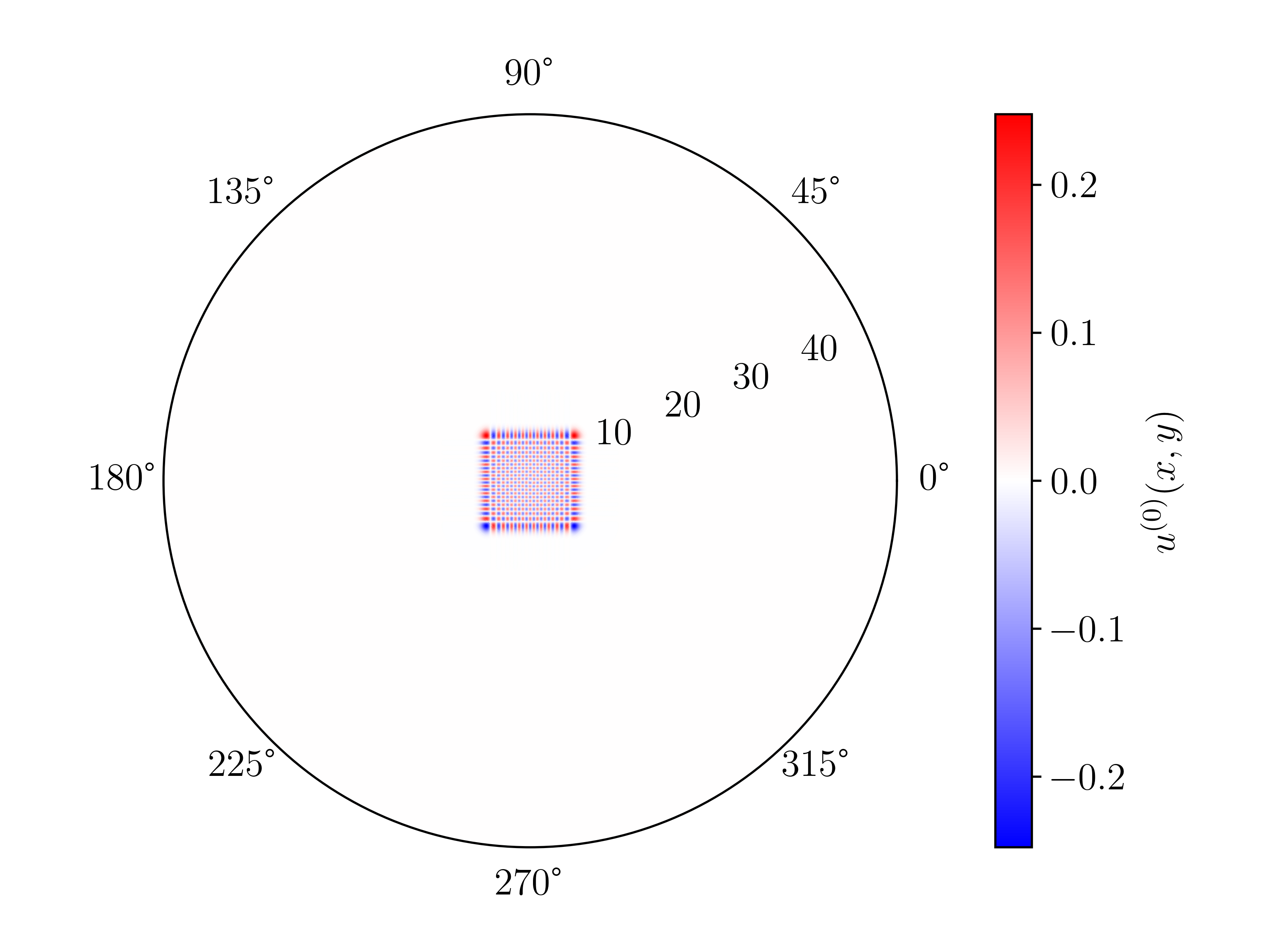}
\includegraphics[width =0.32 \textwidth]{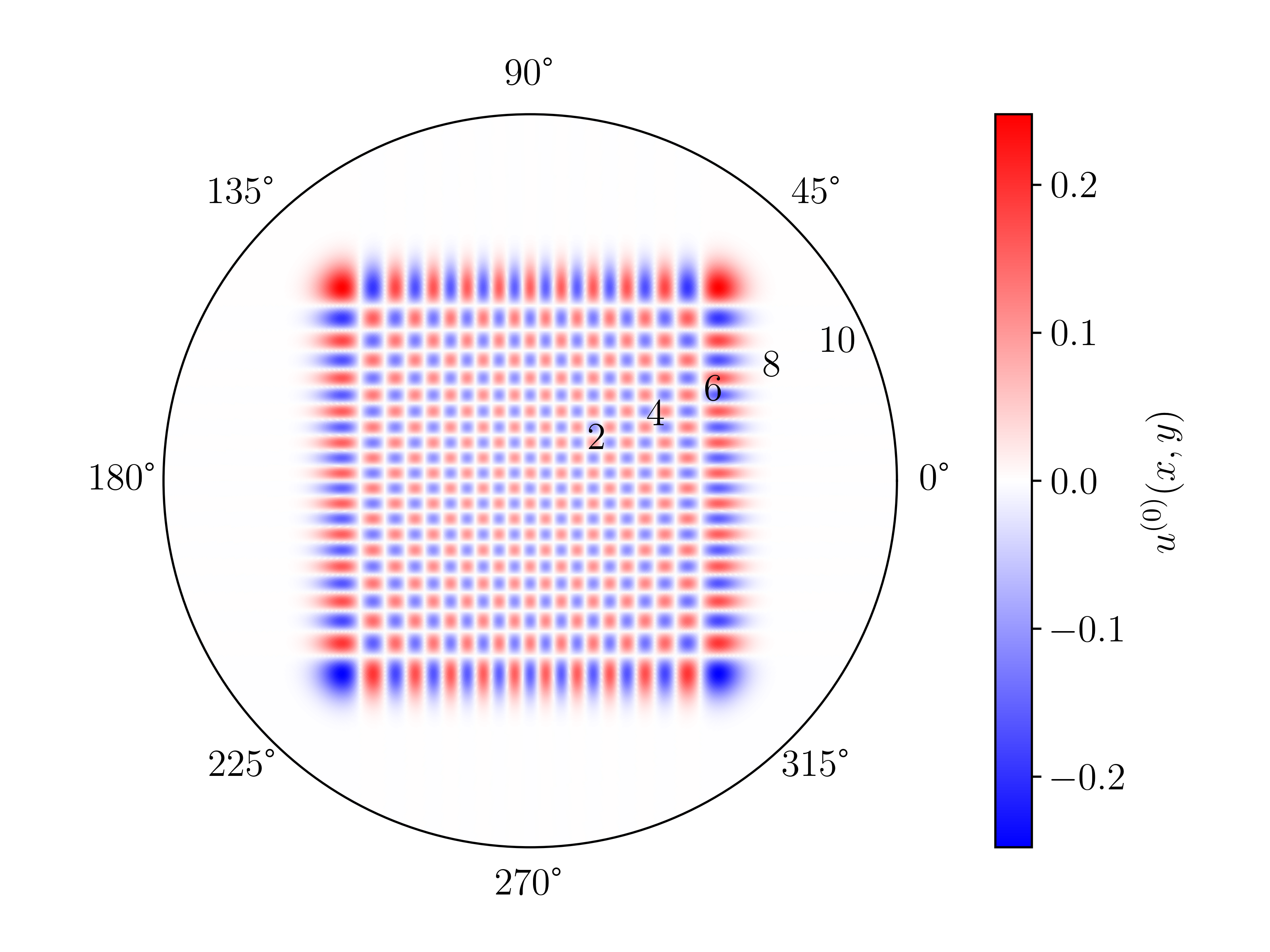}
\includegraphics[width =0.32 \textwidth]{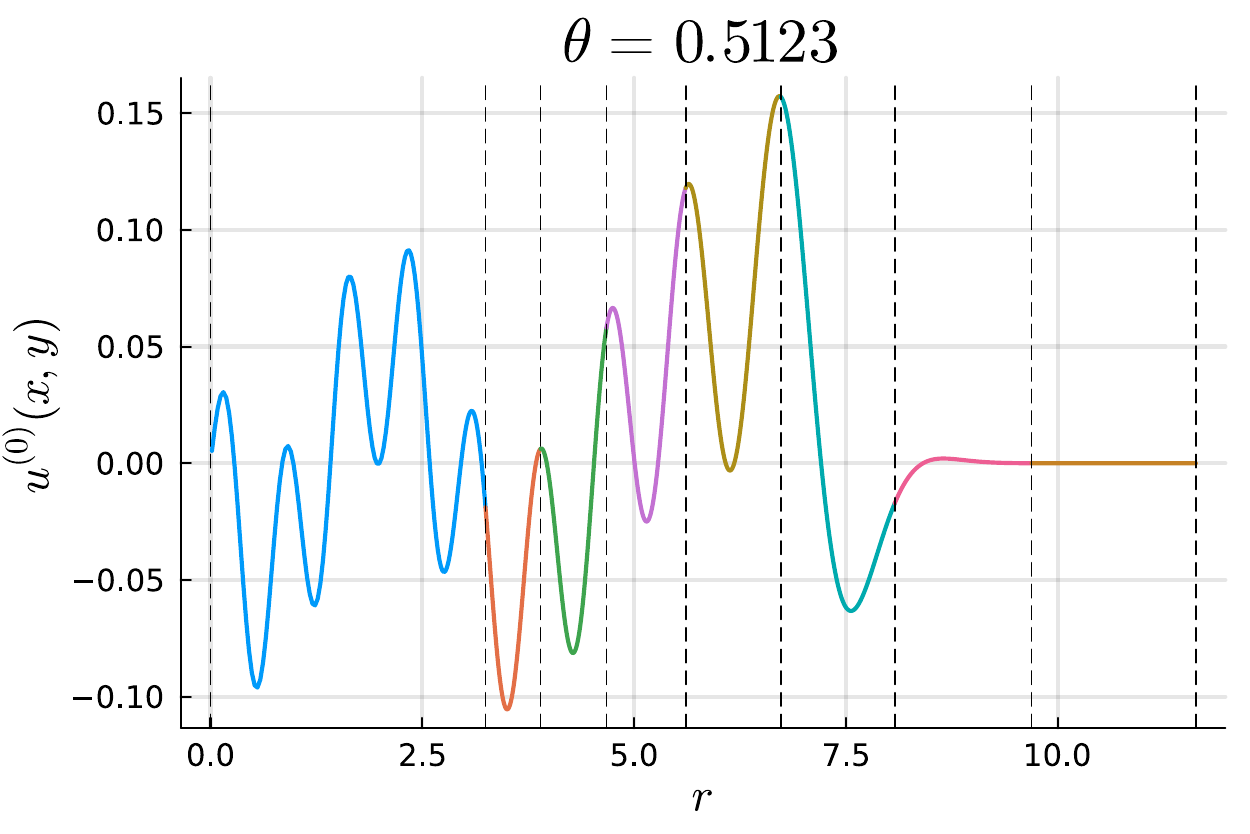}
\caption{Plots of the initial state $u^{(0)}(x,y)$ in the Schr\"odinger equation of \cref{sec:examples:schrodinger} on the whole domain (left) and zoomed in (middle). A slice of the initial state in the first 8 cells is also displayed (right). The black vertical lines in the slice plot indicates the edges of the cells.}
\label{fig:schrodinger-plots}
\end{figure}

\begin{figure}[h!]
\centering
\includegraphics[width =0.49 \textwidth]{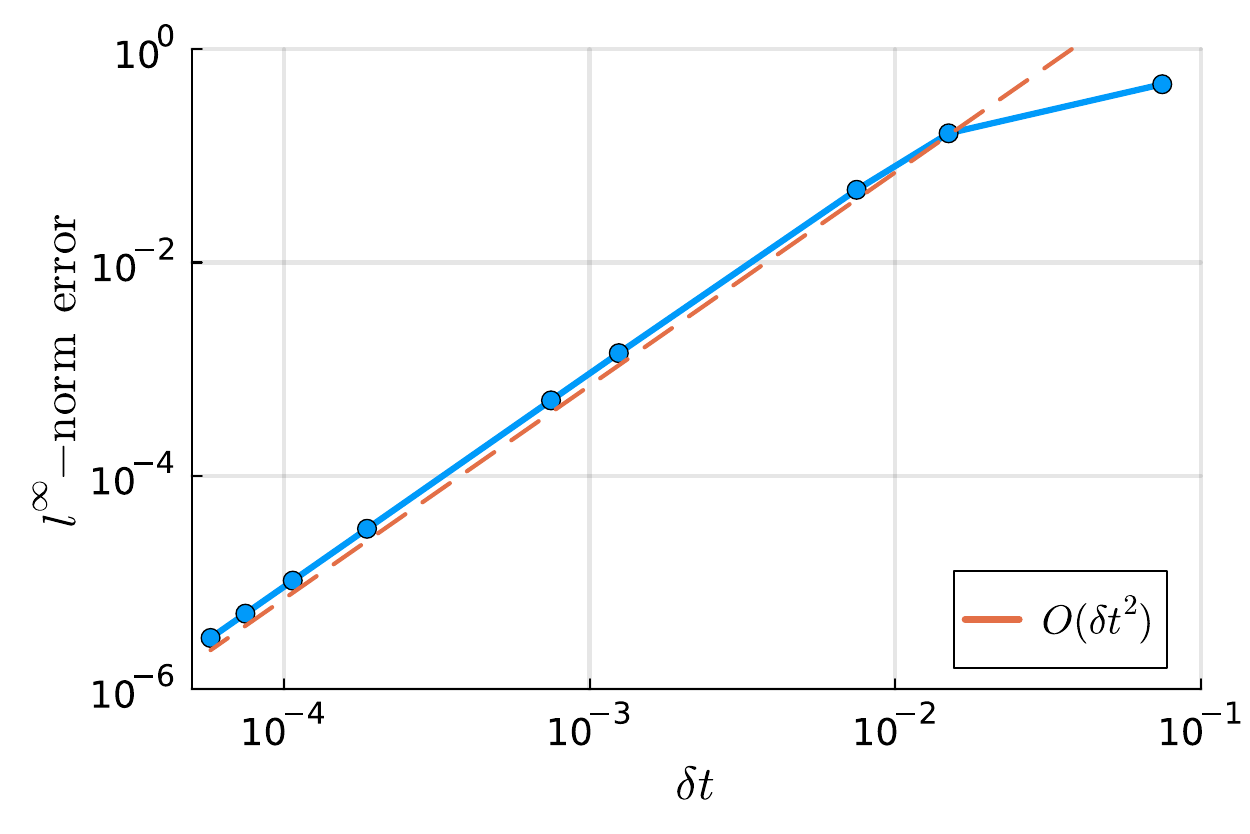}
\includegraphics[width =0.49 \textwidth]{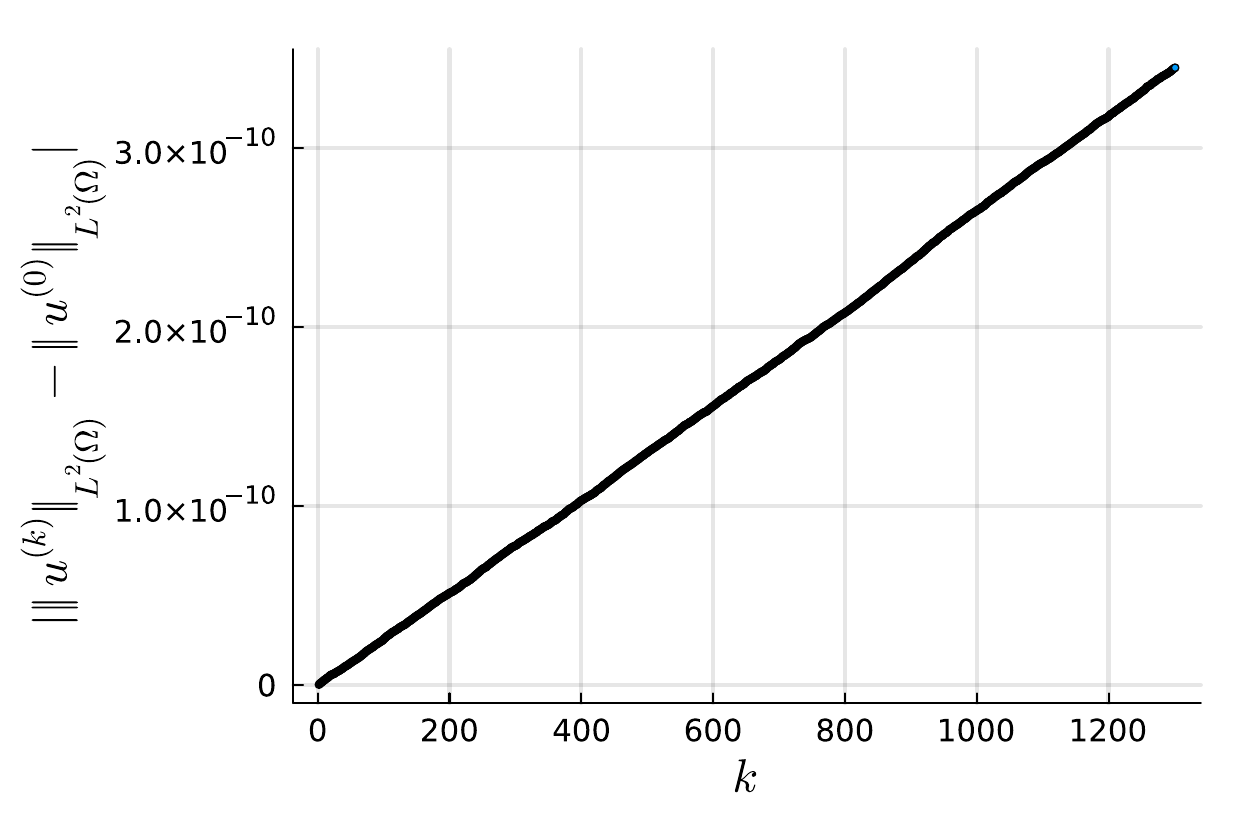}
\caption{(Left) A log-log convergence plot of the $\ell^\infty$-norm error at the final time step $T = 2\pi/E_{20,21}$ as $\delta t \to 0$ for the time-dependent Schr\"odinger equation described in \cref{sec:examples:schrodinger}. The domain is a disk with outer radius 50. It is meshed into 16 cells where we truncate at $N_p=100$ on each element. The plot indicates the expected $\mathcal{O}(\delta t^2)$ convergence of the Crank--Nicolson temporal discretization. (Right) The difference in the $L^2(\Omega)$-norm between a time step iterate and the discretization of the initial state when $\delta t = 5.75 \times 10^{-5}$.}
\label{fig:schrodinger-convergence}
\end{figure}

\subsection{Rotationally anisotropic coefficient}
\label{sec:examples:non-separable}

Here we consider an equation with a Helmholtz coefficient that is rotationally anisotropic, i.e.~$\lambda$ cannot be written as $\lambda(r)$. Let $\Omega = \{10^{-2} < r < 1\}$. We seek $u \in H^1_0(\Omega)$ that satisfies
\begin{align}
(-\Delta - 80^2 x ) u(x,y)  =
\begin{cases}
(1 + \E^{-12 x})\sin(50 x) & \text{if} \;\; 10^{-2}\leq r <1/2,\\
(1 + \E^{-6 x})\sin(50 y) & \text{if} \;\; 1/2 \leq r \leq 1.
\end{cases}
\label{eq:non-separable}
\end{align}
By leveraging the $x$-Jacobi matrix for Zernike annular polynomials (as discussed in \cref{sec:non-separable}), we are able to assemble the FEM linear system efficiently. Moreover, since the coefficient is a polynomial of degree one, we retain sparsity. However, the PDE operator is no longer rotationally invariant and  the resulting discretization matrix is not block diagonal unlike the previous examples. We mesh the domain into two annular cells at $\{10^{-2} \leq r \leq 1/2\}$ and $\{1/2 \leq r \leq 1\}$ which align with the radial discontinuity of the right-hand side. We then consider a discretization degree of $N_p=100$ on both cells, assemble the linear system and solve it via a sparse LU factorization. In \cref{fig:non-separable} we provide a spy plot of the FEM matrix as well as plots of the right-hand side and the solution. Note that when $x<0$, the  PDE operator is locally positive-definite but when $x>0$, we enter a Helmholtz regime. This is reflected in the solution where we notice a switch in the behaviour at $x=0$ from structured oscillations to highly unstructured oscillations akin to those in the solution of the high frequency problem in \cref{sec:examples:high-frequency}.

\begin{figure}[h!]
\centering
\includegraphics[width =0.32 \textwidth]{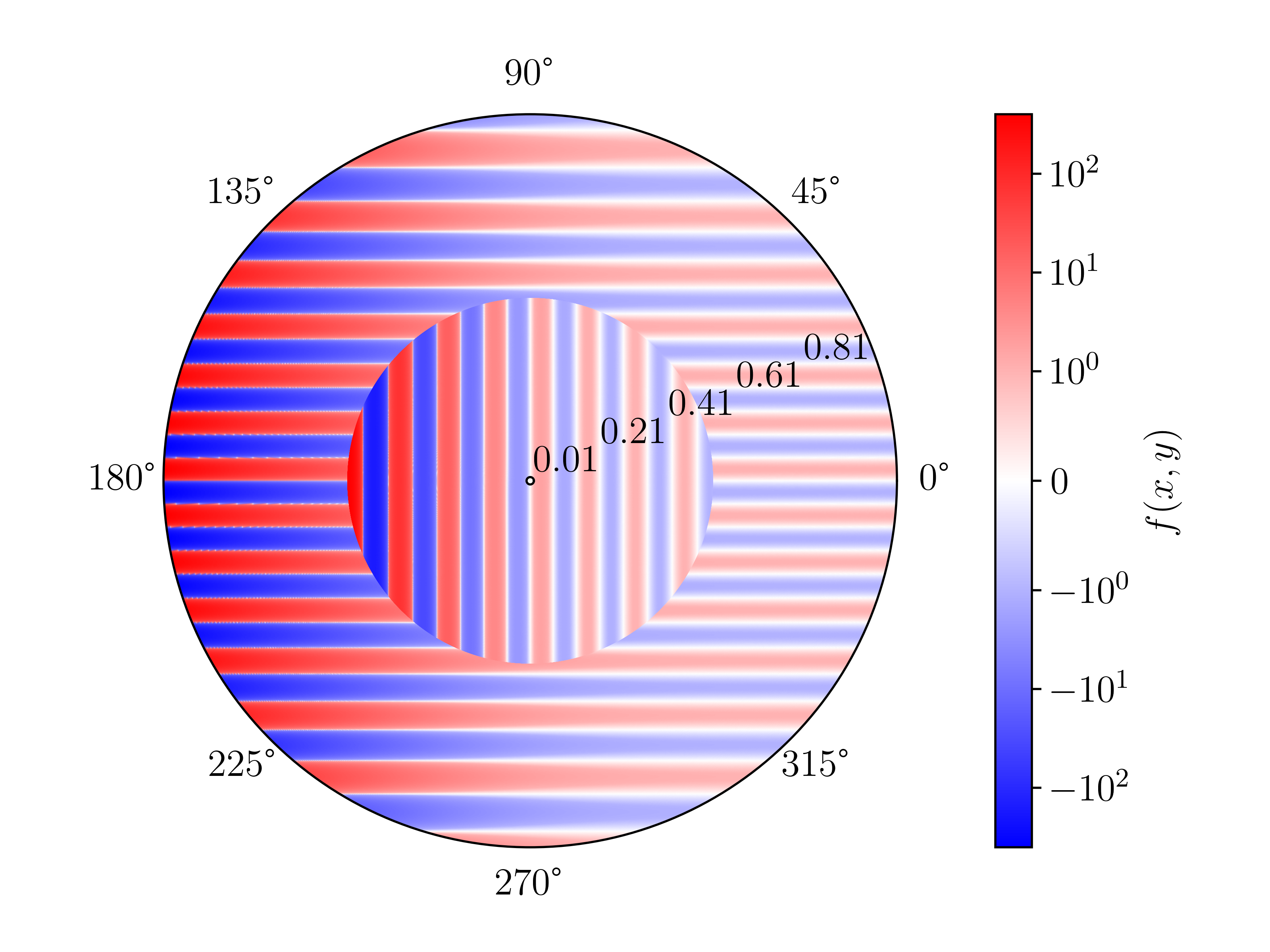}
\includegraphics[width =0.32 \textwidth]{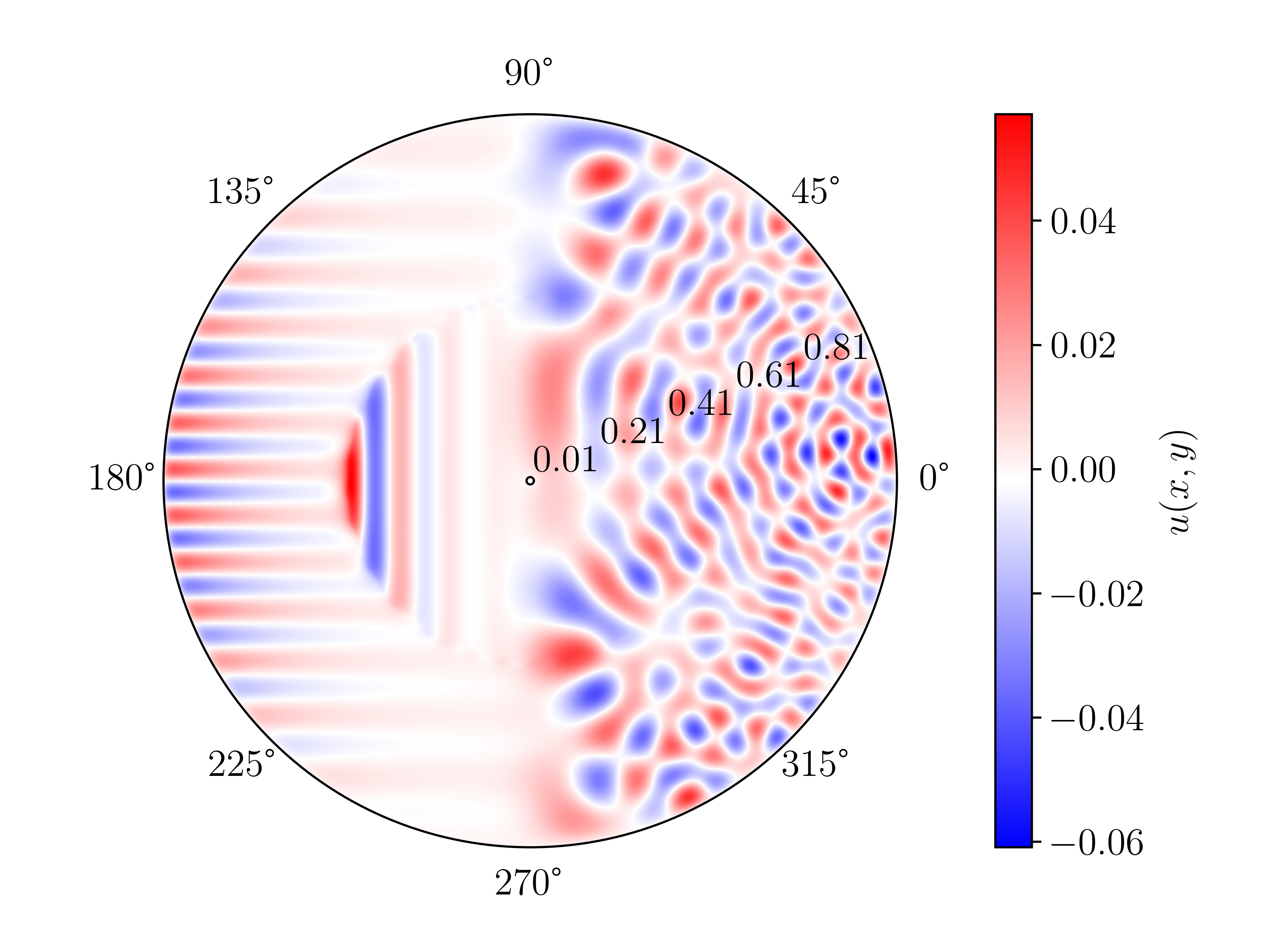}
\includegraphics[width =0.32 \textwidth]{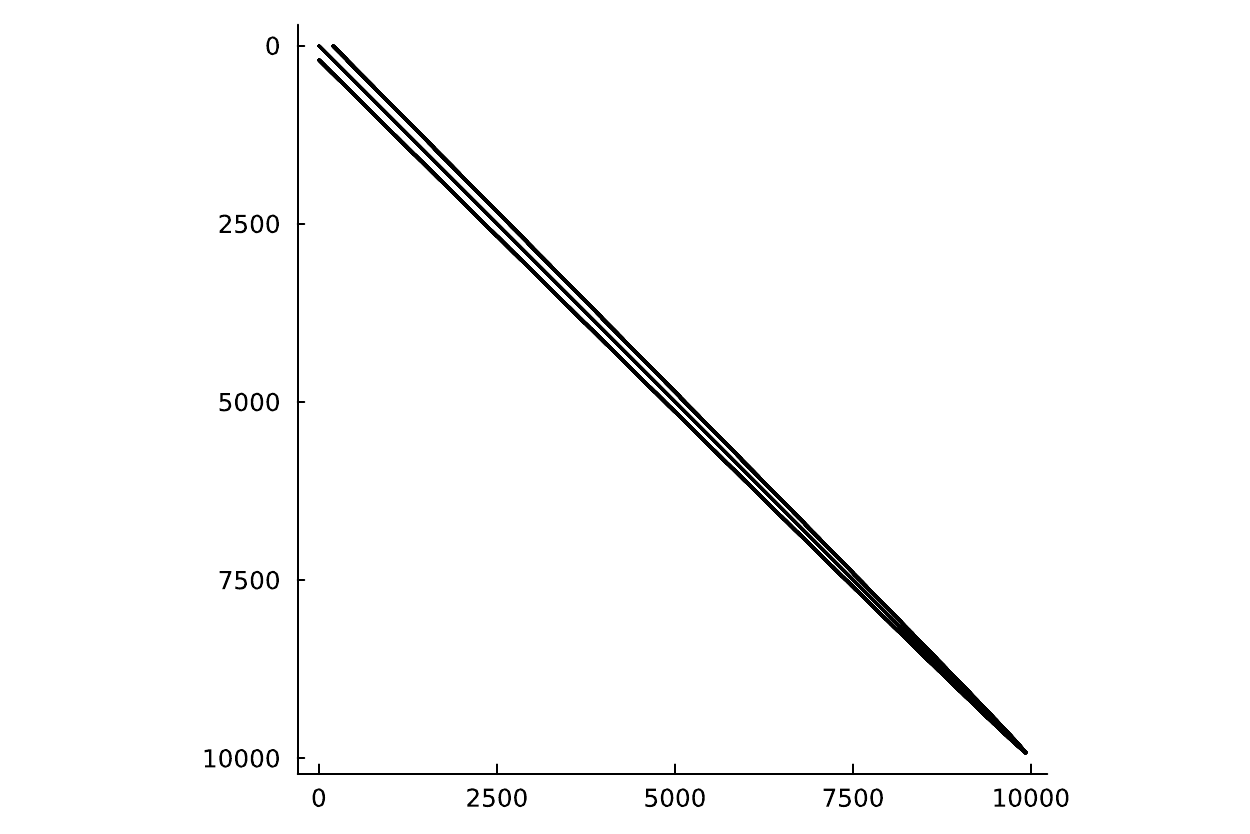}
\caption{ Plots of the right-hand side (left) and the solution (middle) in the non-separable problem of \cref{sec:examples:non-separable} with the right-hand side and rotationally anisotropic Helmholtz coefficient as given in \cref{eq:non-separable}. We mesh the domain into two cells meeting at $r=1/2$ which corresponds with the radial discontinuity of the right-hand side. As $x$ grows from $-1$ to $1$, we notice a change in the behaviour of the solution corresponding to the transition of a solution to the screened Poisson equation to one of a Helmholtz problem. On the right, we provide a spy plot of the FEM discretization matrix of \cref{eq:non-separable} where $N_p=100$ on both cells.}
\label{fig:non-separable}
\end{figure}

\subsection{Singular source term}
\label{sec:examples:hp-refinement}

In this example we consider the unit disk domain $\Omega_0 = \{ r < 1\}$ and seek $u \in H^1_0(\Omega_0)$ that satisfies
\begin{align}
-\Delta u = r^{-3/2}.
\label{eq:hp}
\end{align}
The exact solution to \cref{eq:hp} is $u(x,y) = 4-4(x^2+y^2)^{1/4} =4-4r^{1/2}$. This problem features a non-integrable singularity at $r=0$ in the source term as well as a square-root singularity in the solution. These singularity will impede the convergence of a $p$-refinement strategy. Nevertheless, we will recover spectral convergence to the solution by constructing a mesh that is graded towards the origin \cite[Ch.~3]{Schwab1998}.

In \cref{fig:hp} we consider three discretization strategies. One involves fixing a single-celled mesh (the unit disk) and solely increasing the truncation degree $N_p$ (pure $p$-refinement). Due to the singularity, $p$-refinement is ineffective. Consequentially, even when $N_p = 1004$ (corresponding to 502 degrees of freedom in the first Fourier mode coefficient) we still observe an $\ell^\infty$-norm error of 0.32. The other two strategies involve constructing a mesh of the unit disk domain where $\mathcal{T}_N = \{ 0 \leq r \leq 2^{-2N} \} \cup \{ 2^{-n} \leq r \leq 2^{-(n-1)}\}_{n \in \{1,2,3,\dots,2N\}}$. This results in a mesh with $N_h = 2N+1$ cells which are graded towards the origin. Then, for each mesh $\mathcal{T}_N$ we consider either an FEM discretization with degree $N_p = N$ on each cell (graded mesh, $p$-refinement) or fix $N_p=38$ across all the meshes (graded mesh, $N_p=38$). The latter strategy recovers spectral convergence to the exact solution with respect to number of degrees of freedom in the solution coefficient vector of the first Fourier mode $(m,j)=(0,1)$. Note that due to the rotational symmetry of the solution, the solution coefficients corresponding to other Fourier modes are equal zero. The former strategy achieves a smaller error per degree of freedom. However the convergence is not spectral. In \cref{fig:hp} we also provide a one-dimensional slice of the approximate solution when $N=N_p=38$ at $\theta=0$. The vertical dashed lines indicate the edges of the annular cells in the mesh.

\begin{figure}[h!]
\centering
\includegraphics[width =0.48 \textwidth]{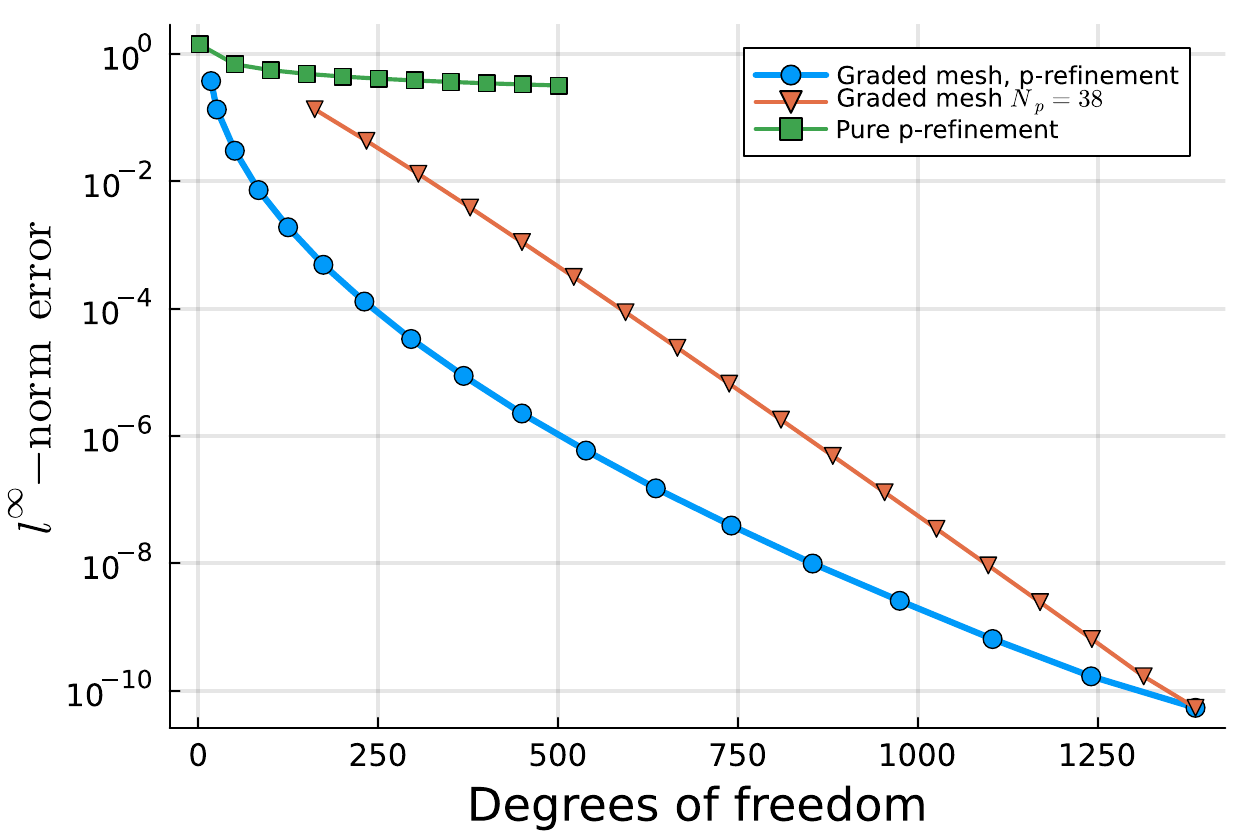}
\includegraphics[width =0.48 \textwidth]{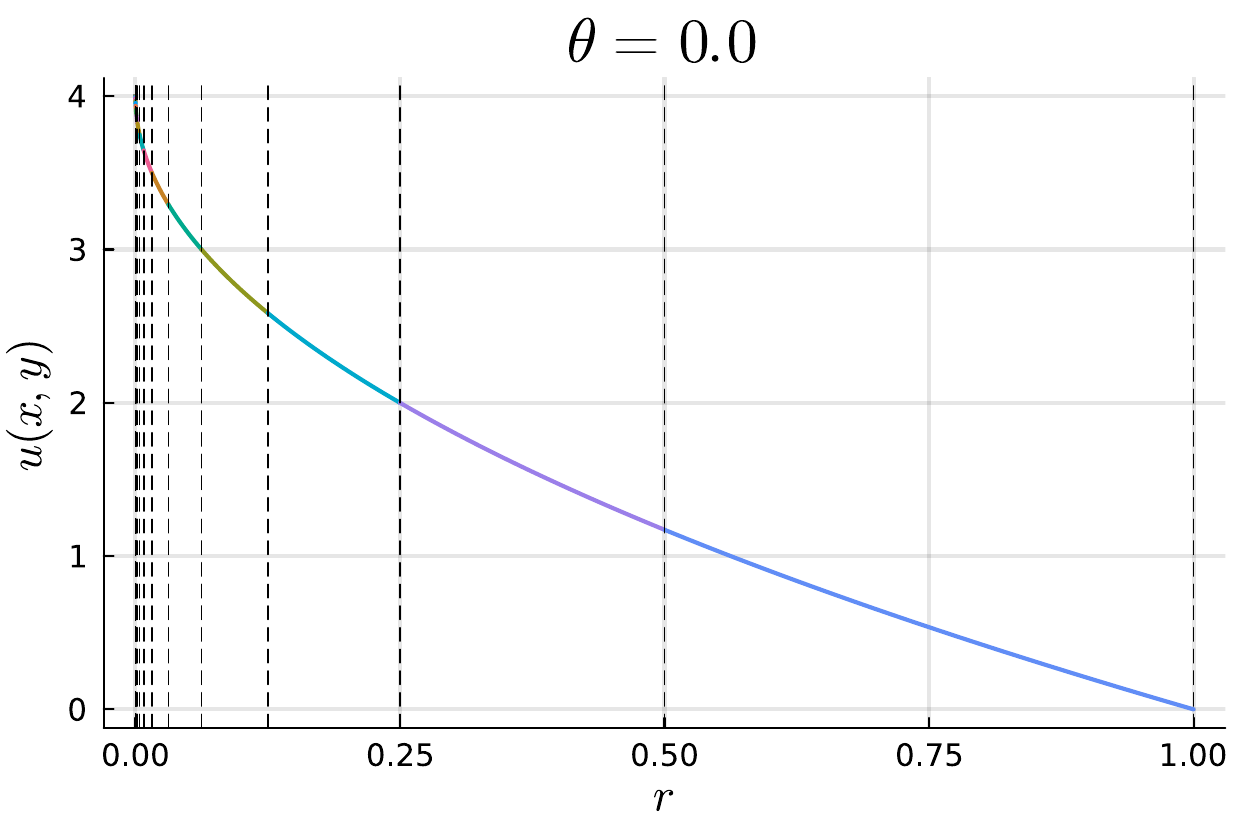}
\caption{ (Left) Convergence to the solution $u$ of \cref{eq:hp} via a pure $p$-refinement on a single-celled mesh and via a graded mesh with a fixed truncation degree $N_p=38$ (graded mesh, $N_p=38$) or where $N_p=N$ (graded mesh, $p$-refinement). The errors are measured against the number of degrees of freedom in the solution coefficient vector of the first Fourier mode $(m,j)=(0,1)$. The convergence of the second strategy is spectral. (Right) A one-dimensional slice of the approximate solution at $N=38$, i.e.~where $N_h = 77$ and $N_p=38$. The dashed vertical lines indicate the edges of the annular cells which are graded towards the origin.}
\label{fig:hp}
\end{figure}

\subsection{Screened Poisson in a 3D cylinder}
\label{sec:examples:ADI}
In this example we solve the screened Poisson on a 3D cylinder with a quasi-optimal complexity setup and solve of $\mathcal{O}(N_h N_p^3 \log^2 N_p)$ and $\mathcal{O}(N_h N_p^3 \log^2 N_h^{1/4} N_p)$, respectively. We utilize the hierarchical tensor-product FEM basis designed in \cref{sec:ADI}. Recall that in 3D we use $N_p$ to denote the truncation degree of each polynomial basis factor of the tensor-product space, such that the tensor-product basis contains polynomials of maximum degree $2N_p$. $N_h$ denotes the number of three-dimensional cells in the mesh.

Let ${\bm \Omega}= \Omega_0 \times (-1,1) \subset \mathbb{R}^3$ be a cylindrical domain. Consider the screened Poisson equation, find $u \in H^1_0({\bm \Omega})$ that satisfies \cref{eq:helmholtz} and choose the Helmholtz coefficient and the right-hand side:
\begin{align}
\lambda(r)
=
\begin{cases}
10^{-2} & \text{if} \;\; 0 \leq r \leq 1/2,\\
50 & \text{if} \;\; 1/2 < r \leq 1,
\end{cases}
\;\; \text{and} \;\;
f(x,y,z)
=
(-\Delta + \lambda(r)) u_e(x,y,z),
\label{eq:adi:2}
\end{align}
where $u_e(x,y,z) = \cos(5x) \tilde{u}(r) \cos(5z) (1-z^6)$. We pick $\tilde{u}(r)$ as defined in \cref{eq:tildeu} with $\lambda_0 = 10^{-2}$, $\lambda_1 = 50$, and $\rho=1/2$. We mesh the domain into four cells: $\{ r \leq 1/2\} \times [-1,0]$, $\{ r \leq 1/2\} \times [0,1]$, $\{ 1/2 \leq r \leq 1\} \times [-1,0]$, and $\{ 1/2 \leq r \leq 1\} \times [0,1]$. In \cref{fig:adi-plots}, we plot the right-hand side and the solution together with slices in $(x,y)$ and $z$ plane. Note the radial discontinuity in the right-hand side. The convergence plot is displayed in \cref{fig:adi-convergence}. We observe spectral convergence of our discretization as we increase $N_p$ in each cell. In \cref{fig:adi-convergence} we also plot the growth of final ADI iteration $\ell_{\text{max}}$, as $N_p \to \infty$, averaged over all the Fourier mode solves \cref{eq:adi:5} in order to solve the full problem \cref{eq:adi:1}. We observe the expected lograthmic growth. 
 
\begin{figure}[h!]
\centering
\includegraphics[width =0.3 \textwidth]{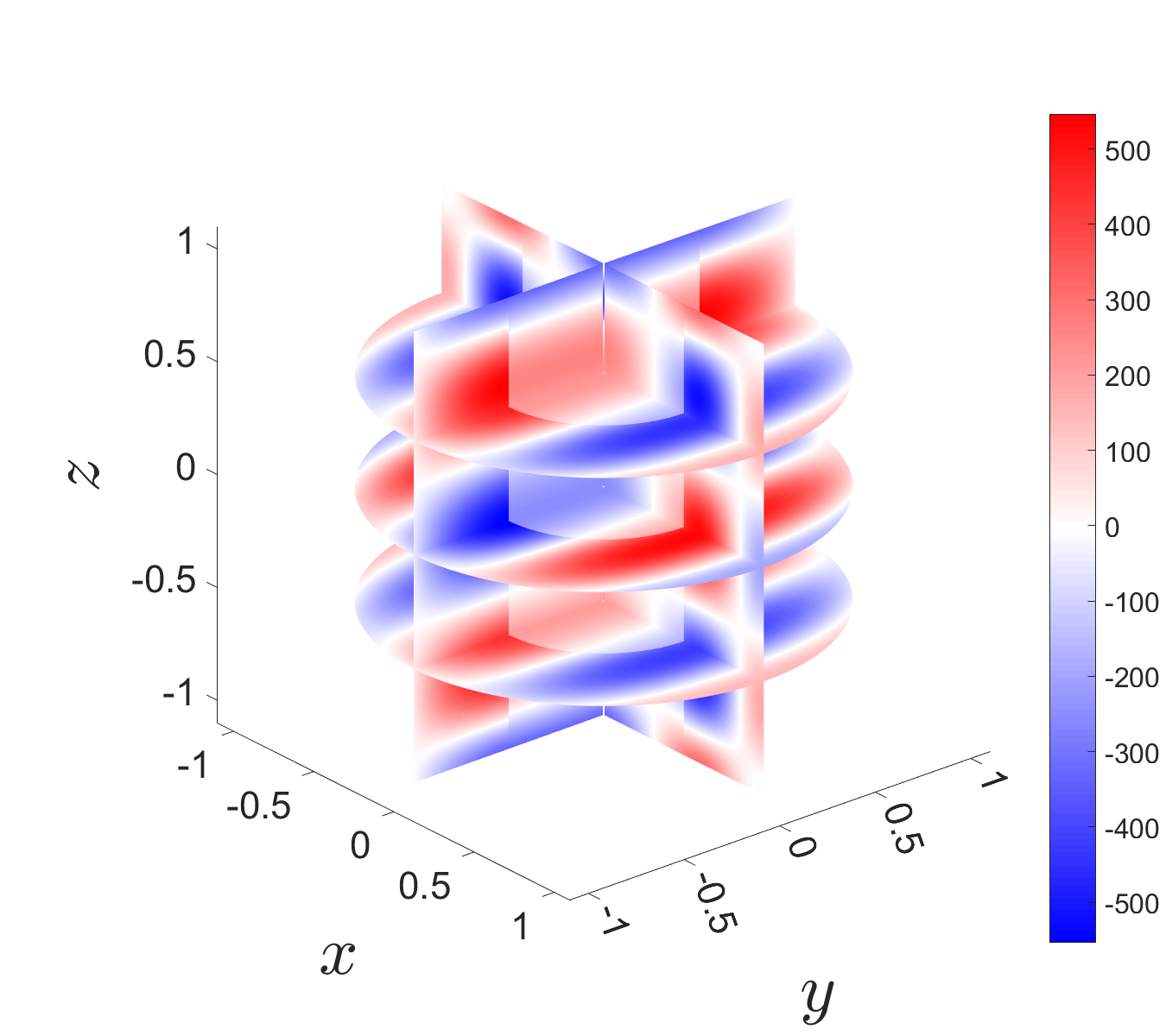}
\includegraphics[width =0.32 \textwidth]{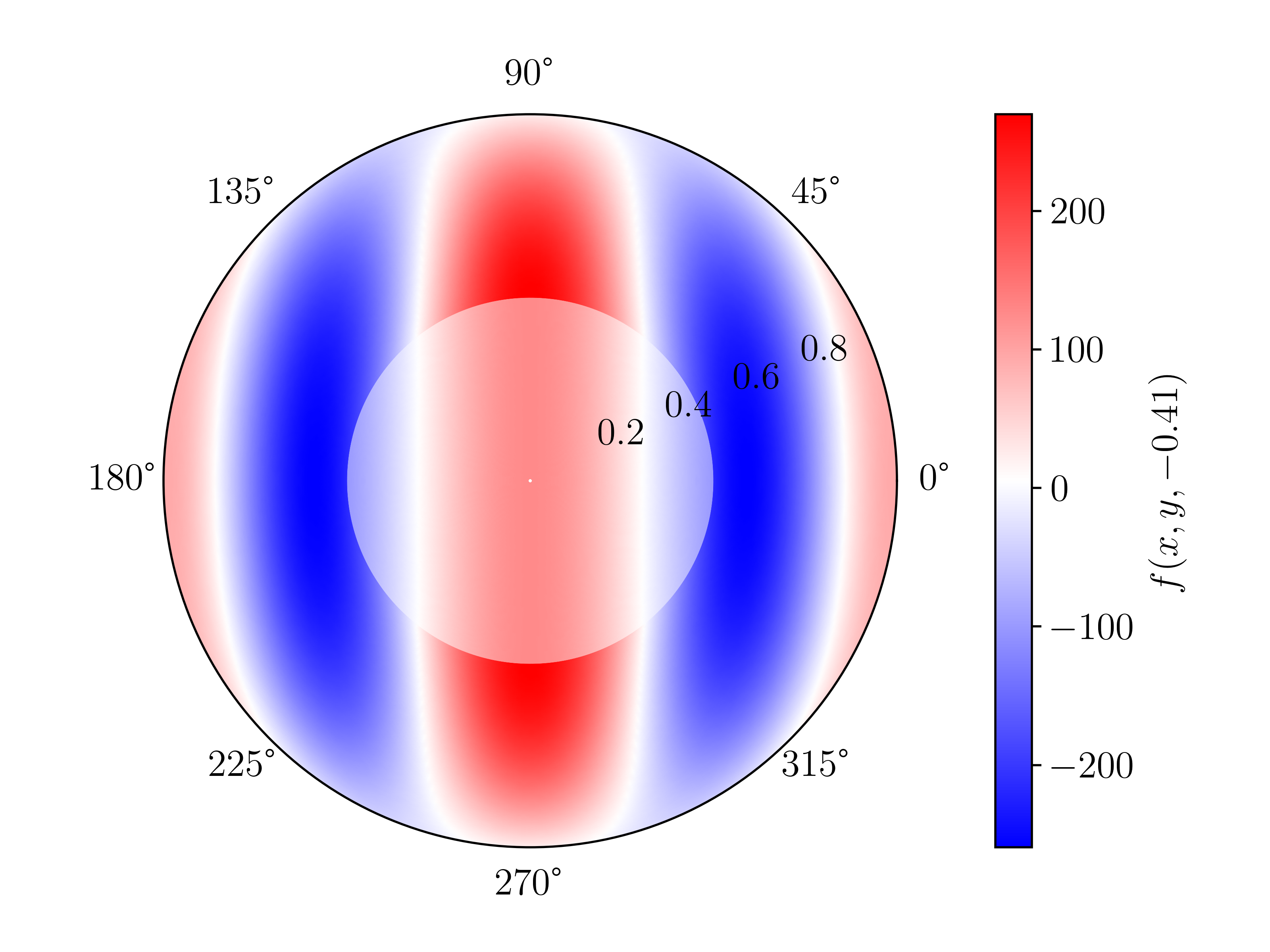}
\includegraphics[width =0.32 \textwidth]{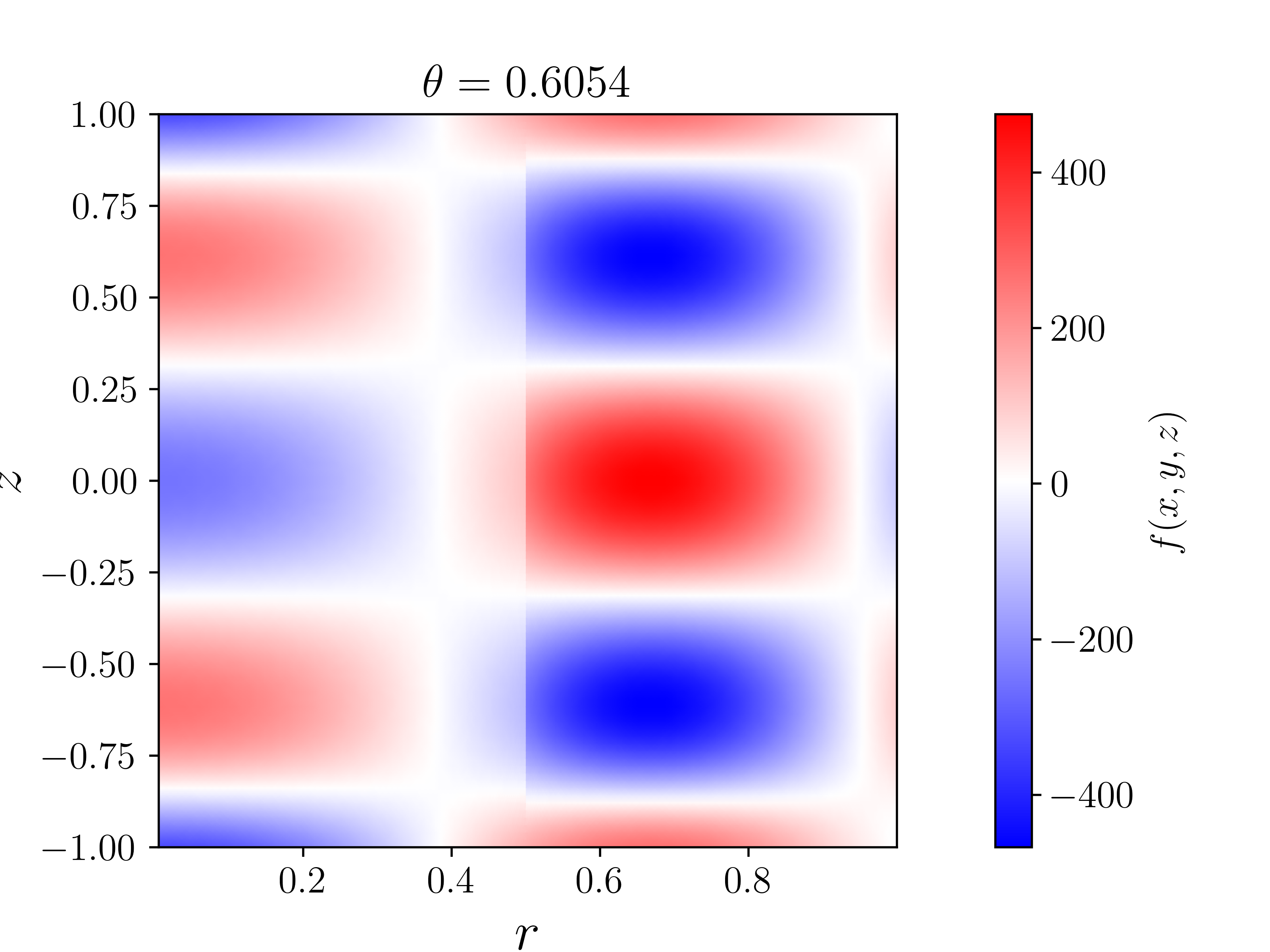}\\
\includegraphics[width =0.3 \textwidth]{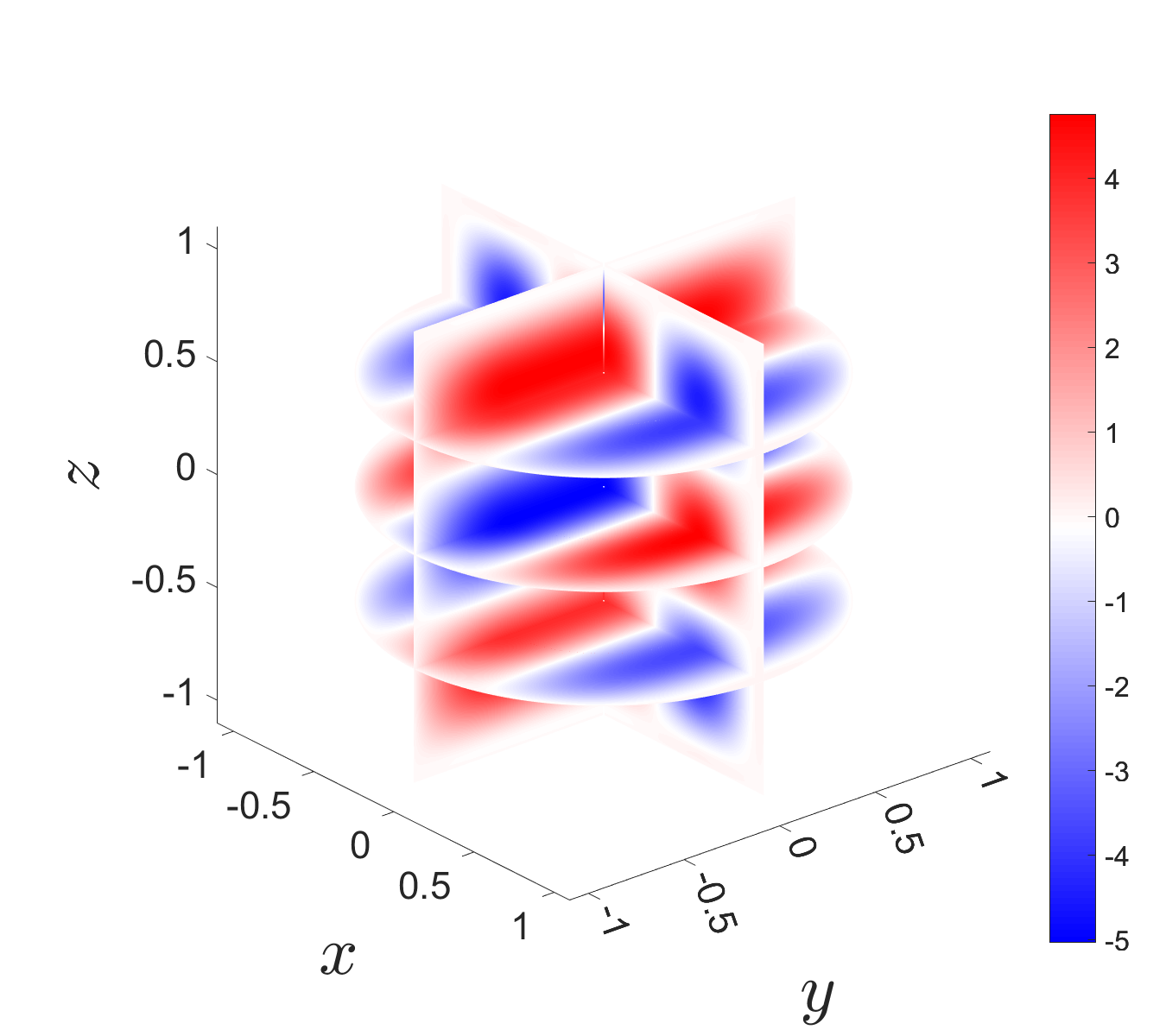}
\includegraphics[width =0.32 \textwidth]{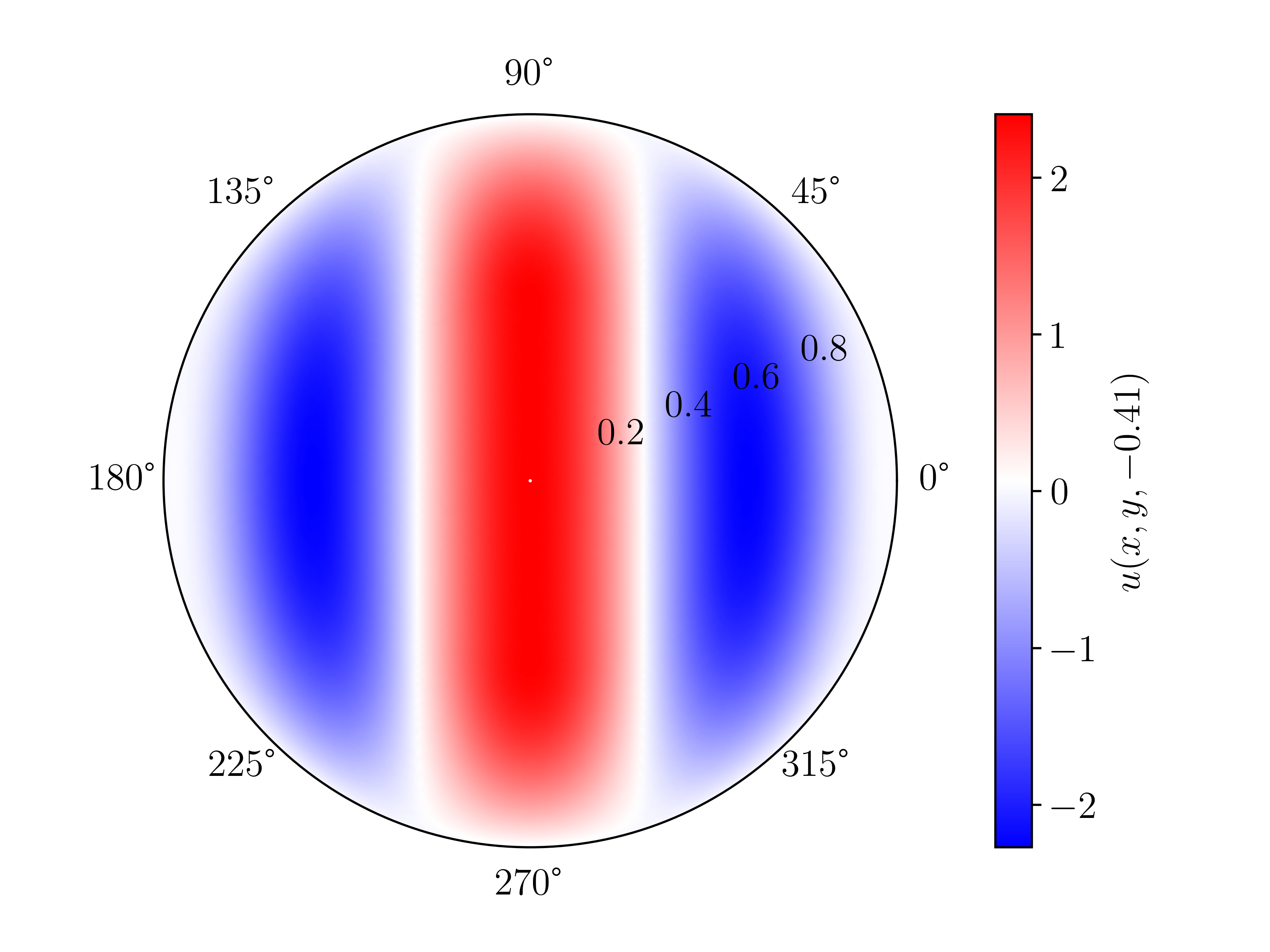}
\includegraphics[width =0.32 \textwidth]{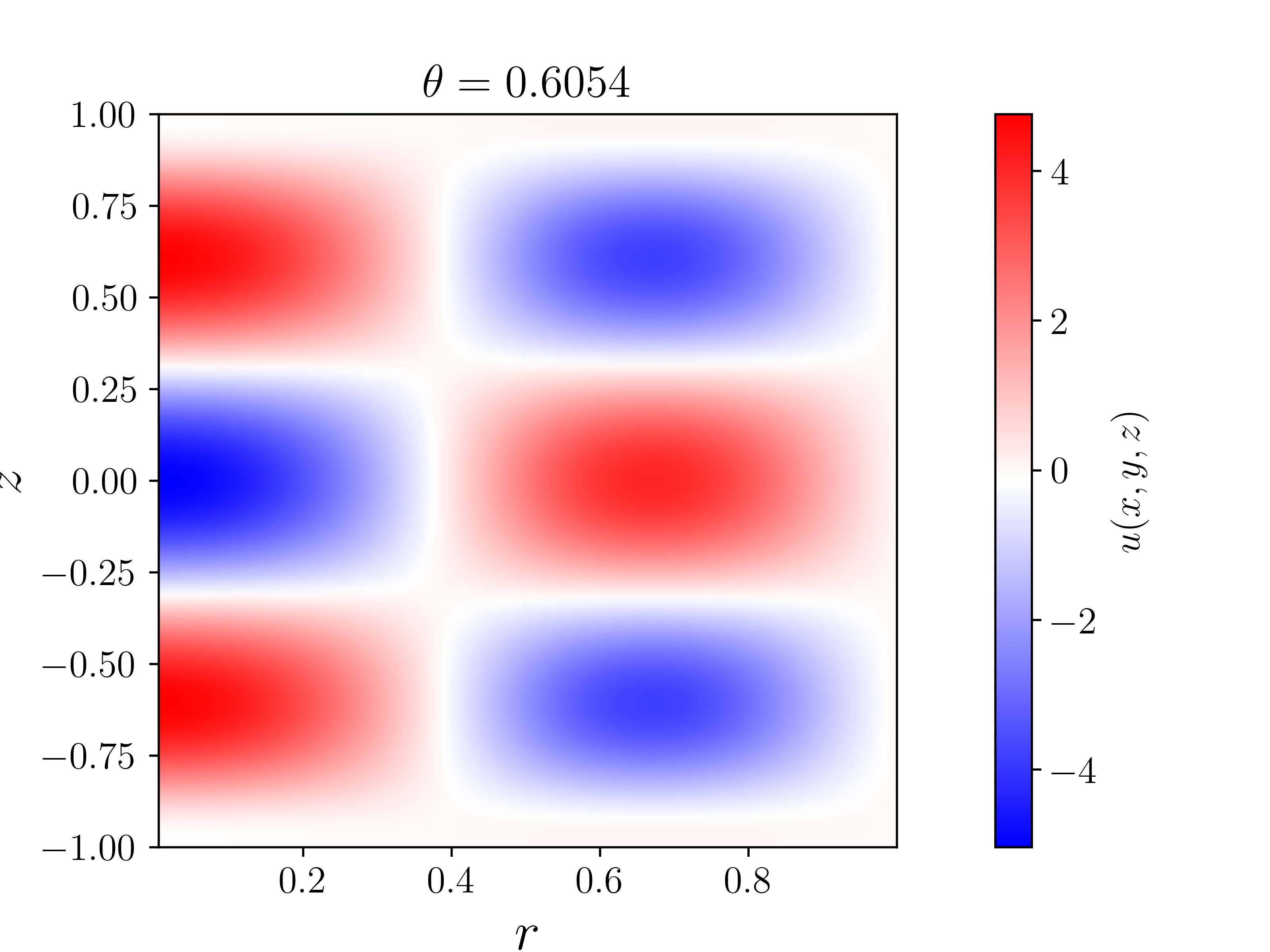}
\caption{Plots of the right-hand side $f(x,y,z)$ (top row) and the solution $u(x,y,z)$ (bottom row) in the 3D cylinder screened Poisson equation of \cref{sec:examples:ADI} with the right-hand side and Helmholtz coefficient as given in \cref{eq:adi:2}. The first column is a visualization on the 3D domain, the second column is a 2D slice in the $(x,y)$-plane and final column is a 2D slice through the $z$-plane.}
\label{fig:adi-plots}
\end{figure}

\begin{figure}[h!]
\centering
\subfloat[$\ell^\infty$-norm error]{\includegraphics[width =0.49 \textwidth]{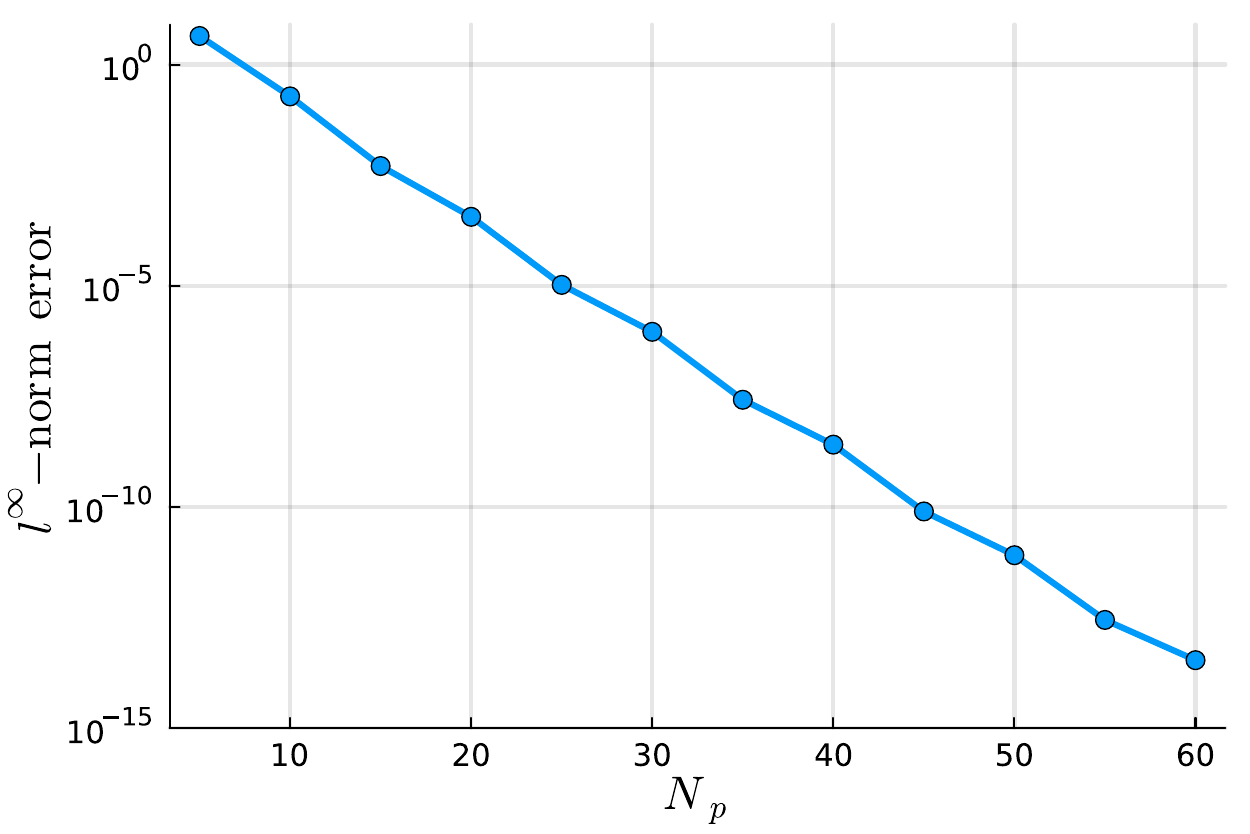}}
\subfloat[Average value of $l_{\text{max}}$]{\includegraphics[width =0.49 \textwidth]{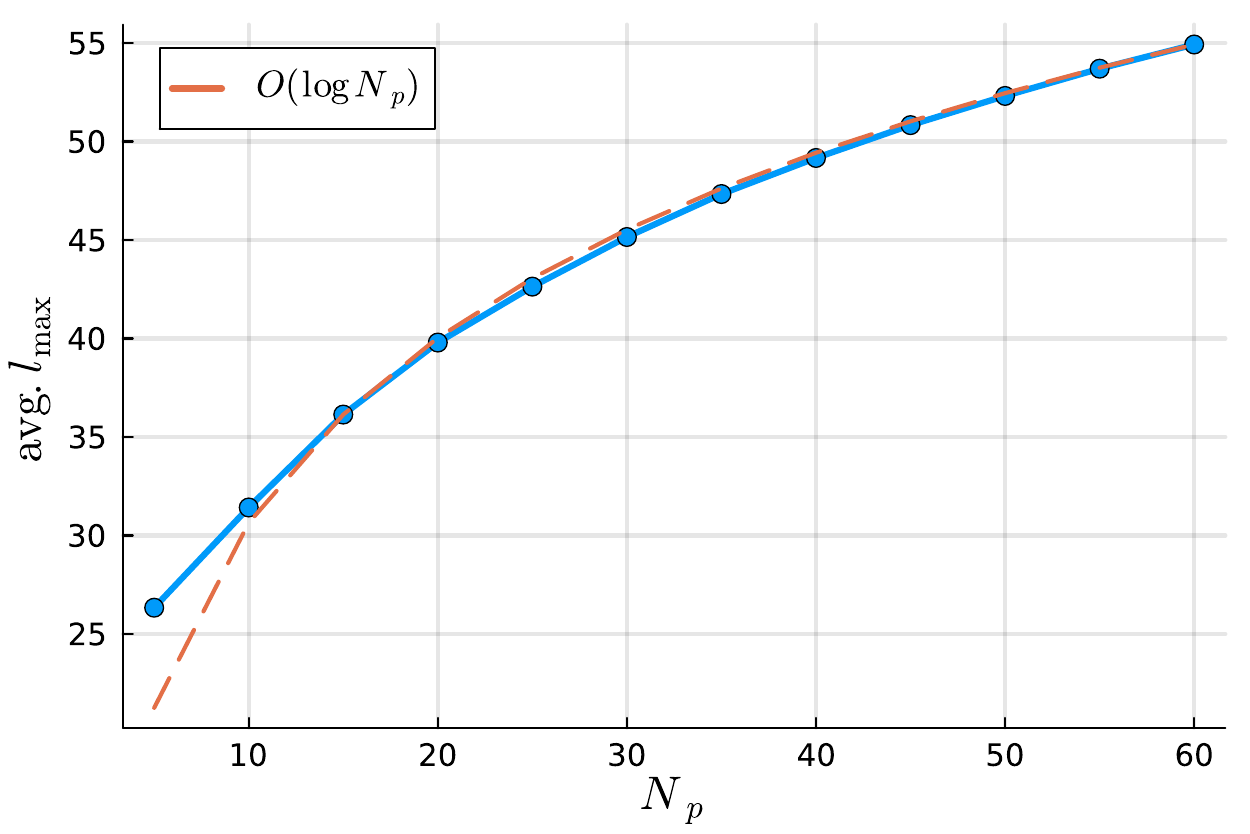}}
\caption{(Left) A semi-log convergence plot of the $\ell^\infty$-norm error of the $hp$-FEM basis for the 3D cylinder screened Poisson problem in \cref{sec:examples:ADI} with a right-hand side and Helmholtz coefficient as given in \cref{eq:adi:2}  with increasing polynomial degree $N_p$ on each of the four cells in the mesh. The plot indicates spectral convergence despite the radial discontinuities in the problem data. (Right) The growth of the final ADI iterate value $\ell_{\text{max}}$ as defined in \cref{th:ADI} averaged across all the Fourier modes solves \cref{eq:adi:5} for increasing $N_p$. We observe logarithmic growth.}
\label{fig:adi-convergence}
\end{figure}

To exemplify the flexibility of the hierarchical basis on the cylinder, in \cref{fig:adi-plots2} we plot the right-hand side and solution to a screened Poisson equation with a radial discontinuous Helmholtz coefficient $\lambda(r^2)$ and a right-hand side $f(x,y,z)$ with discontinuities at $r=1/2$ and $z=0$. Let
\begin{align}
\begin{split}
\lambda(r)
&=
\begin{cases}
1/2 & \text{if} \;\; r \leq 1/2,\\
r^2 & \text{if} \;\; r > 1/2,
\end{cases}
\;\; \text{and} \;\; f(x,y,z) = f_1(x,y)f_2(z) \;\; \text{where}\\
f_1(x,y)
&=
\begin{cases}
2 \cos(20y) & \text{if} \;\; r \leq 1/2,\\
\cos(10x) & \text{if} \;\; r > 1/2,
\end{cases} \;\;
f_2(z)
=
\begin{cases}
2 \cos(20z) & \text{if} \;\;  z \leq 0,\\
\sin(10z) & \text{if} \;\; z \geq 0.
\end{cases} 
\end{split}
\label{eq:adi:6}
\end{align}
We utilize the same discretization as in the previous 3D cylinder example truncation degree $N_p = 60$ and the quasi-optimal complexity solver for 3D cylinders introduced in \cref{sec:ADI}. The discontinuity of the right-hand side and the Helmholtz coefficient is severe. Nevertheless, the hierarchical tensor-product FEM basis accurately approximates the right-hand side and we obtain a solution that is qualitatively accurate.

\section{Conclusions}

In this work we designed a sparse $hp$-finite element method for the disk and annulus domains. The mesh consists of an innermost disk cell and concentric annuli cells. The hierarchical FEM basis contains bubble (internal shape) functions, which are weighted multivariate orthogonal  Zernike (annular)  polynomials whose support is fully contained on one cell, and hat (external shape) functions which are supported on a maximum of two cells. The bubble functions encode the high-order approximation properties of the FEM whereas the hat functions enforce continuity in the solution. The basis induces sparse block-diagonal stiffness and mass matrices where the blocks correspond to the Fourier mode decoupling of the corresponding PDE operators. The discretization retains  symmetry and  sparsity in the induced linear systems for a number of PDEs including the Helmholtz equation with a rotationally invariant and radially discontinuous Helmholtz coefficient. Moreover, the linear systems admit an optimal complexity $\mathcal{O}(N_h N_p^2)$ factorization where $N_h$ is the number of cells in the mesh and $N_p$ is the truncation order on each element.

We considered a number of examples including:
\begin{enumerate}
\itemsep=0pt
\item A high-frequency Helmholtz equation with a radially discontinuous Helmholtz coefficient and right-hand side. 
\item The time-dependent Schr\"odinger equation. The discretization is unitary preserving and, therefore, respects the conservation of energy of the system.   
\item  A rotationally anisotropic coefficient resulting in a non-separable Helmholtz problem. 
\item A singular source term resolved with a graded mesh towards the origin. 
\item The screened Poisson equation on a three-dimensional cylndrical domain with discontinuities in the right-hand side in the radial and $z$-directions. The basis is the tensor product of the FEM basis we developed for the disk and the continuous hierarchical basis for the interval cf.~\cite{Olver2023}, \cite[Ch.~2.5.2]{szabo2011introduction}, and \cite[Ch.~3.1]{Schwab1998}. Using recent results for Zernike (annular) polynomials \cite{Papadopoulos2023, Gutleb2023, Slevinsky2019} we obtain a setup complexity of $\mathcal{O}(N_h N_p^3\log^2 N_p)$. Then via the ADI solver \cite{Olver2023, Fortunato2020}, we derive a solve where we prove the complexity is $\mathcal{O}(N_h N_p^3 \log (N_h^{1/4} N_p))$.
\end{enumerate}
In all examples we observe spectral convergence, potentially after an initial plateau, for increasing truncation degree $N_p$.

\section*{Declarations}
This work was completed with the support of the EPSRC grant EP/T022132/1 ``Spectral element methods for fractional differential equations, with applications in applied analysis and medical imaging" and the Leverhulme Trust Research Project Grant RPG-2019-144 ``Constructive approximation theory on and inside algebraic curves and surfaces". IP was also supported by the Deutsche Forschungsgemeinschaft (DFG, German Research Foundation) under Germany's Excellence Strategy -- The Berlin Mathematics Research Center MATH+ (EXC-2046/1, project ID: 390685689).

\section*{Acknowledgements}
We are grateful to Timon S.~Gutleb for his contributions to the SemiclassicalOrthogonalPolynomials.jl package \cite{SemiPoly.jl2023} which allowed us to build, in optimal complexity, the raising operator matrices used in this work. We are also grateful to Richard M.~Slevinsky for his implementations of the quasi-optimal analysis and synthesis operators for Zernike (annular) polynomials in FastTransforms.jl \cite{FastTransforms.jl2023}. IP would like to express his gratitude to Pranav Singh for the discussions on unitary preserving time-stepping schemes and for providing references that detail the stationary solutions of the Schr\"odinger equation. IP also thanks Daniel Fortunato for the discussions on the ADI algorithm as well as Kars Knook for sharing his Julia implementation of the ADI algorithm.

\printbibliography
\end{document}